\documentclass[english]{article}
\usepackage[T1]{fontenc}
\usepackage[latin9]{inputenc}
\usepackage{babel}
\usepackage{float}
\usepackage{mathtools}
\usepackage{amsmath}
\usepackage{amsthm}
\usepackage{amssymb}
\usepackage[authoryear,round]{natbib}
\usepackage[unicode=true,pdfusetitle,
 bookmarks=true,bookmarksnumbered=false,bookmarksopen=false,
 breaklinks=false,pdfborder={0 0 1},backref=false,colorlinks=false]
 {hyperref}

\makeatletter

\floatstyle{ruled}
\newfloat{algorithm}{tbp}{loa}
\providecommand{\algorithmname}{Algorithm}
\floatname{algorithm}{\protect\algorithmname}

\theoremstyle{plain}
\newtheorem{thm}{\protect\theoremname}[section]
\theoremstyle{definition}
\newtheorem{defn}[thm]{\protect\definitionname}
\theoremstyle{plain}
\newtheorem{lem}[thm]{\protect\lemmaname}
\ifx\proof\undefined
\newenvironment{proof}[1][\protect\proofname]{\par
	\normalfont\topsep6\p@\@plus6\p@\relax
	\trivlist
	\itemindent\parindent
	\item[\hskip\labelsep\scshape #1]\ignorespaces
}{%
	\endtrivlist\@endpefalse
}
\providecommand{\proofname}{Proof}
\fi
\theoremstyle{plain}
\newtheorem{cor}[thm]{\protect\corollaryname}

\usepackage[paperwidth=8.5in, paperheight=11in, margin=1in]{geometry}
\author{ 
	Zijian Liu\thanks{Equal contribution. Stern
		School of Business, New York University, 	\texttt{zl3067@nyu.edu}.}  \and 
	Ta Duy Nguyen\thanks{Equal contribution. Department of Computer Science, Boston University,
		\texttt{{taduy@bu.edu}.}}  \and 
	Thien H. Nguyen\thanks{Equal contribution. Khoury College of Computer Sciences, Northeastern University,
		\texttt{{nguyen.thien@northeastern.edu}.}}  \and 
	Alina Ene\thanks{Department of Computer Science, Boston University, \texttt{{aene@bu.edu}.}}  \and 
	Huy L. Nguyen\thanks{Khoury College of Computer Sciences, Northeastern University,
		\texttt{{hu.nguyen@northeastern.edu}.}}}

\makeatother

\providecommand{\corollaryname}{Corollary}
\providecommand{\definitionname}{Definition}
\providecommand{\lemmaname}{Lemma}
\providecommand{\theoremname}{Theorem}

\begin{document}
\global\long\def\E{\mathbb{\mathbb{E}}}%
\global\long\def\F{\mathcal{F}}%
\global\long\def\R{\mathbb{R}}%
\global\long\def\tn{\widetilde{\nabla}f}%
\global\long\def\hn{\widehat{\nabla}f}%
\global\long\def\n{\nabla f}%
\global\long\def\indicator{\mathbf{1}}%
\global\long\def\mf{f(x^{*})}%
\global\long\def\breg{\mathbf{D}_{\psi}}%
\global\long\def\dom{\mathcal{X}}%
\global\long\def\norm#1{\left\lVert #1\right\rVert }%
\global\long\def\nf{\nabla f}%

\title{High Probability Convergence of Stochastic Gradient Methods }
\maketitle
\begin{abstract}
In this work, we describe a generic approach to show convergence
with high probability for both stochastic convex and non-convex optimization
with sub-Gaussian noise. In previous works for convex optimization,
either the convergence is only in expectation or the bound depends
on the diameter of the domain. Instead, we show high probability convergence
with bounds depending on the initial distance to the optimal solution.
The algorithms use step sizes analogous to the standard settings and
are universal to Lipschitz functions, smooth functions, and their
linear combinations. This method can be applied to the non-convex
case. We demonstrate an $O((1+\sigma^{2}\log(1/\delta))/T+\sigma/\sqrt{T})$
convergence rate when the number of iterations $T$ is known and an
$O((1+\sigma^{2}\log(T/\delta))/\sqrt{T})$ convergence rate when
$T$ is unknown for  SGD, where $1-\delta$ is the desired success
probability. These bounds improve over existing bounds in the literature.
Additionally, we demonstrate that our techniques can be used to obtain
high probability bound for AdaGrad-Norm \citep{ward2019adagrad} that
removes the bounded gradients assumption from previous works. Furthermore,
our technique for AdaGrad-Norm extends to the standard per-coordinate
AdaGrad algorithm \citep{duchi2011adaptive}, providing the first
noise-adapted high probability convergence for AdaGrad. 
\end{abstract}

\section{Introduction}

Stochastic optimization is a fundamental area with extensive applications
in many domains, ranging from machine learning to algorithm design
and beyond. The design and analysis of iterative methods for stochastic
optimization has been the focus of a long line of work, leading to
a rich understanding of the convergence of paradigmatic iterative
methods such as stochastic gradient descent, mirror descent, and accelerated
methods for both convex and non-convex optimization. However, most
of these works only establish convergence guarantees that hold only
in expectation. Although very meaningful, these results do not fully
capture the convergence behaviors of the algorithms when we perform
only a small number of runs of the algorithm, as it is typical in
modern machine learning applications where there are significant computational
and statistical costs associated with performing multiple runs of
the algorithm \citep{harvey2019tight,madden2020high,davis2021low}.
Thus, an important direction is to establish convergence guarantees
for a single run of the algorithm that hold not only in expectation
but also with high probability. 

Compared to the guarantees that hold in expectation, high probability
guarantees are significantly harder to obtain and they hold in more
limited settings with stronger assumptions on the problem settings
and the stochastic noise distribution. Most existing works that establish
high probability guarantees focus on the setting where the length
of the stochastic noise follows a light-tail (sub-Gaussian) distribution
\citep{JuditskyNT11,lan2012optimal,lan2020first,li2020high,madden2020high,kavis2021high}.
Recent works also study the more challenging heavy-tail setting, notably
under a bounded variance \citep{nazin2019algorithms,gorbunov2020stochastic,cutkosky2021high}
or bounded $p$-moment assumption \citep{cutkosky2021high} on the
length of the stochastic noise. Both settings are highly relevant
in practice: \citet{zhang2020adaptive} empirically studied the noise
distribution for two common tasks, training a ResNet model for computer
vision and a BERT transformer model for natural language processing,
and they observed that the noise distribution in the former task is
well-approximated by a sub-Gaussian distribution, and it appears to
be heavy-tailed in the latter task.

Despite this important progress, the convergence of cornerstone methods
is not fully understood even in the more structured light-tailed noise
setting. Specifically, the existing works for both convex and non-convex
optimization rely on strong assumptions on the optimization domain
and the gradients that significantly limit their applicability:

\emph{The problem domain is restricted to either the unconstrained
domain or a constrained domain with bounded Bregman diameter.} The
convergence guarantees established depend on the Bregman diameter
of the domain instead of the initial distance to the optimum. Even
for compact domains, since the diameter can be much larger than the
initial distance, these guarantees are pessimistic and diminish the
benefits of good initializations. Thus an important direction remains
to establish high probability guarantees for general optimization
that scale only with the initial Bregman distance.

\emph{The gradients or stochastic gradients are assumed to be bounded
even in the smooth setting.} These additional assumptions are very
restrictive and they significantly limit the applicability of the
algorithm, e.g., they do not apply to important settings such as quadratic
optimization. Moreover, the stochastic gradient assumption is more
restrictive than other commonly studied assumptions, such as the gradients
and the stochastic noise being bounded almost surely.

The above assumptions are not merely an artifact of the analysis,
and they stem from important considerations and technical challenges.
The high probability convergence guarantees are established via martingale
concentration inequalities that impose necessary conditions on how
much the martingale sequence can change in each step. However, the
natural martingale sequences that arise in optimization depend on
quantities such as the distance between the iterates and the optimum
and the stochastic gradients, which are not a priori bounded. The
aforementioned assumptions ensure that the concentration inequalities
can be readily applied due to the relevant stochastic terms being
all bounded almost surely. These difficulties are even more pronounced
for adaptive algorithms in the AdaGrad family that set the step sizes
based on the stochastic gradients. The adaptive step sizes introduce
correlations between the step sizes and the update directions, and
a crucial component is the analysis of the evolution of the adaptive
step sizes and the cumulative stochastic noise. If the gradients are
bounded, both of these challenges can be overcome by paying error
terms proportional to the lengths of the gradients and stochastic
gradients. Removing the bounded gradient assumptions requires new
technical insights and tools.

In addition to requiring stronger assumptions, due to the technical
challenges involved, several of the prior works are only able to establish
convergence guarantees that are slower than the ideal sub-Gaussian
rates. For example, a common approach is to control the relevant stochastic
quantities across all $T$ iterations of the algorithm via repeated
applications of the concentration inequalities, leading to convergence
rates that have additional factors that are poly-logarithmic in $T$.
Additionally, achieving noise-adaptive rates that improve towards
the deterministic rate as the amount of noise decreases is very challenging
with existing techniques.

\textbf{Our contributions:} This work aims to contribute to this line
of work and overcome the aforementioned challenges. To this end, we
introduce a novel generic approach to show convergence with high probability
under sub-Gaussian gradient noise. Our approach is very general and
flexible, and it can be used both in the convex and non-convex setting.
Using our approach, we establish high-probability convergence guarantees
for several fundamental settings:

In the \emph{convex setting}, we analyze stochastic mirror descent
and stochastic accelerated mirror descent for general optimization
domains and Bregman distances, and we analyze the classical algorithms
without any changes. These well studied algorithms encompass the main
algorithmic frameworks for convex optimization with non-adaptive step
sizes \citep{lan2020first}. Our convergence guarantees scale with
only the Bregman distance between the initial point and the optimum,
and thus they leverage good initializations. Our high-probability
convergence rates are analogous to known results for convergence in
expectation \citep{JuditskyNT11,lan2012optimal}. The algorithms are
universal for both Lipschitz functions and smooth functions.

In the \emph{non-convex setting}, we analyze the SGD as well as the
AdaGrad-Norm algorithm \citep{ward2019adagrad}. Compared to existing
works for SGD \citep{madden2020high,li2020high}, our rates have better
dependency on the time horizon and the success probability. For AdaGrad-Norm,
our approach allows us to remove the restrictive assumption on the
gradients as made in previous work \citep{kavis2021high}. More importantly,
the technique employed to show high probability convergence of AdaGrad-Norm
readily extends to the standard coordinate version of AdaGrad; we
obtain the first results for the high probability convergence guarantee
for AdaGrad \citep{duchi2011adaptive}. 

Although we only focus on sub-Gaussian gradient noise -- a more structured
setting where there still remain significant gaps in our understanding
-- we believe our approach could potentially be applied in more general
settings such as heavy tails noise.

\subsection{Our techniques}

Compared to prior works that rely on black-box applications of martingale
concentration inequalities such as Freedman's inequality and its extensions
\citep{freedman1975tail,harvey2019tight,madden2020high}, we introduce
here a ``white-box'' concentration argument that leverages existing
convergence analyses for first-order methods. The high-level approach
is to define a novel martingale sequence derived from the standard
convergence analyses and analyze its moment generating function from
first principles. By leveraging the structure of the optimization
problem, we are able to overcome a key difficulty associated with
black-box applications of martingale concentration results: these
results pose necessary conditions on how much the martingale sequence
can change, which do not a priori hold for the natural martingales
that arise in optimization. By seamlessly combining the optimization
and probability tool-kits, we obtain a flexible analysis template
that allows us to handle general optimization domains with very large
or even unbounded diameter, general objectives that are not globally
Lipschitz, and adaptive step sizes.

Our technique is inspired by classical works in concentration inequalities,
specifically a type of martingale inequalities where the variance
of the martingale difference is bounded by a linear function of the
previous value. This technique is first applied by \citet{harvey2019tight}
to show high probability convergence for SGD in the strongly convex
setting. Our proof is inspired by the proof of Theorem 7.3 by \citet{chung2006concentration}.
In each time step with iterate $x_{t}$, let $\xi_{t}:=\widehat{\nabla}f\left(x_{t}\right)-\nabla f\left(x_{t}\right)$
be the stochastic error in our gradient estimate. Classical proofs
of convergence evolve around analyzing the sum of $\left\langle \xi_{t},x^{*}-x_{t}\right\rangle $,
which can be viewed as a martingale sequence. Assuming a bounded domain,
the concentration of the sum can be shown via classical martingale
inequalities. The key new insight is that instead of analyzing this
sum, we analyze a related sum where the coefficients decrease over
time to account for the fact that we have a looser grip on the distance
to the optimal solution as time increases. Nonetheless, the coefficients
are kept within a constant factor of each others and the same asymptotic
convergence is attained with high probability.

\subsection{Related work}

\paragraph{Convex optimization:}

\citet{nemirovski2009robust,lan2012optimal} establish high probability
bounds for stochastic mirror descent and accelerated stochastic mirror
descent with sub-Gaussian noise. These rates match the best rates
known in expectation, but they depend on the Bregman diameter $\max_{x,y\in\dom}\breg\left(x,y\right)$
of the domain, which can be very large or even unbounded. Similarly,
\citet{kakade2008generalization,rakhlin2011making,hazan2014beyond,harvey2019tight,dvurechensky2016stochastic}
study the high-probability convergence of SGD that also assume that
the domain has bounded diameter or the function is strongly convex.
In contrast, our work complements its predecessors with a novel concentration
argument that establishes convergence for the general setting of convex
functions under sub-Gaussian gradient noise (as considered in \citet{lan2020first})
that depends only on the distance $\breg\left(x^{*},x_{1}\right)$
from the initial point to the optimum instead of the diameter of the
problem or having to assume that the objective is strongly convex. 

On a different note, \citet{nazin2019algorithms,gorbunov2020stochastic}
consider the more general setting of bounded variance noise. However,
their problem settings are more restricted than ours. Specifically,
\citet{nazin2019algorithms} analyze stochastic mirror descent only
in the setting where the optimization domain has bounded Bregman diameter.
\citet{gorbunov2020stochastic} analyze modified versions of stochastic
gradient descent and accelerated stochastic gradient descent (such
as with clipping), but only for unconstrained optimization with the
$\ell_{2}$ setup. In contrast, our work applies to a more general
optimization setup: we analyze the classical stochastic mirror descent
and accelerated mirror descent without any modifications under general
Bregman distances and arbitrary optimization domains that are possibly
unbounded. Finally,\citet{davis2021low} provides an algorithm to
achieve high probability convergence by solving an auxiliary optimization
problem in each iteration. However, their analysis is restricted to
well-conditioned objectives that are both smooth and strongly convex
and the expensive optimization subroutine can be impractical. 

\paragraph{Non-convex optimization:}

\citet{li2020high} demonstrate a high probability bound for an SGD
algorithm with momentum, while \citet{madden2020high} and \citet{li2022high}
show high probability bounds for vanilla SGD that generalize to the
family of sub-Weibull noise. However, these existing bounds are not
optimal due to the multiplicative dependency $O\left(\log T\log\frac{1}{\delta}\right)$.
In our work, we improve the high-probability convergence for SGD in
the non-convex setting via our novel approach. 

For algorithms with adaptive step size like AdaGrad, \citet{li2020high,kavis2021high}
provide some of the first high probability guarantees in the non-convex
setting. However, there still remains significant gaps in our understanding:
\citet{li2020high} is not fully adaptive due to the dependence of
the initial step size on the problem parameters, whereas \citet{kavis2021high}
requires that the gradients and/or stochastic gradients to be uniformly
bounded almost surely, a strong assumption that excludes even the
quadratic function. In contrast, we establish convergence in high
probability of AdaGrad-Norm \citep{ward2019adagrad,faw2022power}
without further restrictive assumptions. Notably, a key distinction
from prior work is that our does not involve the division by the step
size: this allows a direct extension of our analysis for AdaGrad-Norm
(in which the step size is a scalar) to the general AdaGrad \citep{duchi2011adaptive}
algorithm (where the step size varies for each coordinate). There,
to the best of our knowledge, \citet{defossez2020simple} is the only
work to provide an \emph{in expectation} guarantee for vanilla Adagrad
albeit under strong assumptions. We provide a more detailed comparison
with prior work in the subsequent sections.

Convergence guarantees for the heavy tail noise regime has also been
studied for non-convex objectives. However, some form of gradient
clipping is required in most works to deal with the large variance.
The work \citet{zhang2020adaptive} proposes a gradient clipping algorithm
that can converge \emph{in expectation} for noise distributions with
heavier tail---that is the $p$-moment is bounded for $1<p\le2$.
\citet{cutkosky2021high} propose a more complex clipped SGD algorithm
with momentum under the same noise assumption, for which they show
a high probability convergence. However, \citet{cutkosky2021high}
rely on the bounded moments of the stochastic gradients for the non-convex
setting; a restrictive assumption that excludes quadratic objectives.
In contrast, we focus on standard algorithms (albeit under sub-Gaussian
noise) that have been more widely used: stochastic mirror descent,
stochastic gradient descent, and AdaGrad-Norm. Our technique is general
and we believe that it is possible to extend it to the heavy-tail
noise setting.

\section{Preliminaries}

We consider the problem $\min_{x\in\dom}f(x)$ where $f:\R^{d}\to\R$
is the objective function and $\dom$ is the domain of the problem.
In the convex case, we consider the general setting where $f$ is
potentially not strongly convex and the domain $\dom$ is convex but
not necessarily compact. The distance between solutions in $\dom$
is measured by a general norm $\left\Vert \cdot\right\Vert $. Let
$\left\Vert \cdot\right\Vert _{*}$ denote the dual norm of $\left\Vert \cdot\right\Vert $.
In the non-convex case, we consider the setting where $\dom$ is $\R^{d}$
and $\left\Vert \cdot\right\Vert $ is the $\ell_{2}$ norm.

In this paper, we use the following assumptions:

\textbf{(1) Existence of a minimizer}: In the convex setting, we assume
that there exists $x^{*}=\arg\min_{x\in\dom}f(x)$.

\textbf{(1') Existence of a minimizer}: In the nonconvex setting,
we assume that $f$ admits a finite lower bound $\inf_{x\in\dom}f(x)\coloneqq f_{*}>-\infty$.

\textbf{(2) Unbiased estimator}: We assume to have access to a history
independent, non-biased gradient estimator $\hn(x)$ for any $x\in\dom$,
that is $\E\left[\hn(x)\mid x\right]=\n(x)$.

\textbf{(3) Sub-Gaussian noise}: $\left\Vert \hn(x)-\n(x)\right\Vert _{*}$
is a $\sigma$-sub-Gaussian random variable (Definition \ref{def:subgaussian-random-variable}).

There are several equivalent definitions of sub-Gaussian random variables
up to an absolute constant scaling (see, e.g., Proposition 2.5.2 in
\citet{vershynin2018high}). For convenience, we use the following
property as the definition. 
\begin{defn}
\label{def:subgaussian-random-variable}A random variable $X$ is
$\sigma$-sub-Gaussian if
\[
\E\left[\exp\left(\lambda^{2}X^{2}\right)\right]\leq\exp\left(\lambda^{2}\sigma^{2}\right)\text{ for all }\lambda\text{ such that }\left|\lambda\right|\leq\frac{1}{\sigma}.
\]
\end{defn}
We will also use the following helper lemma whose proof we defer
to the Appendix.
\begin{lem}
\label{lem:helper-taylor-vector}Suppose $X\in\R^{d}$ such that $\E\left[X\right]=0$
and $\left\Vert X\right\Vert $ is a $\sigma$-sub-Gaussian random
variable, then for any $a\in\R^{d}$, $0\le b\le\frac{1}{2\sigma}$,
\begin{align*}
\E\left[\exp\left(\left\langle a,X\right\rangle +b^{2}\left\Vert X\right\Vert ^{2}\right)\right] & \le\exp\left(3\left(\left\Vert a\right\Vert _{*}^{2}+b^{2}\right)\sigma^{2}\right).
\end{align*}
Especially, when $b=0$, we have 
\[
\E\left[\exp\left(\left\langle a,X\right\rangle \right)\right]\le\exp\left(2\left\Vert a\right\Vert _{*}^{2}\sigma^{2}\right).
\]
\end{lem}

\section{Convex case: Stochastic Mirror Descent and Accelerated Stochastic
Mirror Descent \label{sec:convex}}

In this section, we analyze the Stochastic Mirror Descent algorithm
(Algorithm \ref{alg:md}) and Accelerated Stochastic Mirror Descent
algorithm (Algorithm \ref{alg:acc-md}) for convex optimization. We
define the Bregman divergence $\breg\left(x,y\right)=\psi\left(x\right)-\psi\left(y\right)-\left\langle \nabla\psi\left(y\right),x-y\right\rangle $
where $\psi:\R^{d}\to\R$ is an $1$-strongly convex mirror map with
respect to $\left\Vert \cdot\right\Vert $ on $\dom$. We remark that
the domain of $\psi$ is defined as $\R^{d}$ for simplicity, though
which is not necessary.

\subsection{Analysis of Stochastic Mirror Descent}

\begin{algorithm}
\caption{Stochastic Mirror Descent Algorithm}

\label{alg:md}

\textbf{Parameters:} initial point $x_{1}\in\dom$, step sizes $\left\{ \eta_{t}\right\} $,
strongly convex mirror map $\psi$

for $t=1$ to $T$:

$\quad$$x_{t+1}=\arg\min_{x\in\dom}\left\{ \eta_{t}\left\langle \widehat{\nabla}f\left(x_{t}\right),x\right\rangle +\breg\left(x,x_{t}\right)\right\} $

return $\frac{1}{T}\sum_{t=1}^{T}x_{t}$
\end{algorithm}

The end result of this section is the convergence guarantee of Algorithm
\ref{alg:md} for constant step sizes (when the time horizon $T$
is known) and time-varying step sizes (when $T$ is unknown) presented
in Theorem \ref{thm:md-convergence}. However, we will emphasize more
on presenting the core idea of our approach, which will serve as the
basis for the analysis in subsequent sections. For simplicity, here
we consider the non-smooth setting, and assume that $f$ is $G$-Lipschitz
continuous, i.e., we have $\left\Vert \nabla f(x)\right\Vert _{*}\leq G$
for all $x\in\dom$. However, this is not necessary. The analysis
for the smooth setting follows via a simple modification to the analysis
presented here as well as the analysis for the accelerated setting
given in the next section. 
\begin{thm}
\label{thm:md-convergence}Assume $f$ is $G$-Lipschitz continuous
and satisfies Assumptions (1), (2), (3), with probability at least
$1-\delta$, the iterate sequence $(x_{t})_{t\ge1}$ output by Algorithm
\ref{alg:md} satisfies

(1) Setting $\eta_{t}=\sqrt{\frac{\breg\left(x^{*},x_{1}\right)}{6\left(G^{2}+\sigma^{2}\left(1+\log\left(\frac{1}{\delta}\right)\right)\right)T}}$,
then $\breg\left(x^{*},x_{T+1}\right)\leq4\breg\left(x^{*},x_{1}\right)$,
and
\begin{align*}
\frac{1}{T}\sum_{t=1}^{T}\left(f\left(x_{t}\right)-f\left(x^{*}\right)\right) & \le\frac{4\sqrt{6}}{\sqrt{T}}\sqrt{\breg\left(x^{*},x_{1}\right)\left(G^{2}+\sigma^{2}\left(1+\log\left(\frac{1}{\delta}\right)\right)\right)}.
\end{align*}

(2) Setting $\eta_{t}=\sqrt{\frac{\breg\left(x^{*},x_{1}\right)}{6\left(G^{2}+\sigma^{2}\left(1+\log\left(\frac{1}{\delta}\right)\right)\right)t}}$,
then $\breg\left(x^{*},x_{T+1}\right)\leq2(2+\log T)\breg\left(x^{*},x_{1}\right)$,
and
\begin{align*}
\frac{1}{T}\sum_{t=1}^{T}\left(f\left(x_{t}\right)-f\left(x^{*}\right)\right) & \le\frac{2\sqrt{6}}{\sqrt{T}}(2+\log T)\sqrt{\breg\left(x^{*},x_{1}\right)\left(G^{2}+\sigma^{2}\left(1+\log\left(\frac{1}{\delta}\right)\right)\right)}.
\end{align*}

\end{thm}
We define $\xi_{t}:=\widehat{\nabla}f\left(x_{t}\right)-\nabla f\left(x_{t}\right)$
and let $\F_{t}=\sigma\left(\xi_{1},\dots,\xi_{t-1}\right)$ denote
the natural filtration. Note that $x_{t}$ is $\F_{t}$-measurable.
The starting point of our analysis is the following inequality that
follows from the standard stochastic mirror descent analysis (see,
e.g., \citet{lan2020first}). We include the proof in the Appendix
for completeness.
\begin{lem}
\citet{lan2020first}\label{lem:md-basic-analysis} For every iteration
$t$, we have
\begin{align*}
A_{t} & \coloneqq\eta_{t}\left(f\left(x_{t}\right)-f\left(x^{*}\right)\right)-\eta_{t}^{2}G^{2}+\breg\left(x^{*},x_{t+1}\right)-\breg\left(x^{*},x_{t}\right)\\
 & \leq\eta_{t}\left\langle \xi_{t},x^{*}-x_{t}\right\rangle +\eta_{t}^{2}\left\Vert \xi_{t}\right\Vert _{*}^{2}.
\end{align*}
\end{lem}
We now turn our attention to our main concentration argument. Towards
our goal of obtaining a high-probability convergence rate, we analyze
the moment generating function for a random variable that is closely
related to the left-hand side of the inequality above. We let $\{w_{t}\}$
be a sequence where $w_{t}\ge0$ for all $t$. We define
\begin{align*}
Z_{t} & =w_{t}A_{t}-v_{t}\breg\left(x^{*},x_{t}\right), & \forall\,1\le t\le T\\
\mbox{where }v_{t} & =6\sigma^{2}\eta_{t}^{2}w_{t}^{2}\\
\text{and }S_{t} & =\sum_{i=t}^{T}Z_{i}, & \forall\,1\le t\le T+1
\end{align*}
Before proceeding with the analysis, we provide intuition for our
approach. If we consider $S_{1}$, we see that it combines the gains
in function value gaps with weights given by the sequence $\left\{ w_{t}\right\} $
and the losses given by the Bregman divergence terms $\breg\left(x^{*},x_{t}\right)$
with coefficients $v_{t}$ chosen based on the step size $\eta_{t}$
and $w_{t}$. The intuition here is that we want to transfer the error
from the stochastic error terms on the RHS of Lemma \ref{lem:md-basic-analysis}
into the loss term $v_{t}\breg\left(x^{*},x_{t}\right)$ then leverage
the progression of the Bregman divergence $\breg\left(x^{*},x_{t+1}\right)-\breg\left(x^{*},x_{t}\right)$
to absorb this loss. For the first step, we can do that by setting
the coefficient $v_{t}$ to equalize coefficient of divergence term
that will appear from the RHS of Lemma \ref{lem:md-basic-analysis}.
For the second step, we can aim at making all the divergence terms
telescope, by selecting $v_{t}$ and $w_{t}$ such that $w_{t}+v_{t}\le w_{t-1}$
to have a telescoping sum of the terms $w_{t}\breg\left(x^{*},x_{t+1}\right)-w_{t-1}\breg\left(x^{*},x_{t}\right)$.
In the end we will obtain a bound for the function value gaps in terms
of only the deterministic quantities, namely $\eta_{t},w_{t},G$ and
the initial distance. In Theorem \ref{thm:md-concentration-subgaussian},
we upper bound the moment generating function of $S_{1}$ and derive
a set of conditions for the weights $\left\{ w_{t}\right\} $ that
allow us to absorb the stochastic errors. In Corollary \ref{cor:md-convergence},
we show how to choose the weights $\left\{ w_{t}\right\} $ and obtain
a convergence rate that matches the standard rates that hold in expectation.

We now give our main concentration argument that bounds the moment
generating function of $S_{t}$ inspired by the proof of Theorem 7.3
in \citet{chung2006concentration}.
\begin{thm}
\label{thm:md-concentration-subgaussian}Suppose that $w_{t}\eta_{t}^{2}\leq\frac{1}{4\sigma^{2}}$
for every $1\leq t\leq T$. For every $1\le t\leq T+1$, we have
\begin{align*}
\E\left[\exp\left(S_{t}\right)\mid\F_{t}\right]\le & \exp\left(3\sigma^{2}\sum_{i=t}^{T}w_{i}\eta_{i}^{2}\right).
\end{align*}
\end{thm}
\begin{proof}
We proceed by induction on $t$. Consider the base case $t=T+1$.
We have the inequality holds true trivially. Next, we consider $1\leq t\leq T$.
We have
\begin{align}
\E\left[\exp\left(S_{t}\right)\mid\F_{t}\right] & =\E\left[\exp\left(Z_{t}+S_{t+1}\right)\mid\F_{t}\right]\nonumber \\
 & =\E\left[\E\left[\exp\left(Z_{t}+S_{t+1}\right)\mid\F_{t+1}\right]\mid\F_{t}\right].\label{eq:1}
\end{align}
We now analyze the inner expectation. Conditioned on $\F_{t+1}$,
$Z_{t}$ is fixed. Using the inductive hypothesis, we obtain
\begin{align}
\E\left[\exp\left(Z_{t}+S_{t+1}\right)\mid\F_{t+1}\right] & \le\exp\left(Z_{t}\right)\exp\left(3\sigma^{2}\sum_{i=t+1}^{T}w_{i}\eta_{i}^{2}\right).\label{eq:2}
\end{align}
Plugging into (\ref{eq:1}), we obtain
\begin{align}
\E\left[\exp\left(S_{t}\right)\mid\F_{t}\right] & \le\E\left[\exp\left(Z_{t}\right)\mid\F_{t}\right]\exp\left(3\sigma^{2}\sum_{i=t+1}^{T}w_{i}\eta_{i}^{2}\right).\label{eq:3}
\end{align}
By Lemma \ref{lem:md-basic-analysis}
\begin{align*}
\exp\left(Z_{t}\right) & =\exp\bigg(w_{t}\big(\eta_{t}\left(f\left(x_{t}\right)-f\left(x^{*}\right)\right)-\eta_{t}^{2}G^{2}+\breg\left(x^{*},x_{t+1}\right)-\breg\left(x^{*},x_{t}\right)\big)-v_{t}\breg\left(x^{*},x_{t}\right)\bigg)\\
 & \le\exp\bigg(w_{t}\eta_{t}\left\langle \xi_{t},x^{*}-x_{t}\right\rangle +w_{t}\eta_{t}^{2}\left\Vert \xi_{t}\right\Vert _{*}^{2}\bigg)\exp\left(-v_{t}\breg\left(x^{*},x_{t}\right)\right).
\end{align*}
Next, we analyze the first term in the last line of the above inequality
in expectation. Since $\E\left[\left\langle \xi_{t},x^{*}-x_{t}\right\rangle \mid\F_{t}\right]=0$
we can use Lemma \ref{lem:helper-taylor-vector} to obtain
\begin{align}
\E\left[\exp\left(w_{t}\eta_{t}\left\langle \xi_{t},x^{*}-x_{t}\right\rangle +w_{t}\eta_{t}^{2}\left\Vert \xi_{t}\right\Vert _{*}^{2}\right)\mid\F_{t}\right] & \le\exp\left(3\sigma^{2}\left(w_{t}^{2}\eta_{t}^{2}\left\Vert x^{*}-x_{t}\right\Vert ^{2}+w_{t}\eta_{t}^{2}\right)\right)\nonumber \\
 & \le\exp\left(3\sigma^{2}\left(2w_{t}^{2}\eta_{t}^{2}\breg\left(x^{*},x_{t}\right)+w_{t}\eta_{t}^{2}\right)\right)\label{eq:4}
\end{align}
where in the last line we used that $\breg\left(x^{*},x_{t}\right)\geq\frac{1}{2}\left\Vert x^{*}-x_{t}\right\Vert ^{2}$
from the strong convexity of $\psi$.

Plugging back into (\ref{eq:3}) and using that $v_{t}=6\sigma^{2}\eta_{t}^{2}w_{t}^{2}$,
we obtain the desired inequality
\begin{align*}
\E\left[\exp\left(S_{t}\right)\mid\F_{t}\right]\le & \exp\left(\left(6\sigma^{2}\eta_{t}^{2}w_{t}^{2}-v_{t}\right)\breg\left(x^{*},x_{t}\right)+3\sigma^{2}\sum_{i=t}^{T}w_{i}\eta_{i}^{2}\right)\\
= & \exp\left(3\sigma^{2}\sum_{i=t}^{T}w_{i}\eta_{i}^{2}\right).
\end{align*}
\end{proof}
Using Theorem \ref{thm:md-concentration-subgaussian} and Markov's
inequality, we obtain the following convergence guarantee.
\begin{cor}
\label{cor:md-convergence}Suppose the sequence $\left\{ w_{t}\right\} $
satisfies the conditions of Theorem \ref{thm:md-concentration-subgaussian}
and that $w_{t}+6\sigma^{2}\eta_{t}^{2}w_{t}^{2}\le w_{t-1}.$ For
any $\delta>0$, with probability at least $1-\delta$:
\begin{align*}
\sum_{t=1}^{T}w_{t}\eta_{t}\left(f\left(x_{t}\right)-f\left(x^{*}\right)\right)+w_{T}\breg\left(x^{*},x_{T+1}\right) & \leq w_{0}\breg\left(x^{*},x_{1}\right)+\left(G^{2}+3\sigma^{2}\right)\sum_{t=1}^{T}w_{t}\eta_{t}^{2}+\log\left(\frac{1}{\delta}\right).
\end{align*}
\end{cor}
With the above result in hand, we complete the convergence analysis
by showing how to define the sequence $\left\{ w_{t}\right\} $ with
the desired properties. For the stochastic Mirror Descent algorithm
with fixed step sizes $\eta_{t}=\frac{\eta}{\sqrt{T}}$, we set $w_{T}=\frac{1}{12\sigma^{2}\eta^{2}}$
and $w_{t-1}=w_{t}+\frac{6}{T}\sigma^{2}\eta^{2}w_{t}^{2}$ for all
$1\leq t\leq T$. For Stochastic Mirror Descent algorithm with time-varying
step sizes $\eta_{t}=\frac{\eta}{\sqrt{t}}$, we set $w_{T}=\frac{1}{12\sigma^{2}\eta^{2}\left(\sum_{t=1}^{T}\frac{1}{t}\right)}$
and $w_{t-1}=w_{t}+6\sigma^{2}\eta_{t}^{2}w_{t}^{2}$ for all $1\leq t\leq T$.
In the appendix, we show that these choices have the give us the results
in Theorem \ref{thm:md-convergence}.

\subsection{Analysis of Accelerated Stochastic Mirror Descent}

\begin{algorithm}
\caption{Accelerated Stochastic Mirror Descent Algorithm \citet{lan2020first}.}

\label{alg:acc-md}

\textbf{Parameters:} initial point $x_{0}=y_{0}=z_{0}\in\dom$, step
size $\eta$, strongly convex mirror map $\psi$

for $t=1$ to $T$:

$\quad$Set $\alpha_{t}=\frac{2}{t+1}$

$\quad$$x_{t}=\left(1-\alpha_{t}\right)y_{t-1}+\alpha_{t}z_{t-1}$

$\quad$$z_{t}=\arg\min_{x\in\dom}\left(\eta_{t}\left\langle \widehat{\nabla}f(x_{t}),x\right\rangle +\breg\left(x,z_{t-1}\right)\right)$

$\quad$$y_{t}=\left(1-\alpha_{t}\right)y_{t-1}+\alpha_{t}z_{t}$

return $y_{T}$
\end{algorithm}

In this section, we extend the analysis detailed in the previous section
to analyze the Accelerated Stochastic Mirror Descent Algorithm (Algorithm
(\ref{alg:acc-md})). We assume that $f$ satisfies the following
condition: for all $x,y\in\dom$
\begin{equation}
f(y)\le f(x)+\left\langle \nabla f\left(x\right),y-x\right\rangle +G\left\Vert y-x\right\Vert +\frac{L}{2}\left\Vert y-x\right\Vert ^{2}.\label{eq:acc-md-condition}
\end{equation}

Note that $L$-smooth functions, $G$-Lipschitz functions, and their
sums all satisfy the above condition. The full convergence guarantees
are given in Theorem \ref{thm:acc-md-convergence}. We will only highlight
the application of the previous analysis in this case. As before,
we define $\xi_{t}:=\widehat{\nabla}f\left(x_{t}\right)-\nabla f\left(x_{t}\right)$.

We also start with the inequalities shown in the standard analysis,
e.g, from \citet{lan2020first} (proof in the Appendix).
\begin{lem}
\citet{lan2020first} \label{lem:acc-md-basic-analysis}For every
iteration $t$, we have
\begin{align*}
B_{t} & \coloneqq\frac{\eta_{t}}{\alpha_{t}}\left(f\left(y_{t}\right)-f\left(x^{*}\right)\right)-\frac{\eta_{t}\left(1-\alpha_{t}\right)}{\alpha_{t}}\left(f\left(y_{t-1}\right)-f\left(x^{*}\right)\right)\\
 & \quad-\frac{\eta_{t}^{2}}{1-L\alpha_{t}\eta_{t}}G^{2}+\breg\left(x^{*},z_{t}\right)-\breg\left(x^{*},z_{t-1}\right)\\
 & \leq\eta_{t}\left\langle \xi_{t},x^{*}-z_{t-1}\right\rangle +\frac{\eta_{t}^{2}}{1-L\alpha_{t}\eta_{t}}\left\Vert \xi_{t}\right\Vert _{*}^{2}.
\end{align*}
\end{lem}
We now turn our attention to our main concentration argument. Similar
to the previous section, we define 
\begin{align*}
Z_{t} & =w_{t}B_{t}-v_{t}\breg\left(x^{*},z_{t-1}\right), & \forall\,1\leq t\leq T\\
\mbox{where }v_{t} & =6\sigma^{2}w_{t}^{2}\eta_{t}^{2}\\
\text{and }S_{t} & =\sum_{i=t}^{T}Z_{i}, & \forall\,1\le t\leq T+1
\end{align*}
Notice that here we are following the exact same step as before. By
transferring the error terms in the RHS of Lemma \ref{lem:acc-md-basic-analysis}
into the Bregman divergence terms $\breg\left(x^{*},z_{t-1}\right)$,
we can absorb them by setting the coefficients appropriately. In the
same manner, we can show the following theorem.
\begin{thm}
\label{thm:acc-md-concentration-subgaussian}Suppose that $\frac{w_{t}\eta_{t}^{2}}{1-L\alpha_{t}\eta_{t}}\leq\frac{1}{4\sigma^{2}}$
for every $0\leq t\leq T$. For every $1\leq t\leq T+1$, we have
\begin{align*}
\E\left[\exp\left(S_{t}\right)\mid\F_{t}\right] & \le\exp\Bigg(3\sigma^{2}\sum_{i=t}^{T}w_{i}\frac{\eta_{i}^{2}}{1-L\alpha_{i}\eta_{i}}\Bigg).
\end{align*}
\end{thm}

\begin{cor}
\label{cor:acc-md-convergence}Suppose the sequence $\left\{ w_{t}\right\} $
satisfies the conditions of Theorem \ref{thm:acc-md-concentration-subgaussian}.
For any $\delta>0$, the following event holds with probability at
least $1-\delta$:
\begin{align*}
 & \sum_{t=1}^{T}w_{t}\left(\frac{\eta_{t}}{\alpha_{t}}\left(f\left(y_{t}\right)-f\left(x^{*}\right)\right)-\frac{\eta_{t}\left(1-\alpha_{t}\right)}{\alpha_{t}}\left(f\left(y_{t-1}\right)-f\left(x^{*}\right)\right)\right)+w_{T}\breg\left(x^{*},z_{T}\right)\\
\leq & w_{0}\breg\left(x^{*},z_{0}\right)+\left(G^{2}+3\sigma^{2}\right)\sum_{t=1}^{T}w_{t}\frac{\eta_{t}^{2}}{1-L\alpha_{t}\eta_{t}}+\log\left(\frac{1}{\delta}\right).
\end{align*}
\end{cor}
With the above result in hand, we can complete the convergence analysis
by showing how to define the sequence $\left\{ w_{t}\right\} $ with
the desired properties. Theorem \ref{thm:acc-md-convergence} can
be obtained from corollaries \ref{cor:acc-md-convergence-final} and
\ref{cor:acc-md-convergence-final-vary} provided in the appendix,
for constant and time-varying step sizes.

\section{Non-convex case: Stochastic Gradient Descent and AdaGrad \label{sec:nonconvex}}

In this section, we consider non-convex objectives and analyze the
Stochastic Gradient Descent algorithm (Algorithm \ref{alg:sgd}) along
with two versions of AdaGrad: (1) AdaGrad-Norm \citet{ward2019adagrad}
(Algorithm \ref{alg:adagrad-norm}), where the step-size is a scalar,
and (2) the original AdaGrad algorithm \citet{duchi2011adaptive}
(Algorithm (\ref{alg:adagrad-coord})), where the step-size for each
coordinates varies. Since AdaGrad-Norm is simpler to analyze, most
results for AdaGrad have been for this scalar version either in-expectation
\citet{ward2019adagrad,faw2022power,li2020high,li2019convergence,liu2022convergence,ene2021adaptive}
or high-probability \citet{kavis2021high}. For the standard AdaGrad
algorithm, to the best of our knowledge, \citet{defossez2020simple}
is the only work that has analyzed the standard version of AdaGrad
in expectation, but their result does not adapt to noise and requires
a strong assumption: the stochastic gradients are uniformly bounded.
On the other hand, our high probability result for vanilla AdaGrad
adapts to noise and holds under relatively mild assumptions. 

Recall that, we assume that the optimization problem has domain $\dom=\R^{d}$.
As usual in non-convex analysis, we assume that $f$ is an $L$-smooth
function: $\left\Vert \n(x)-\n(y)\right\Vert \le L\left\Vert x-y\right\Vert $
for all $x,y\in\R^{d}$. Smoothness implies the following quadratic
upperbound that we will utilize: for all $x,y\in\R^{d}$
\begin{align}
f(y)-f(x) & \le\left\langle \nabla f(x),y-x\right\rangle +\frac{L}{2}\left\Vert y-x\right\Vert ^{2}.\label{eq:smoothness}
\end{align}

\subsection{Analysis of Stochastic Gradient Descent \label{sec:sgd}}

\begin{algorithm}
\caption{Stochastic Gradient Descent (SGD)}
\label{alg:sgd}

\textbf{Parameters}: initial point $x_{1}$, step sizes $\left\{ \eta_{t}\right\} $

for $t=1$ to $T$ do

$\quad$$x_{t+1}=x_{t}-\eta_{t}\hn(x_{t})$ 
\end{algorithm}

In this section, we will prove the following convergence guarantee
of Algorithm \ref{alg:sgd}. 
\begin{thm}
\label{thm:sgd-convergence}Assume $f$ is $L$-smooth and satisfies
Assumptions (1'), (2), (3). Let $\Delta_{1}\coloneqq f(x_{1})-f_{*}$.
With probability at least $1-\delta$, the iterate sequence $(x_{t})_{t\ge1}$
output by Algorithm \ref{alg:sgd} satisfies

(1) Setting $\eta_{t}=\min\left\{ \frac{1}{L};\sqrt{\frac{\Delta_{1}}{\sigma^{2}LT}}\right\} $,
\begin{align*}
\frac{1}{T}\sum_{t=1}^{T}\left\Vert \nabla f(x_{t})\right\Vert ^{2} & \le\frac{2\Delta_{1}L}{T}+5\sigma\sqrt{\frac{\Delta_{1}L}{T}}+\frac{12\sigma^{2}\log\frac{1}{\delta}}{T};
\end{align*}

(2) Setting $\eta_{t}=\frac{1}{L\sqrt{t}}$,
\begin{align*}
\frac{1}{T}\sum_{t=1}^{T}\left\Vert \nabla f(x_{t})\right\Vert ^{2} & \le\frac{2\Delta_{1}L+3\sigma^{2}\left(1+\log T\right)+12\sigma^{2}\log\frac{1}{\delta}}{\sqrt{T}}.
\end{align*}

\end{thm}
\textbf{Comparison with prior works:} When the time horizon $T$ is
known to the algorithm, by choosing the step size $\eta$ in part
$(1)$ of Theorem \ref{thm:sgd-convergence}, the bound is adaptive
to noise, i.e, when $\sigma=0$ we recover $O(\frac{1}{T})$ convergence
rate of the (deterministic) gradient descent algorithm. Notice that
the bound in this case does not have a $\log T$ term incurred. When
$T$ is unknown, the extra $\log T$ appears as a result of setting
a time-varying step size $\eta_{t}=\frac{1}{L\sqrt{t}}$. This $\log T$
appears as an additive term to the $\log\frac{1}{\delta}$ term, as
opposed to being multiplicative, i.e, $\log T\log\frac{1}{\delta}$
as in previous works \citet{li2020high,madden2020high,li2022high}.

\textbf{Analysis: }To proceed, we define for $t\geq1$
\begin{align*}
\Delta_{t} & :=f(x_{t})-f_{*};\quad\xi_{t}:=\hn(x_{t})-\nabla f(x_{t}).
\end{align*}
We let $\F_{t}:=\sigma\left(\xi_{1},\dots,\xi_{t-1}\right)$ denote
the natural filtration. Note that $x_{t}$ is $\F_{t}$-measurable.
The following lemma serves as a fundamental step of our analysis;
the proof of which can be found in the appendix.
\begin{lem}
\label{lem:sgd-basic-inequality}For $t\ge1$, we have
\begin{align}
C_{t} & \coloneqq\eta_{t}\left(1-\frac{L\eta_{t}}{2}\right)\left\Vert \nabla f(x_{t})\right\Vert ^{2}+\Delta_{t+1}-\Delta_{t}\nonumber \\
 & \le\left(L\eta_{t}^{2}-\eta_{t}\right)\left\langle \nabla f(x_{t}),\xi_{t}\right\rangle +\frac{L\eta_{t}^{2}}{2}\left\Vert \xi_{t}\right\Vert ^{2}.\label{eq:sgd-basic-inequality}
\end{align}
\end{lem}
Now we can follow the similar concentration argument from the convex
setting. The difference now is the error term in the RHS of (\ref{eq:sgd-basic-inequality})
can be transferred into the gradient term $\left\Vert \nabla f(x_{t})\right\Vert ^{2}$
instead of a function value gap term. This actually makes things easier
since this term can be readily absorbed by the gradient term in $C_{t}$,
and we do not have to carefully impose an additional condition on
$w_{t}$ to make a telescoping sum. For $w_{t}\ge0$, we define
\begin{align*}
Z_{t} & =w_{t}C_{t}-v_{t}\left\Vert \nabla f(x_{t})\right\Vert ^{2}, & \forall\,1\le t\le T\\
\mbox{where }v_{t} & =3\sigma^{2}w_{t}^{2}\eta_{t}^{2}(\eta_{t}L-1)^{2}\\
\mbox{and }S_{t} & =\sum_{i=t}^{T}Z_{i}. & \forall\,1\le t\le T+1
\end{align*}
Using the same technique as in the previous Section, we can prove
the following key inequality.
\begin{thm}
\label{thm:sgd-moment-inequality}Suppose for all $1\le t\le T$,
$\eta_{t},w_{t}$ satisfying $0\leq w_{t}\eta_{t}^{2}L\leq\frac{1}{2\sigma^{2}}$
then
\begin{equation}
\E\left[\exp\left(S_{t}\right)\mid\F_{t}\right]\leq\exp\left(3\sigma^{2}\sum_{s=t}^{T}\frac{w_{t}\eta_{t}^{2}L}{2}\right).\label{eq:sgd-key-inequality}
\end{equation}
\end{thm}
Markov's inequality gives us the following guarantee.
\begin{cor}
\textup{\label{cor:sgd-general-guarantee}For all $1\le t\le T$,
if $\eta_{t}L\le1$ and }$0\leq w_{t}\eta_{t}^{2}L\leq\frac{1}{2\sigma^{2}}$\textup{
}then\textup{
\begin{align}
 & \sum_{t=1}^{T}\left[w_{t}\eta_{t}\left(1-\frac{\eta_{t}L}{2}\right)-v_{t}\right]\left\Vert \nabla f(x_{t})\right\Vert ^{2}+w_{T}\Delta_{T+1}\nonumber \\
\leq & w_{1}\Delta_{1}+\left(\sum_{t=2}^{T}(w_{t}-w_{t-1})\Delta_{t}+3\sigma^{2}\sum_{t=1}^{T}\frac{w_{t}\eta_{t}^{2}L}{2}\right)+\log\frac{1}{\delta}.\label{eq:sgd-general-bound}
\end{align}
}
\end{cor}
Equipped with Lemmas \ref{lem:sgd-basic-inequality} and \ref{thm:sgd-moment-inequality},
we are ready to prove Theorem \ref{thm:sgd-convergence} by specifying
the choice of $w_{t}$ that satisfy the condition of Lemma \ref{thm:sgd-moment-inequality}.
In the first case, we choose $\eta_{t}=\eta$, $w_{t}=w=\frac{1}{6\sigma^{2}\eta}$
where $\eta=\min\{\frac{1}{L};\sqrt{\frac{\Delta_{1}}{\sigma^{2}LT}}\}$.
In the second case, we set $\eta_{t}=\frac{\eta}{\sqrt{t}}$, $w_{t}=w=\frac{1}{6\sigma^{2}\eta}$
where $\eta=\frac{1}{L}$. We show the full proof in the appendix.

\subsection{High probability convergence of AdaGrad-Norm and AdaGrad \label{subsec:main-body-adagradnorm-and-adagrad}}

In this section, we present our main results for the high probability
convergence for non-convex objectives of AdaGrad-Norm \citet{ward2019adagrad}
(Algorithm \ref{alg:adagrad-norm}) as well as the standard AdaGrad
\citet{duchi2011adaptive} algorithm (Algorithm \ref{alg:adagrad-coord})
that updates each coordinate separately. Here, $d\in\mathbb{N}$ denotes
the dimension of the problem, $v_{i}$ denotes the $i$-th coordinate
of a vector $v$, and $\widehat{\nabla}_{i}f(x_{t})$ denotes the
$i$-th coordinate of the stochastic gradient at time $t$. 

\begin{figure*}[t]

\begin{minipage}[t]{0.475\columnwidth}%
\begin{algorithm}[H]
\caption{AdaGrad-Norm}
\label{alg:adagrad-norm}

\textbf{Parameters}: $x_{1},\eta>0$.

for $t=1$ to $T$

$\quad$$b_{t}=\sqrt{b_{0}^{2}+\sum_{i=1}^{t}\|\widehat{\nabla}f(x_{i})\|^{2}}$

$\quad$$x_{t+1}=x_{t}-\frac{\eta}{b_{t}}\widehat{\nabla}f(x_{t})$
\end{algorithm}
\end{minipage}\hfill{}%
\begin{minipage}[t]{0.475\columnwidth}%
\begin{algorithm}[H]
\caption{AdaGrad}

\label{alg:adagrad-coord}

\textbf{Parameters}: $x_{1},b_{0}\in\R^{d}$ and $\eta\in\R$.

for $t=1$ to $T$ do

$\quad$$b_{t,i}=\sqrt{b_{0,i}^{2}+\sum_{j=1}^{t}\widehat{\nabla}_{i}f(x_{j})^{2}},$
for $i\in[d]$.

$\quad$$x_{t+1,i}=x_{t,i}-\frac{\eta}{b_{t,i}}\widehat{\nabla}_{i}f(x_{t})$,
for $i\in[d]$.
\end{algorithm}
\end{minipage}

\end{figure*}

\paragraph{Comparison with prior works.}

\citet{ward2019adagrad,faw2022power} show the convergence of AdaGrad-Norm
with polynomial dependency on $\mathrm{poly}\left(\frac{1}{\delta}\right)$
where $1-\delta$ is the success probability. The latter relaxes several
assumptions made in the former, including the boundedness of the gradients
and noise variance. When assuming a sub-Gaussian noise, \citet{kavis2021high}
show a convergence in high probability, but still assume that the
gradients are bounded which circumvents many of the difficulties due
to the error term. We remove this assumption and establish the convergence
of AdaGrad-Norm in the theorem \ref{thm:adagrad-nonconvex}. Unlike
existing work, the technique employed to prove this theorem readily
extends to the standard version of AdaGrad (Algorithm \ref{alg:adagrad-coord})
with per-coordinate update.

For simplicity, we let $\Delta_{t}:=f(x_{t})-f_{*}$, where $f_{*}$
is any valid lower bound for $f$. 
\begin{thm}
\label{thm:adagrad-nonconvex}If $f$ is $L$-smooth and satisfies
assumptions (1'), (2) and (3). With probability at least $1-\delta$,
the iterate sequence $(x_{t})_{t\geq1}$ output by AdaGrad-Norm (Algorithm
\ref{alg:adagrad-norm}) satisfies
\begin{align*}
\frac{1}{T}\sum_{t=1}^{T}\norm{\n(x_{t})}^{2} & \leq g(\delta)\cdot O\left(\frac{\sigma}{\sqrt{T}}+\frac{r(\delta)}{T}\right).
\end{align*}
where 
\begin{align*}
g(\delta) & :=O\left(\Delta_{1}+c(\delta)\sqrt{\log\frac{T}{\delta}}+L\log\left(\sigma\sqrt{T}+r(\delta)\right)\right)\\
c(\delta) & :=O\left(\sigma^{3}\log\left(\frac{1}{\delta}\right)+\sigma\log\left(1+\sigma^{2}T+\sigma^{2}\log\frac{1}{\delta}\right)+\sigma\log\left(\sigma\sqrt{T}+r(\delta)\right)\right),\text{ and }\\
r(\delta) & :=O(\Delta_{1}+\sigma^{2}\log\frac{1}{\delta}+L\log L)
\end{align*}
are polylog terms. 
\end{thm}
The next theorem show the first convergence result in high-probability
for vanilla AdaGrad in the non-convex regime.
\begin{thm}
\label{thm:adagrad-coord-nonconvex}If $f$ is $L$-smooth and satisfies
assumptions (1'), (2) and (3). With probability at least $1-\delta$,
the iterate sequence $(x_{t})_{t\geq1}$ output by AdaGrad (Algorithm
\ref{alg:adagrad-coord}) satisfies
\begin{align*}
\frac{1}{T}\sum_{t=1}^{T}\left\Vert \nabla f(x_{t})\right\Vert _{1}^{2} & \leq g(\delta)\cdot O\left(\frac{\left\Vert \sigma\right\Vert _{1}}{\sqrt{T}}+\frac{r(\delta)}{T}\right),
\end{align*}
where 
\begin{align*}
g(\delta) & :=O\left(\Delta_{1}+\left(d\sigma_{\max}+\sum_{i=1}^{d}c_{i}(\delta)\right)\sqrt{\log\frac{dT}{\delta}}+dL\log\left(\left\Vert \sigma\right\Vert _{1}\sqrt{T}+r(\delta)\right)\right),\\
c_{i}(\delta) & :=O\left(\sigma_{i}^{3}\log\left(\frac{d}{\delta}\right)+\sigma_{i}\log\left(1+\sigma_{i}^{2}T+\sigma_{i}^{2}\log\frac{d}{\delta}\right)+\left\Vert \sigma\right\Vert _{1}\log\left(\left\Vert \sigma\right\Vert _{1}\sqrt{T}+r(\delta)\right)\right),\text{ and }\\
r(\delta) & :=O\left(\Delta_{1}+\left\Vert \sigma^{2}\right\Vert _{1}\log\left(\frac{d}{\delta}\right)+\left\Vert \sigma\right\Vert _{1}\sqrt{\log\frac{d}{\delta}}+Ld\log L\right),
\end{align*}
are polylog terms.
\end{thm}
Both of these results are adaptive to noise: the rate $\tilde{O}\left(\frac{1}{\sqrt{T}}\right)$
will improve to $\tilde{O}\left(\frac{1}{T}\right)$ as the noise
$\sigma$ approaches $0$. Furthermore, these results hold regardless
of how $\eta$ and $b_{0}$ is set. 

\paragraph{Analysis overview.}

The first key new technique is unlike prior works, we do not use the
division by the step size, which makes the analysis of AdaGrad-Norm
and AdaGrad virtually the same. We can thus focus on AdaGrad-Norm.
To obtain a high probability bound, our analysis of AdaGrad-Norm utilizes
the same martingale concentration technique as presented throughout
this paper to bound the error terms $\eta_{t}\left\langle \nabla f(x_{t}),\xi_{t}\right\rangle $.
However, the step size $\eta_{t}=\frac{\eta}{b_{t}}$ now has a dependency
on the randomness at time $t$ due to $b_{t}$, preventing us from
applying Lemma \ref{lem:helper-taylor-vector}. To circumvent this,
inspired by \citet{ward2019adagrad}, we introduce a proxy step size
$a_{t}:=b_{t-1}^{2}+\norm{\nabla f(x_{t})}^{2}$ that replaces the
stochastic gradient with the true gradient at time $t$ for analysis
purposes. Using that along with standard smoothness analysis, we obtain:
\begin{lem}
\label{lem:adagrad-basics}For $t\ge1$, let $\xi_{t}=\hn(x_{t})-\n(x_{t})$,
$a_{t}^{2}:=b_{t-1}^{2}+\norm{\nabla f(x_{t})}^{2}$, and $M_{t}=\max_{i\le t}\left\Vert \xi_{i}\right\Vert $,
then we have
\begin{align*}
\sum_{t=1}^{T}\frac{\norm{\nf(x_{t})}^{2}}{b_{t}} & \leq\frac{\Delta_{1}}{\eta}+\frac{M_{T}}{2}\left[\sum_{t=1}^{T}\frac{\norm{\nf(x_{t})}^{2}}{a_{t}^{2}}+\sum_{t=1}^{T}\frac{\norm{\xi_{t}}^{2}}{b_{t}^{2}}\right]-\sum_{t=1}^{T}\frac{1}{a_{t}}\left\langle \nf(x_{t}),\xi_{t}\right\rangle +\sum_{t=1}^{T}\frac{L\eta}{2b_{t}^{2}}\norm{\hn(x_{t})}^{2}.
\end{align*}
\end{lem}
Now, the randomness at time $t$ of the error term $\frac{1}{a_{t}}\left\langle \nf(x_{t}),\xi_{t}\right\rangle $
only depends on $\xi_{t}$, which follows a sub-Gaussian distribution
with mean $0$. Hence, we can utilize our previous techniques to bound
$-\sum_{t=1}^{T}\frac{1}{a_{t}}\left\langle \nf(x_{t}),\xi_{t}\right\rangle $
with high probability. Comparing to the analysis in expectation from
\citet{ward2019adagrad}, terms like $\sum_{t=1}^{T}\frac{\norm{\nf(x_{t})}^{2}}{a_{t}^{2}}$
must be handled more carefully to obtain a high probability bound.
A bound for $M_{T}$ has also been derived in previous works by \citet{li2020high,liu2022convergence}.
Combining with Lemma \ref{lem:adagrad-basics}, we obtain the following
lemma.
\begin{lem}
\label{lem:martingale-dif-error-term}With probability at least $1-2\delta$,
we have
\begin{align*}
\sum_{t=1}^{T}\frac{\norm{\nf(x_{t})}^{2}}{b_{t}} & \leq\frac{\Delta_{1}}{\eta}+\sigma\sqrt{\log\frac{T}{\delta}}\left[8\log\left(\frac{b_{T}}{b_{0}}\right)+5\sum_{t=1}^{T}\frac{\norm{\xi_{t}}^{2}}{b_{t}^{2}}\right]+\sigma\sqrt{\log\frac{1}{\delta}}+L\eta\log\frac{b_{T}}{b_{0}}.
\end{align*}
\end{lem}
Since, $\sum_{t=1}^{T}\frac{\norm{\nf(x_{t})}^{2}}{b_{t}}\geq\frac{1}{b_{T}}\sum_{t=1}^{T}\norm{\nf(x_{t})}^{2}$,
it suffices to bound $b_{T}$ and $\sum_{t=1}^{T}\frac{\norm{\xi_{t}}^{2}}{b_{t}^{2}}$
from this point on (see Lemma \ref{lem:bound-on-bT} and Lemma \ref{lem:bound-on-eps-sq-by-bt-sq}).
The analysis for these terms utilize similar martingale techniques
throughout this paper, where the details are deferred to Section \ref{sec:AdaGrad-Norm-second-proof}
of the Appendix. For the coordinate version of AdaGrad, since our
techniques only rely on addition and scalar multiplication, we can
(with some effort) generalize our technique to the standard per-coordinate
AdaGrad algorithm. The full proofs for vanilla AdaGrad are presented
in Section \ref{sec:AdaGrad-(coordinate)-analysis} of the Appendix.

\section{Conclusion}

In this work, we present a generic approach to prove high probability
convergence of stochastic gradient methods under sub-Gaussian noise.
In the convex case, we show high probability bounds for stochastic
and accelerated stochastic mirror descent that depend on the distance
from the initial solution to the optimal solution and do not require
the bounded domain or bounded Bregman divergence assumptions. In the
non-convex case, we apply the same approach and obtain a high probability
bound for SGD that improves over existing works. We also show that
the boundedness of the gradients can be removed when showing high
probability convergence of AdaGrad-Norm. Finally, we show that our
analysis for AdaGrad-Norm can be extended to the standard per-coordinate
AdaGrad algorithm to obtain one of the first high probability convergence
result for standard AdaGrad. 

For future work, it would be interesting to see whether our method
can be applied to analyze AdaGrad-Norm and/or AdaGrad in the convex
setting without restrictive assumptions. Extending this approach to
the heavy tail setting and finding its applications in other problems
are some of the potential future directions.

\bibliographystyle{plainnat}
\bibliography{ref}

\begin{thebibliography}{28}
\providecommand{\natexlab}[1]{#1}
\providecommand{\url}[1]{\texttt{#1}}
\expandafter\ifx\csname urlstyle\endcsname\relax
  \providecommand{\doi}[1]{doi: #1}\else
  \providecommand{\doi}{doi: \begingroup \urlstyle{rm}\Url}\fi

\bibitem[Chung and Lu(2006)]{chung2006concentration}
Fan Chung and Linyuan Lu.
\newblock Concentration inequalities and martingale inequalities: a survey.
\newblock \emph{Internet mathematics}, 3\penalty0 (1):\penalty0 79--127, 2006.

\bibitem[Cutkosky and Mehta(2021)]{cutkosky2021high}
Ashok Cutkosky and Harsh Mehta.
\newblock High-probability bounds for non-convex stochastic optimization with
  heavy tails.
\newblock \emph{Advances in Neural Information Processing Systems},
  34:\penalty0 4883--4895, 2021.

\bibitem[Davis et~al.(2021)Davis, Drusvyatskiy, Xiao, and Zhang]{davis2021low}
Damek Davis, Dmitriy Drusvyatskiy, Lin Xiao, and Junyu Zhang.
\newblock From low probability to high confidence in stochastic convex
  optimization.
\newblock \emph{Journal of machine learning research}, 22\penalty0 (49), 2021.

\bibitem[D{\'e}fossez et~al.(2022)D{\'e}fossez, Bottou, Bach, and
  Usunier]{defossez2020simple}
Alexandre D{\'e}fossez, L{\'e}on Bottou, Francis Bach, and Nicolas Usunier.
\newblock A simple convergence proof of adam and adagrad.
\newblock \emph{Transactions on Machine Learning Research}, 2022.

\bibitem[Duchi et~al.(2011)Duchi, Hazan, and Singer]{duchi2011adaptive}
John Duchi, Elad Hazan, and Yoram Singer.
\newblock Adaptive subgradient methods for online learning and stochastic
  optimization.
\newblock \emph{Journal of machine learning research}, 12\penalty0 (7), 2011.

\bibitem[Dvurechensky and Gasnikov(2016)]{dvurechensky2016stochastic}
Pavel Dvurechensky and Alexander Gasnikov.
\newblock Stochastic intermediate gradient method for convex problems with
  stochastic inexact oracle.
\newblock \emph{Journal of Optimization Theory and Applications}, 171\penalty0
  (1):\penalty0 121--145, 2016.

\bibitem[Ene et~al.(2021)Ene, Nguyen, and Vladu]{ene2021adaptive}
Alina Ene, Huy~L Nguyen, and Adrian Vladu.
\newblock Adaptive gradient methods for constrained convex optimization and
  variational inequalities.
\newblock In \emph{Proceedings of the AAAI Conference on Artificial
  Intelligence}, volume~35, pages 7314--7321, 2021.

\bibitem[Faw et~al.(2022)Faw, Tziotis, Caramanis, Mokhtari, Shakkottai, and
  Ward]{faw2022power}
Matthew Faw, Isidoros Tziotis, Constantine Caramanis, Aryan Mokhtari, Sanjay
  Shakkottai, and Rachel Ward.
\newblock The power of adaptivity in sgd: Self-tuning step sizes with unbounded
  gradients and affine variance.
\newblock \emph{arXiv preprint arXiv:2202.05791}, 2022.

\bibitem[Freedman(1975)]{freedman1975tail}
David~A Freedman.
\newblock On tail probabilities for martingales.
\newblock \emph{the Annals of Probability}, pages 100--118, 1975.

\bibitem[Gorbunov et~al.(2020)Gorbunov, Danilova, and
  Gasnikov]{gorbunov2020stochastic}
Eduard Gorbunov, Marina Danilova, and Alexander Gasnikov.
\newblock Stochastic optimization with heavy-tailed noise via accelerated
  gradient clipping.
\newblock \emph{Advances in Neural Information Processing Systems},
  33:\penalty0 15042--15053, 2020.

\bibitem[Harvey et~al.(2019)Harvey, Liaw, Plan, and Randhawa]{harvey2019tight}
Nicholas~JA Harvey, Christopher Liaw, Yaniv Plan, and Sikander Randhawa.
\newblock Tight analyses for non-smooth stochastic gradient descent.
\newblock In \emph{Conference on Learning Theory}, pages 1579--1613. PMLR,
  2019.

\bibitem[Hazan and Kale(2014)]{hazan2014beyond}
Elad Hazan and Satyen Kale.
\newblock Beyond the regret minimization barrier: optimal algorithms for
  stochastic strongly-convex optimization.
\newblock \emph{The Journal of Machine Learning Research}, 15\penalty0
  (1):\penalty0 2489--2512, 2014.

\bibitem[Juditsky et~al.(2011)Juditsky, Nemirovski, and Tauvel]{JuditskyNT11}
Anatoli Juditsky, Arkadi Nemirovski, and Claire Tauvel.
\newblock Solving variational inequalities with stochastic mirror-prox
  algorithm.
\newblock \emph{Stochastic Systems}, 1\penalty0 (1):\penalty0 17--58, 2011.

\bibitem[Kakade and Tewari(2008)]{kakade2008generalization}
Sham~M Kakade and Ambuj Tewari.
\newblock On the generalization ability of online strongly convex programming
  algorithms.
\newblock \emph{Advances in Neural Information Processing Systems}, 21, 2008.

\bibitem[Kavis et~al.(2021)Kavis, Levy, and Cevher]{kavis2021high}
Ali Kavis, Kfir~Yehuda Levy, and Volkan Cevher.
\newblock High probability bounds for a class of nonconvex algorithms with
  adagrad stepsize.
\newblock In \emph{International Conference on Learning Representations}, 2021.

\bibitem[Lan(2012)]{lan2012optimal}
Guanghui Lan.
\newblock An optimal method for stochastic composite optimization.
\newblock \emph{Mathematical Programming}, 133\penalty0 (1):\penalty0 365--397,
  2012.

\bibitem[Lan(2020)]{lan2020first}
Guanghui Lan.
\newblock \emph{First-order and stochastic optimization methods for machine
  learning}.
\newblock Springer, 2020.

\bibitem[Li and Liu(2022)]{li2022high}
Shaojie Li and Yong Liu.
\newblock High probability guarantees for nonconvex stochastic gradient descent
  with heavy tails.
\newblock In \emph{International Conference on Machine Learning}, pages
  12931--12963. PMLR, 2022.

\bibitem[Li and Orabona(2019)]{li2019convergence}
Xiaoyu Li and Francesco Orabona.
\newblock On the convergence of stochastic gradient descent with adaptive
  stepsizes.
\newblock In \emph{The 22nd International Conference on Artificial Intelligence
  and Statistics}, pages 983--992. PMLR, 2019.

\bibitem[Li and Orabona(2020)]{li2020high}
Xiaoyu Li and Francesco Orabona.
\newblock A high probability analysis of adaptive sgd with momentum.
\newblock \emph{arXiv preprint arXiv:2007.14294}, 2020.

\bibitem[Liu et~al.(2022)Liu, Nguyen, Ene, and Nguyen]{liu2022convergence}
Zijian Liu, Ta~Duy Nguyen, Alina Ene, and Huy~L Nguyen.
\newblock On the convergence of adagrad on {$\mathbb{R}^{d} $}: Beyond
  convexity, non-asymptotic rate and acceleration.
\newblock \emph{arXiv preprint arXiv:2209.14827}, 2022.

\bibitem[Madden et~al.(2020)Madden, Dall'Anese, and Becker]{madden2020high}
Liam Madden, Emiliano Dall'Anese, and Stephen Becker.
\newblock High probability convergence and uniform stability bounds for
  nonconvex stochastic gradient descent.
\newblock \emph{arXiv preprint arXiv:2006.05610}, 2020.

\bibitem[Nazin et~al.(2019)Nazin, Nemirovsky, Tsybakov, and
  Juditsky]{nazin2019algorithms}
Alexander~V Nazin, Arkadi~S Nemirovsky, Alexandre~B Tsybakov, and Anatoli~B
  Juditsky.
\newblock Algorithms of robust stochastic optimization based on mirror descent
  method.
\newblock \emph{Automation and Remote Control}, 80\penalty0 (9):\penalty0
  1607--1627, 2019.

\bibitem[Nemirovski et~al.(2009)Nemirovski, Juditsky, Lan, and
  Shapiro]{nemirovski2009robust}
Arkadi Nemirovski, Anatoli Juditsky, Guanghui Lan, and Alexander Shapiro.
\newblock Robust stochastic approximation approach to stochastic programming.
\newblock \emph{SIAM Journal on optimization}, 19\penalty0 (4):\penalty0
  1574--1609, 2009.

\bibitem[Rakhlin et~al.(2011)Rakhlin, Shamir, and Sridharan]{rakhlin2011making}
Alexander Rakhlin, Ohad Shamir, and Karthik Sridharan.
\newblock Making gradient descent optimal for strongly convex stochastic
  optimization.
\newblock \emph{arXiv preprint arXiv:1109.5647}, 2011.

\bibitem[Vershynin(2018)]{vershynin2018high}
Roman Vershynin.
\newblock \emph{High-dimensional probability: An introduction with applications
  in data science}, volume~47.
\newblock Cambridge university press, 2018.

\bibitem[Ward et~al.(2019)Ward, Wu, and Bottou]{ward2019adagrad}
Rachel Ward, Xiaoxia Wu, and Leon Bottou.
\newblock Adagrad stepsizes: Sharp convergence over nonconvex landscapes.
\newblock In \emph{International Conference on Machine Learning}, pages
  6677--6686. PMLR, 2019.

\bibitem[Zhang et~al.(2020)Zhang, Karimireddy, Veit, Kim, Reddi, Kumar, and
  Sra]{zhang2020adaptive}
Jingzhao Zhang, Sai~Praneeth Karimireddy, Andreas Veit, Seungyeon Kim, Sashank
  Reddi, Sanjiv Kumar, and Suvrit Sra.
\newblock Why are adaptive methods good for attention models?
\newblock \emph{Advances in Neural Information Processing Systems ({NeurIPS})},
  33:\penalty0 15383--15393, 2020.

\end{thebibliography}

\appendix

\section{Proof of Lemma \ref{lem:helper-taylor-vector}}

To prove Lemma \ref{lem:helper-taylor-vector}, we will prove the
following statement.
\begin{lem}
\label{lem:helper-taylor}For any $a\ge0$, $0\le b\le\frac{1}{2\sigma}$
and an $\sigma$-sub-Gaussian random variable $X$,
\[
\E\left[1+b^{2}X^{2}+\sum_{i=2}^{\infty}\frac{1}{i!}\left(aX+b^{2}X^{2}\right)^{i}\right]\le\exp\left(3\left(a^{2}+b^{2}\right)\sigma^{2}\right).
\]
Especially, when $b=0$, we have 
\[
\E\left[1+\sum_{i=2}^{\infty}\frac{1}{i!}\left(aX\right)^{i}\right]\le\exp\left(2a^{2}\sigma^{2}\right).
\]
\end{lem}
\begin{proof}[Proof of Lemma\textbf{ }\ref{lem:helper-taylor}]
Consider two cases either $a\ge1/(2\sigma)$ or $a\le1/(2\sigma)$.
First suppose $a\ge1/(2\sigma)$. We use the inequality $uv\le\frac{u^{2}}{4}+v^{2}$
here to first obtain
\begin{align*}
\left(aX+b^{2}X^{2}\right)^{i} & \leq\left|aX+b^{2}X^{2}\right|^{i}\leq\left(a\left|X\right|+b^{2}X^{2}\right)^{i}\leq\left(\frac{1}{4\sigma^{2}}X^{2}+a^{2}\sigma^{2}+b^{2}X^{2}\right)^{i}.
\end{align*}
Thus, we have
\begin{align*}
\E\left[1+b^{2}X^{2}+\sum_{i=2}^{\infty}\frac{1}{i!}\left(aX+b^{2}X^{2}\right)^{i}\right] & \le\E\left[1+b^{2}X^{2}+\sum_{i=2}^{\infty}\frac{1}{i!}\left(\frac{1}{4\sigma^{2}}X^{2}+a^{2}\sigma^{2}+b^{2}X^{2}\right)^{i}\right]\\
 & =\E\left[b^{2}X^{2}+\exp\left(\left(\frac{1}{4\sigma^{2}}+b^{2}\right)X^{2}+a^{2}\sigma^{2}\right)-\left(\frac{1}{4\sigma^{2}}+b^{2}\right)X^{2}-a^{2}\sigma^{2}\right]\\
 & =\E\left[\exp\left(\left(\frac{1}{4\sigma^{2}}+b^{2}\right)X^{2}+a^{2}\sigma^{2}\right)-\frac{1}{4\sigma^{2}}X^{2}-a^{2}\sigma^{2}\right]\\
 & \le\exp\left(\left(\frac{1}{4\sigma^{2}}+b^{2}\right)\sigma^{2}+a^{2}\sigma^{2}\right)\\
 & \le\exp\left(2a^{2}\sigma^{2}+b^{2}\sigma^{2}\right)\\
 & \leq\exp\left(3\left(a^{2}+b^{2}\right)\sigma^{2}\right).
\end{align*}
Next, let $c=\max(a,b)\le1/(2\sigma)$. We have
\begin{align*}
\E\left[1+b^{2}X^{2}+\sum_{i=2}^{\infty}\frac{1}{i!}\left(aX+b^{2}X^{2}\right)^{i}\right] & =\E\left[\exp\left(aX+b^{2}X^{2}\right)-aX\right]\\
 & \le\E\left[\left(aX+\exp\left(a^{2}X^{2}\right)\right)\exp\left(b^{2}X^{2}\right)-aX\right]\\
 & =\E\left[\exp\left(\left(a^{2}+b^{2}\right)X^{2}\right)+aX\left(\exp\left(b^{2}X^{2}\right)-1\right)\right]\\
 & \le\E\left[\exp\left(\left(a^{2}+b^{2}\right)X^{2}\right)+c\left|X\right|\left(\exp\left(c^{2}X^{2}\right)-1\right)\right]\\
 & \le\E\left[\exp\left(\left(a^{2}+b^{2}\right)X^{2}\right)+\exp\left(2c^{2}X^{2}\right)-1\right]\\
 & \le\E\left[\exp\left(\left(a^{2}+b^{2}+2c^{2}\right)X^{2}\right)\right]\\
 & \le\exp\left(\left(a^{2}+b^{2}+2c^{2}\right)\sigma^{2}\right)\\
 & \leq\exp\left(3\left(a^{2}+b^{2}\right)\sigma^{2}\right).
\end{align*}
In the first inequality, we use the inequality $e^{x}-x\le e^{x^{2}}\forall x$.
In the third inequality, we use $x\left(e^{x^{2}}-1\right)\le e^{2x^{2}}-1\ \forall x$.
This inequality can be proved with the Taylor expansion.
\begin{align*}
x\left(e^{x^{2}}-1\right) & =\sum_{i=1}^{\infty}\frac{1}{i!}x^{2i+1}\\
 & \le\sum_{i=1}^{\infty}\frac{1}{i!}\frac{x^{2i}+x^{2i+2}}{2}\\
 & =\frac{x^{2}}{2}+\sum_{i=2}^{\infty}\left(\frac{1+i}{2i!}\right)x^{2i}\\
 & \le\frac{x^{2}}{2}+\sum_{i=2}^{\infty}\left(\frac{2^{i}}{i!}\right)x^{2i}\\
 & \le e^{2x^{2}}-1.
\end{align*}
The case when $b=0$ simply follows from the above proof.
\end{proof}

\begin{proof}[Proof of Lemma \ref{lem:helper-taylor-vector}]
Using Taylor expansion of $e^{x}$ and the fact that $\E\left[X\right]=0$
we have 
\begin{align*}
\E\left[\exp\left(\left\langle a,X\right\rangle +b^{2}\left\Vert X\right\Vert ^{2}\right)\right] & =\E\left[1+\left\langle a,X\right\rangle +b^{2}\left\Vert X\right\Vert ^{2}+\sum_{i=2}^{\infty}\frac{1}{i!}\left(\left\langle a,X\right\rangle +b^{2}\left\Vert X\right\Vert ^{2}\right)^{i}\right]\\
 & =\E\left[1+b^{2}\left\Vert X\right\Vert ^{2}+\sum_{i=2}^{\infty}\frac{1}{i!}\left(\left\langle a,X\right\rangle +b^{2}\left\Vert X\right\Vert ^{2}\right)^{i}\right]\\
 & \le\E\left[1+b^{2}\left\Vert X\right\Vert ^{2}+\sum_{i=2}^{\infty}\frac{1}{i!}\left(\left\Vert a\right\Vert _{*}\left\Vert X\right\Vert +b^{2}\left\Vert X\right\Vert ^{2}\right)^{i}\right]
\end{align*}
where for the last line we use Cauchy-Schwartz to obtain $\left\langle a,X\right\rangle \le\left\Vert a\right\Vert _{*}\left\Vert X\right\Vert $.
Now applying Lemma \ref{lem:helper-taylor}, we obtain 
\begin{align*}
\E\left[\exp\left(\left\langle a,X\right\rangle +b^{2}\left\Vert X\right\Vert ^{2}\right)\right] & \le\exp\left(3\left(\left\Vert a\right\Vert _{*}^{2}+b^{2}\right)\sigma^{2}\right).
\end{align*}
Similarly, we obtain the corresponding bound for the case $b=0$.
\end{proof}

\section{Missing Proofs from Section \ref{sec:convex}}

\subsection{Stochastic Mirror Descent}

\begin{proof}[Proof of Lemma (\ref{lem:md-basic-analysis})]
 By the optimality condition, we have
\[
\left\langle \eta_{t}\widehat{\nabla}f(x_{t})+\nabla_{x}\breg\left(x_{t+1},x_{t}\right),x^{*}-x_{t+1}\right\rangle \ge0
\]
 and thus
\[
\left\langle \eta_{t}\widehat{\nabla}f(x_{t}),x_{t+1}-x^{*}\right\rangle \leq\left\langle \nabla_{x}\breg\left(x_{t+1},x_{t}\right),x^{*}-x_{t+1}\right\rangle .
\]
Note that 
\begin{align*}
\left\langle \nabla_{x}\breg\left(x_{t+1},x_{t}\right),x^{*}-x_{t+1}\right\rangle  & =\left\langle \nabla\psi\left(x_{t+1}\right)-\nabla\psi\left(x_{t}\right),x^{*}-x_{t+1}\right\rangle \\
 & =\breg\left(x^{*},x_{t}\right)-\breg\left(x_{t+1},x_{t}\right)-\breg\left(x^{*},x_{t+1}\right),
\end{align*}
 and thus
\begin{align*}
\eta_{t}\left\langle \widehat{\nabla}f(x_{t}),x_{t+1}-x^{*}\right\rangle  & \leq\breg\left(x^{*},x_{t}\right)-\breg\left(x^{*},x_{t+1}\right)-\breg\left(x_{t+1},x_{t}\right)\\
 & \leq\breg\left(x^{*},x_{t}\right)-\breg\left(x^{*},x_{t+1}\right)-\frac{1}{2}\left\Vert x_{t+1}-x_{t}\right\Vert ^{2},
\end{align*}
 where we have used that $\breg\left(x_{t+1},x_{t}\right)\geq\frac{1}{2}\left\Vert x_{t+1}-x_{t}\right\Vert ^{2}$
by the strong convexity of $\psi$.

By convexity, we have

\[
f\left(x_{t}\right)-f\left(x^{*}\right)\le\left\langle \nabla f\left(x_{t}\right),x_{t}-x^{*}\right\rangle =\left\langle \xi_{t},x^{*}-x_{t}\right\rangle +\left\langle \widehat{\nabla}f\left(x_{t}\right),x_{t}-x^{*}\right\rangle .
\]
Combining the two inequalities, we obtain
\begin{align*}
 & \eta_{t}\left(f\left(x_{t}\right)-f\left(x^{*}\right)\right)+\breg\left(x^{*},x_{t+1}\right)-\breg\left(x^{*},x_{t}\right)\\
 & \le\eta_{t}\left\langle \xi_{t},x^{*}-x_{t}\right\rangle +\eta_{t}\left\langle \widehat{\nabla}f(x_{t}),x_{t}-x_{t+1}\right\rangle -\frac{1}{2}\left\Vert x_{t+1}-x_{t}\right\Vert ^{2}\\
 & \le\eta_{t}\left\langle \xi_{t},x^{*}-x_{t}\right\rangle +\frac{\eta_{t}^{2}}{2}\left\Vert \widehat{\nabla}f(x_{t})\right\Vert _{*}^{2}.
\end{align*}
Using the triangle inequality and the bounded gradient assumption
$\left\Vert \nabla f(x)\right\Vert _{*}\leq G$ , we obtain
\[
\left\Vert \widehat{\nabla}f(x_{t})\right\Vert _{*}^{2}=\left\Vert \xi_{t}+\nabla f(x_{t})\right\Vert _{*}^{2}\leq2\left\Vert \xi_{t}\right\Vert _{*}^{2}+2\left\Vert \nabla f(x_{t})\right\Vert _{*}^{2}\leq2\left(\left\Vert \xi_{t}\right\Vert _{*}^{2}+G^{2}\right).
\]
Thus
\[
\eta_{t}\left(f\left(x_{t}\right)-f\left(x^{*}\right)\right)+\breg\left(x^{*},x_{t+1}\right)-\breg\left(x^{*},x_{t}\right)\leq\eta_{t}\left\langle \xi_{t},x^{*}-x_{t}\right\rangle +\eta_{t}^{2}\left(\left\Vert \xi_{t}\right\Vert _{*}^{2}+G^{2}\right)
\]
as needed.
\end{proof}

\begin{proof}[Proof of Corollary \ref{cor:md-convergence}]
Let
\[
K=3\sigma^{2}\sum_{t=1}^{T}w_{t}\eta_{t}^{2}+\log\left(\frac{1}{\delta}\right).
\]
By Theorem \ref{thm:md-concentration-subgaussian} and Markov's inequality,
we have
\begin{align*}
\Pr\left[S_{1}\geq K\right] & \leq\Pr\left[\exp\left(S_{1}\right)\geq\exp\left(K\right)\right]\\
 & \leq\exp\left(-K\right)\E\left[\exp\left(S_{1}\right)\right]\\
 & \leq\exp\left(-K\right)\exp\left(3\sigma^{2}\sum_{t=1}^{T}w_{t}\eta_{t}^{2}\right)\\
 & =\delta.
\end{align*}
Note that since $v_{t}+w_{t}\le w_{t-1}$
\begin{align*}
S_{1} & =\sum_{t=1}^{T}Z_{t}=\sum_{t=1}^{T}w_{t}\eta_{t}\left(f\left(x_{t}\right)-f\left(x^{*}\right)\right)-G^{2}\sum_{t=1}^{T}w_{t}\eta_{t}^{2}+\sum_{t=1}^{T}\left(w_{t}\breg\left(x^{*},x_{t+1}\right)-(v_{t}+w_{t})\breg\left(x^{*},x_{t}\right)\right)\\
 & \ge\sum_{t=1}^{T}w_{t}\eta_{t}\left(f\left(x_{t}\right)-f\left(x^{*}\right)\right)-G^{2}\sum_{t=1}^{T}w_{t}\eta_{t}^{2}+\sum_{t=1}^{T}\left(w_{t}\breg\left(x^{*},x_{t+1}\right)-w_{t-1}\breg\left(x^{*},x_{t}\right)\right)\\
 & =\sum_{t=1}^{T}w_{t}\eta_{t}\left(f\left(x_{t}\right)-f\left(x^{*}\right)\right)-G^{2}\sum_{t=1}^{T}w_{t}\eta_{t}^{2}+w_{T}\breg\left(x^{*},x_{T+1}\right)-w_{0}\breg\left(x^{*},x_{1}\right).
\end{align*}
Therefore, with probability at least $1-\delta$, we have
\[
\sum_{t=1}^{T}w_{t}\eta_{t}\left(f\left(x_{t}\right)-f\left(x^{*}\right)\right)+w_{T}\breg\left(x^{*},x_{T+1}\right)\leq w_{0}\breg\left(x^{*},x_{1}\right)+\left(G^{2}+3\sigma^{2}\right)\sum_{t=1}^{T}w_{t}\eta_{t}^{2}+\log\left(\frac{1}{\delta}\right).
\]
\end{proof}

With the above result in hand, we complete the convergence analysis
by showing how to define the sequence $\left\{ w_{t}\right\} $ with
the desired properties. Theorem \ref{thm:md-convergence} can be obtained
from the two following corollaries.
\begin{cor}
\label{cor:md-convergence-constant-stepsize}Suppose we run the Stochastic
Mirror Descent algorithm with fixed step sizes $\eta_{t}=\frac{\eta}{\sqrt{T}}$.
Let $w_{T}=\frac{1}{12\sigma^{2}\eta^{2}}$ and $w_{t-1}=w_{t}+\frac{6}{T}\sigma^{2}\eta^{2}w_{t}^{2}$
for all $1\leq t\leq T$. The sequence $\left\{ w_{t}\right\} $ satisfies
the conditions required by Corollary \ref{cor:md-convergence}. By
Corollary \ref{cor:md-convergence}, for any $\delta>0$, the following
events hold with probability at least $1-\delta$: $\breg\left(x^{*},x_{T+1}\right)\leq2\breg\left(x^{*},x_{1}\right)+12\left(G^{2}+\sigma^{2}\left(1+\log\left(\frac{1}{\delta}\right)\right)\right)\eta^{2}$,
and
\begin{align*}
\frac{1}{T}\sum_{t=1}^{T}\left(f\left(x_{t}\right)-f\left(x^{*}\right)\right) & \le\frac{1}{\sqrt{T}}\frac{2\breg\left(x^{*},x_{1}\right)}{\eta}+\frac{12}{\sqrt{T}}\left(G^{2}+\sigma^{2}\left(1+\log\left(\frac{1}{\delta}\right)\right)\right)\eta.
\end{align*}
In particular, setting $\eta_{t}=\sqrt{\frac{\breg\left(x^{*},x_{1}\right)}{6\left(G^{2}+\sigma^{2}\left(1+\log\left(\frac{1}{\delta}\right)\right)\right)T}}$
we obtain the first case of Theorem \ref{thm:md-convergence}.
\end{cor}
\begin{proof}[Proof of Corollary (\ref{cor:md-convergence-constant-stepsize}) ]
Recall from Corollary \ref{cor:md-convergence} that the sequence
$\left\{ w_{t}\right\} $ needs to satisfy the following conditions
for all $1\leq t\leq T$:
\begin{align*}
w_{t}+6\sigma^{2}\eta_{t}^{2} & w_{t}^{2}\leq w_{t-1}\\
w_{t}\eta_{t}^{2} & \leq\frac{1}{4\sigma^{2}}.
\end{align*}
Let $C=6\sigma^{2}\eta^{2}$. We set $w_{T}=\frac{1}{C+6\sigma^{2}\eta^{2}}=\frac{1}{2C}$.
For $1\leq t\leq T$, we set $w_{t}$ so that the first condition
holds with equality
\[
w_{t-1}=w_{t}+6\sigma^{2}w_{t}^{2}\eta_{t}^{2}=w_{t}+\frac{6}{T}\sigma^{2}\eta^{2}w_{t}^{2}.
\]
We can show by induction that, for every $1\leq t\leq T$, we have
\[
w_{t}\leq\frac{1}{C+\frac{6}{T}\sigma^{2}\eta^{2}t}.
\]
The base case $t=T$ follows from the definition of $w_{T}$. Consider
$1\leq t\le T$. Using the definition of $w_{t-1}$ and the inductive
hypothesis, we obtain 
\begin{align*}
w_{t-1} & =w_{t}+\frac{6}{T}\sigma^{2}\eta^{2}w_{t}^{2}\\
 & \leq\frac{1}{C+\frac{6}{T}\sigma^{2}\eta^{2}t}+\frac{6\sigma^{2}\eta^{2}}{T\left(C+\frac{6}{T}\sigma^{2}\eta^{2}t\right)^{2}}\\
 & \leq\frac{1}{C+\frac{6}{T}\sigma^{2}\eta^{2}t}+\frac{\left(C+\frac{6}{T}\sigma^{2}\eta^{2}t\right)-\left(C+\frac{6}{T}\sigma^{2}\eta^{2}(t-1)\right)}{\left(C+\frac{6}{T}\sigma^{2}\eta^{2}\left(t-1\right)\right)\left(C+\frac{6}{T}\sigma^{2}\eta^{2}t\right)}\\
 & =\frac{1}{C+\frac{6}{T}\sigma^{2}\eta^{2}\left(t-1\right)}
\end{align*}
as needed.

Using this fact, we now show that $\left\{ w_{t}\right\} $ satisfies
the second condition. Indeed, for every $1\leq t\leq T$, we have
\[
w_{t}\eta_{t}^{2}=w_{t}\frac{\eta^{2}}{T}\leq\frac{\eta^{2}}{6\sigma^{2}\eta^{2}t}=\frac{1}{6\sigma^{2}}.
\]

Thus, by Corollary \ref{cor:md-convergence}, with probability $\geq1-\delta$,
we have
\begin{align*}
\sum_{t=1}^{T}w_{t}\eta_{t}\left(f\left(x_{t}\right)-f\left(x^{*}\right)\right)+w_{T}\breg\left(x^{*},x_{T+1}\right) & \leq w_{0}\breg\left(x^{*},x_{1}\right)+\left(G^{2}+3\sigma^{2}\right)\sum_{t=1}^{T}w_{t}\eta_{t}^{2}+\log\left(\frac{1}{\delta}\right).
\end{align*}
Note that $w_{T}=\frac{1}{2C}$ and $\frac{1}{2C}\leq w_{t}\leq\frac{1}{C}$
for all $0\leq t\leq T$. Thus we obtain
\begin{align*}
\frac{\eta}{\sqrt{T}}\sum_{t=1}^{T}\left(f\left(x_{t}\right)-f\left(x^{*}\right)\right)+\breg\left(x^{*},x_{T+1}\right) & \leq2\breg\left(x^{*},x_{1}\right)+2\left(G^{2}+3\sigma^{2}\right)\eta^{2}+2C\log\left(\frac{1}{\delta}\right)\\
 & =2\breg\left(x^{*},x_{1}\right)+2\left(G^{2}+3\sigma^{2}\right)\eta^{2}+12\sigma^{2}\log\left(\frac{1}{\delta}\right)\eta^{2}\\
 & \le2\breg\left(x^{*},x_{1}\right)+12\left(G^{2}+\sigma^{2}\left(1+\log\left(\frac{1}{\delta}\right)\right)\right)\eta^{2}.
\end{align*}
Thus we have
\[
\frac{1}{T}\sum_{t=1}^{T}\left(f\left(x_{t}\right)-f\left(x^{*}\right)\right)\leq\frac{1}{\sqrt{T}}\left(\frac{2\breg\left(x^{*},x_{1}\right)}{\eta}+12\left(G^{2}+\sigma^{2}\left(1+\log\left(\frac{1}{\delta}\right)\right)\right)\eta\right)
\]
and
\[
\breg\left(x^{*},x_{T+1}\right)\leq2\breg\left(x^{*},x_{1}\right)+12\left(G^{2}+\sigma^{2}\left(1+\log\left(\frac{1}{\delta}\right)\right)\right)\eta^{2}.
\]
\end{proof}

The analysis extends to the setting where the $T$ is not known and
we use the step sizes $\eta_{t}=\frac{\eta}{\sqrt{t}}$.
\begin{cor}
\label{cor:md-convergence-varying-stepsize}Suppose we run the Stochastic
Mirror Descent algorithm with time-varying step sizes $\eta_{t}=\frac{\eta}{\sqrt{t}}$.
Let $w_{T}=\frac{1}{12\sigma^{2}\eta^{2}\left(\sum_{t=1}^{T}\frac{1}{t}\right)}$
and $w_{t-1}=w_{t}+6\sigma^{2}\eta_{t}^{2}w_{t}^{2}$ for all $1\leq t\leq T$.
The sequence $\left\{ w_{t}\right\} $ satisfies the conditions required
by Corollary \ref{cor:md-convergence}. By Corollary \ref{cor:md-convergence},
for any $\delta>0$, the following events hold with probability at
least $1-\delta$: $\breg\left(x^{*},x_{T+1}\right)\leq2\breg\left(x^{*},x_{1}\right)+12\left(G^{2}+\sigma^{2}\left(1+\log\left(\frac{1}{\delta}\right)\right)\right)\eta^{2}(1+\log T)$,
and
\begin{align*}
\frac{1}{T}\sum_{t=1}^{T}\left(f\left(x_{t}\right)-f\left(x^{*}\right)\right) & \le\frac{1}{\sqrt{T}}\frac{2\breg\left(x^{*},x_{1}\right)}{\eta}+\frac{12}{\sqrt{T}}\left(G^{2}+\sigma^{2}\left(1+\log\left(\frac{1}{\delta}\right)\right)\right)\eta(1+\log T).
\end{align*}
In particular, setting $\eta_{t}=\sqrt{\frac{\breg\left(x^{*},x_{1}\right)}{6\left(G^{2}+\sigma^{2}\left(1+\ln\left(\frac{1}{\delta}\right)\right)\right)t}}$
we obtain the second case of Theorem \ref{thm:md-convergence}.
\end{cor}
\begin{proof}[Proof of Corollary (\ref{cor:md-convergence-varying-stepsize})]
Recall from Corollary \ref{cor:md-convergence} that the sequence
$\left\{ w_{t}\right\} $ needs to satisfy the following conditions
for all $1\leq t\leq T$:
\begin{align*}
w_{t}+6\sigma^{2}\eta_{t}^{2} & w_{t}^{2}\leq w_{t-1}\\
w_{t}\eta_{t}^{2} & \leq\frac{1}{4\sigma^{2}}.
\end{align*}
Let $M_{t}=6\sigma^{2}\sum_{i=1}^{t}\eta_{i}^{2}$ and $C=M_{T}=6\sigma^{2}\eta^{2}\left(\sum_{t=1}^{T}\frac{1}{t}\right)$.
We set $w_{T}=\frac{1}{C+M_{T}}$. For $1\leq t\leq T$, we set $w_{t}$
so that the first condition holds with equality
\[
w_{t-1}=w_{t}+6\sigma^{2}\eta_{t}^{2}w_{t}^{2}.
\]
We can show by induction that, for every $1\leq t\leq T$, we have
\[
w_{t}\leq\frac{1}{C+M_{t}}.
\]
The base case $t=T$ follows from the definition of $w_{T}$. Consider
$1\leq t\le T$. Using the definition of $w_{t}$ and the inductive
hypothesis, we obtain 
\begin{align*}
w_{t-1} & =w_{t}+6\sigma^{2}\eta_{t}^{2}w_{t}^{2}\\
 & \leq\frac{1}{C+M_{t}}+\frac{6\sigma^{2}\eta_{t}^{2}}{\left(C+M_{t}\right)^{2}}\\
 & \le\frac{1}{C+M_{t}}+\frac{\left(C+M_{t}\right)-\left(C+M_{t-1}\right)}{\left(C+M_{t}\right)\left(C+M_{t-1}\right)}\\
 & =\frac{1}{C+M_{t-1}}
\end{align*}
as needed.

Using this fact, we now show that $\left\{ w_{t}\right\} $ satisfies
the second condition. For every $1\leq t\leq T$, we have
\[
w_{t}\eta_{t}^{2}\leq\frac{\eta_{t}^{2}}{C}\le\frac{\eta_{t}^{2}}{6\sigma^{2}\eta_{t}^{2}}=\frac{1}{6\sigma^{2}}
\]
as needed.

Thus, by Corollary \ref{cor:md-convergence}, with probability $\geq1-\delta$,
we have
\begin{align*}
\sum_{t=1}^{T}w_{t}\eta_{t}\left(f\left(x_{t}\right)-f\left(x^{*}\right)\right)+w_{T}\breg\left(x^{*},x_{T+1}\right) & \leq w_{0}\breg\left(x^{*},x_{1}\right)+\left(G^{2}+3\sigma^{2}\right)\sum_{t=1}^{T}w_{t}\eta_{t}^{2}+\log\left(\frac{1}{\delta}\right).
\end{align*}
Note that $w_{T}=\frac{1}{2C}$ and $\frac{1}{2C}\leq w_{t}\leq\frac{1}{C}$
for all $1\leq t\leq T$. Thus we obtain
\[
\frac{1}{2C}\eta_{T}\sum_{t=1}^{T}\left(f\left(x_{t}\right)-f\left(x^{*}\right)\right)+\frac{1}{2C}\breg\left(x^{*},x_{T+1}\right)\leq\frac{1}{C}\breg\left(x^{*},x_{1}\right)+\left(G^{2}+3\sigma^{2}\right)\frac{1}{C}\sum_{t=1}^{T}\eta_{t}^{2}+\log\left(\frac{1}{\delta}\right).
\]
Plugging in $\eta_{t}=\frac{\eta}{\sqrt{t}}$ and simplifying, we
obtain
\begin{align*}
\frac{\eta}{\sqrt{T}}\sum_{t=1}^{T}\left(f\left(x_{t}\right)-f\left(x^{*}\right)\right)+\breg\left(x^{*},x_{T+1}\right) & \leq2\breg\left(x^{*},x_{1}\right)+\left(2G^{2}+6\sigma^{2}\right)\eta^{2}\left(\sum_{t=1}^{T}\frac{1}{t}\right)+2C\log\left(\frac{1}{\delta}\right)\\
 & =2\breg\left(x^{*},x_{1}\right)+\left(2G^{2}+6\sigma^{2}\left(1+2\log\left(\frac{1}{\delta}\right)\right)\right)\eta^{2}\left(\sum_{t=1}^{T}\frac{1}{t}\right).
\end{align*}
Thus we have
\[
\frac{1}{T}\sum_{t=1}^{T}\left(f\left(x_{t}\right)-f\left(x^{*}\right)\right)\leq\frac{1}{\sqrt{T}}\left(\frac{2\breg\left(x^{*},x_{1}\right)}{\eta}+\left(2G^{2}+6\sigma^{2}\left(1+2\log\left(\frac{1}{\delta}\right)\right)\right)\eta\left(\sum_{t=1}^{T}\frac{1}{t}\right)\right)
\]
and
\[
\breg\left(x^{*},x_{T+1}\right)\leq2\breg\left(x^{*},x_{1}\right)+\left(2G^{2}+6\sigma^{2}\left(1+2\log\left(\frac{1}{\delta}\right)\right)\right)\eta^{2}\left(\sum_{t=1}^{T}\frac{1}{t}\right).
\]
\end{proof}

\subsection{Accelerated Stochastic Mirror Descent}

The convergence of Algorithm \ref{alg:acc-md} is given in the following
Theorem.
\begin{thm}
\label{thm:acc-md-convergence}Assume $f$ satisfies Assumptions (1),
(2), (3) and condition (\ref{eq:acc-md-condition}), with probability
at least $1-\delta$, 

(1) Setting $\eta_{t}=\min\left\{ \frac{t}{4L},\frac{\sqrt{\breg\left(x^{*},z_{0}\right)}t}{\sqrt{6}\sqrt{G^{2}+\sigma^{2}\left(1+\log\left(\frac{1}{\delta}\right)\right)}T^{3/2}}\right\} $,
then $\breg\left(x^{*},z_{T}\right)\leq4\breg\left(x^{*},z_{0}\right)$
and
\begin{align*}
f\left(y_{T}\right)-f\left(x^{*}\right) & \leq\frac{16L\breg\left(x^{*},z_{0}\right)}{T^{2}}+\frac{8\sqrt{6}}{\sqrt{T}}\sqrt{\breg\left(x^{*},z_{0}\right)\left(G^{2}+\left(1+\log\left(\frac{1}{\delta}\right)\right)\sigma^{2}\right)}.
\end{align*}

(2) Setting $\eta_{t}=\min\left\{ \frac{t}{4L},\frac{\sqrt{\breg\left(x^{*},z_{0}\right)}}{\sqrt{6}\sqrt{G^{2}+\sigma^{2}\left(1+\log\left(\frac{1}{\delta}\right)\right)}t^{1/2}}\right\} $,
then $\breg\left(x^{*},z_{T}\right)\leq2(2+\log T)\breg\left(x^{*},z_{0}\right)$
and
\begin{align*}
f\left(y_{T}\right)-f\left(x^{*}\right) & \leq\frac{16L\breg\left(x^{*},z_{0}\right)}{T^{2}}+\frac{4\sqrt{6}(2+\log T)}{\sqrt{T}}\sqrt{\breg\left(x^{*},z_{0}\right)\left(G^{2}+\left(1+\log\left(\frac{1}{\delta}\right)\right)\sigma^{2}\right)}.
\end{align*}

\end{thm}
\begin{proof}[Proof of Lemma \ref{lem:acc-md-basic-analysis}]
 Starting with smoothness, we obtain
\begin{align*}
f\left(y_{t}\right) & \le f\left(x_{t}\right)+\left\langle \nabla f\left(x_{t}\right),y_{t}-x_{t}\right\rangle +G\left\Vert y_{t}-x_{t}\right\Vert +\frac{\beta}{2}\left\Vert y_{t}-x_{t}\right\Vert ^{2}\ \forall x\in\dom\\
 & =f\left(x_{t}\right)+\left\langle \nabla f\left(x_{t}\right),y_{t-1}-x_{t}\right\rangle +\left\langle \nabla f\left(x_{t}\right),y_{t}-y_{t-1}\right\rangle +G\left\Vert y_{t}-x_{t}\right\Vert +\frac{\beta}{2}\left\Vert y_{t}-x_{t}\right\Vert ^{2}\\
 & =\left(1-\alpha_{t}\right)\underbrace{\left(f\left(x_{t}\right)+\left\langle \nabla f\left(x_{t}\right),y_{t-1}-x_{t}\right\rangle \right)}_{\text{convexity}}+\alpha_{t}\underbrace{\left(f\left(x_{t}\right)+\left\langle \nabla f\left(x_{t}\right),y_{t-1}-x_{t}\right\rangle \right)}_{\text{convexity}}\\
 & +\alpha_{t}\left\langle \nabla f\left(x_{t}\right),z_{t}-y_{t-1}\right\rangle +G\left\Vert y_{t}-x_{t}\right\Vert +\frac{\beta}{2}\left\Vert y_{t}-x_{t}\right\Vert ^{2}\\
 & \le\left(1-\alpha_{t}\right)f\left(y_{t-1}\right)+\alpha_{t}f\left(x_{t}\right)+\alpha_{t}\left\langle \nabla f\left(x_{t}\right),z_{t}-x_{t}\right\rangle +G\underbrace{\left\Vert y_{t}-x_{t}\right\Vert }_{=\alpha_{t}\left\Vert z_{t}-z_{t-1}\right\Vert }+\frac{\beta}{2}\underbrace{\left\Vert y_{t}-x_{t}\right\Vert ^{2}}_{=\alpha_{t}^{2}\left\Vert z_{t}-z_{t-1}\right\Vert ^{2}}\\
 & =\left(1-\alpha_{t}\right)f\left(y_{t-1}\right)+\alpha_{t}f\left(x_{t}\right)+\alpha_{t}\left\langle \nabla f\left(x_{t}\right),z_{t}-x_{t}\right\rangle +G\alpha_{t}\left\Vert z_{t}-z_{t-1}\right\Vert +\frac{\beta}{2}\alpha_{t}^{2}\left\Vert z_{t}-z_{t-1}\right\Vert ^{2}.
\end{align*}
 By the optimality condition for $z_{t}$,
\[
\eta_{t}\left\langle \widehat{\nabla}f(x_{t}),z_{t}-x^{*}\right\rangle \leq\left\langle \nabla_{x}\breg\left(z_{t},z_{t-1}\right),x^{*}-z_{t}\right\rangle =\breg\left(x^{*},z_{t-1}\right)-\breg\left(z_{t},z_{t-1}\right)-\breg\left(x^{*},z_{t}\right).
\]
Rearranging, we obtain
\begin{align*}
\breg\left(x^{*},z_{t}\right)-\breg\left(x^{*},z_{t-1}\right)+\breg\left(z_{t},z_{t-1}\right) & \leq\eta_{t}\left\langle \widehat{\nabla}f\left(x_{t}\right),x^{*}-z_{t}\right\rangle =\eta_{t}\left\langle \nabla f\left(x_{t}\right)+\xi_{t},x^{*}-z_{t}\right\rangle .
\end{align*}
By combining the two inequalities, we obtain
\begin{align*}
 & f\left(y_{t}\right)+\frac{\alpha_{t}}{\eta_{t}}\left(\breg\left(x^{*},z_{t}\right)-\breg\left(x^{*},z_{t-1}\right)+\breg\left(z_{t},z_{t-1}\right)\right)\\
 & \leq\left(1-\alpha_{t}\right)f\left(y_{t-1}\right)+\alpha_{t}\underbrace{\left(f\left(x_{t}\right)+\left\langle \nabla f\left(x_{t}\right),x^{*}-x_{t}\right\rangle \right)}_{\text{convexity}}\\
 & +G\alpha_{t}\left\Vert z_{t}-z_{t-1}\right\Vert +\frac{\beta}{2}\alpha_{t}^{2}\left\Vert z_{t}-z_{t-1}\right\Vert ^{2}+\alpha_{t}\left\langle \xi_{t},x^{*}-z_{t}\right\rangle \\
 & \leq\left(1-\alpha_{t}\right)f\left(y_{t-1}\right)+\alpha_{t}f\left(x^{*}\right)+G\alpha_{t}\left\Vert z_{t}-z_{t-1}\right\Vert +\frac{\beta}{2}\alpha_{t}^{2}\left\Vert z_{t}-z_{t-1}\right\Vert ^{2}+\alpha_{t}\left\langle \xi_{t},x^{*}-z_{t}\right\rangle .
\end{align*}
Subtracting $f\left(x^{*}\right)$ from both sides, rearranging, and
using that $\breg\left(z_{t},z_{t-1}\right)\geq\frac{1}{2}\left\Vert z_{t}-z_{t-1}\right\Vert ^{2}$,
we obtain
\begin{align*}
 & f\left(y_{t}\right)-f\left(x^{*}\right)+\frac{\alpha_{t}}{\eta_{t}}\left(\breg\left(x^{*},z_{t}\right)-\breg\left(x^{*},z_{t-1}\right)\right)\\
 & \leq\left(1-\alpha_{t}\right)\left(f\left(y_{t-1}\right)-f\left(x^{*}\right)\right)+\alpha_{t}\left\langle \xi_{t},x^{*}-z_{t}\right\rangle +G\alpha_{t}\left\Vert z_{t}-z_{t-1}\right\Vert -\alpha_{t}\frac{1-\beta\alpha_{t}\eta_{t}}{2\eta_{t}}\left\Vert z_{t}-z_{t-1}\right\Vert ^{2}\\
 & =\left(1-\alpha_{t}\right)\left(f\left(y_{t-1}\right)-f\left(x^{*}\right)\right)+\alpha_{t}\left\langle \xi_{t},x^{*}-z_{t-1}\right\rangle +\alpha_{t}\left\langle \xi_{t},z_{t}-z_{t-1}\right\rangle +G\alpha_{t}\left\Vert z_{t}-z_{t-1}\right\Vert -\alpha_{t}\frac{1-\beta\alpha_{t}\eta_{t}}{2\eta_{t}}\left\Vert z_{t}-z_{t-1}\right\Vert ^{2}\\
 & \le\left(1-\alpha_{t}\right)\left(f\left(y_{t-1}\right)-f\left(x^{*}\right)\right)+\alpha_{t}\left\langle \xi_{t},x^{*}-z_{t-1}\right\rangle +\alpha_{t}\left\Vert z_{t}-z_{t-1}\right\Vert \left(\left\Vert \xi_{t}\right\Vert _{*}+G\right)-\alpha_{t}\frac{1-\beta\alpha_{t}\eta_{t}}{2\eta_{t}}\left\Vert z_{t}-z_{t-1}\right\Vert ^{2}\\
 & \leq\left(1-\alpha_{t}\right)\left(f\left(y_{t-1}\right)-f\left(x^{*}\right)\right)+\alpha_{t}\left\langle \xi_{t},x^{*}-z_{t-1}\right\rangle +\frac{\alpha_{t}\eta_{t}}{2\left(1-\beta\alpha_{t}\eta_{t}\right)}\left(\left\Vert \xi_{t}\right\Vert _{*}+G\right)^{2}.
\end{align*}
Finally, we divide by $\frac{\alpha_{t}}{\eta_{t}}$, and obtain
\begin{align*}
 & \frac{\eta_{t}}{\alpha_{t}}\left(f\left(y_{t}\right)-f\left(x^{*}\right)\right)+\breg\left(x^{*},z_{t}\right)-\breg\left(x^{*},z_{t-1}\right)\\
 & \leq\frac{\eta_{t}}{\alpha_{t}}\left(1-\alpha_{t}\right)\left(f\left(y_{t-1}\right)-f\left(x^{*}\right)\right)+\eta_{t}\left\langle \xi_{t},x^{*}-z_{t-1}\right\rangle +\frac{\eta_{t}^{2}}{2\left(1-\beta\alpha_{t}\eta_{t}\right)}\left(\left\Vert \xi_{t}\right\Vert _{*}+G\right)^{2}\\
 & \leq\frac{\eta_{t}}{\alpha_{t}}\left(1-\alpha_{t}\right)\left(f\left(y_{t-1}\right)-f\left(x^{*}\right)\right)+\eta_{t}\left\langle \xi_{t},x^{*}-z_{t-1}\right\rangle +\frac{\eta_{t}^{2}}{1-\beta\alpha_{t}\eta_{t}}\left(\left\Vert \xi_{t}\right\Vert _{*}^{2}+G^{2}\right).
\end{align*}
\end{proof}

\begin{proof}[Proof of Theorem \ref{thm:acc-md-concentration-subgaussian}]
We proceed by induction on $t$. Consider the base case $t=T+1$,
the inequality trivially holds. Next, we consider $t\leq T$. We have
\begin{align}
\E\left[\exp\left(S_{t}\right)\mid\F_{t}\right] & =\E\left[\exp\left(Z_{t}+S_{t+1}\right)\mid\F_{t}\right]=\E\left[\E\left[\exp\left(Z_{t}+S_{t+1}\right)\mid\F_{t+1}\right]\mid\F_{t}\right].\label{eq:1-1}
\end{align}
We now analyze the inner expectation. Conditioned on $\F_{t+1}$,
$Z_{t}$ is fixed. Using the inductive hypothesis, we obtain
\begin{align}
\E\left[\exp\left(Z_{t}+S_{t+1}\right)\mid\F_{t+1}\right] & \leq\exp\left(Z_{t}\right)\exp\left(3\sigma^{2}\sum_{i=t+1}^{T}w_{i}\frac{\eta_{i}^{2}}{1-L\alpha_{i}\eta_{i}}\right).\label{eq:2-1}
\end{align}
Let $X_{t}=\eta_{t}\left\langle \xi_{t},x^{*}-z_{t-1}\right\rangle $.
By Lemma \ref{lem:acc-md-basic-analysis}, we have
\begin{align*}
 & \frac{\eta_{t}}{\alpha_{t}}\left(f\left(y_{t}\right)-f\left(x^{*}\right)\right)-\frac{\eta_{t}}{\alpha_{t}}\left(1-\alpha_{t}\right)\left(f\left(y_{t-1}\right)-f\left(x^{*}\right)\right)-\frac{\eta_{t}^{2}}{1-L\alpha_{t}\eta_{t}}G^{2}\\
 & \quad+\breg\left(x^{*},z_{t}\right)-\breg\left(x^{*},z_{t-1}\right)\\
 & \leq X_{t}+\frac{\eta_{t}^{2}}{\left(1-L\alpha_{t}\eta_{t}\right)}\left\Vert \xi_{t}\right\Vert _{*}^{2}
\end{align*}
and thus
\begin{align*}
Z_{t} & \leq w_{t}X_{t}+w_{t}\frac{\eta_{t}^{2}}{1-L\alpha_{t}\eta_{t}}\left\Vert \xi_{t}\right\Vert _{*}^{2}-v_{t}\breg\left(x^{*},z_{t-1}\right).
\end{align*}
Plugging into (\ref{eq:2-1}), we obtain
\begin{align*}
 & \E\left[\exp\left(Z_{t}+S_{t+1}\right)\mid\F_{t+1}\right]\\
 & \leq\exp\left(w_{t}X_{t}-v_{t}\breg\left(x^{*},z_{t-1}\right)+w_{t}\frac{\eta_{t}^{2}}{1-L\alpha_{t}\eta_{t}}\left\Vert \xi_{t}\right\Vert _{*}^{2}+3\sigma^{2}\sum_{i=t+1}^{T}w_{i}\frac{\eta_{i}^{2}}{1-L\alpha_{i}\eta_{i}}\right).
\end{align*}
Plugging into (\ref{eq:1-1}), we obtain
\begin{align}
 & \E\left[\exp\left(S_{t}\right)\mid\F_{t}\right]\nonumber \\
 & \leq\exp\left(-v_{t}\breg\left(x^{*},z_{t-1}\right)+3\sigma^{2}\sum_{i=t+1}^{T}w_{i}\frac{\eta_{i}^{2}}{1-L\alpha_{i}\eta_{i}}\right)\E\left[\exp\left(w_{t}X_{t}+w_{t}\frac{\eta_{t}^{2}}{1-L\alpha_{t}\eta_{t}}\left\Vert \xi_{t}\right\Vert _{*}^{2}\right)\mid\F_{t}\right].\label{eq:3-1}
\end{align}
Next, we analyze the expectation on the RHS of the above inequality.
Note that $X_{t}=\eta_{t}\left\langle \xi_{t},x^{*}-z_{t-1}\right\rangle $
and $\E\left[X_{t}\mid\F_{t}\right]=0$. Applying Lemma \ref{lem:helper-taylor-vector},
we obtain
\begin{align}
 & \E\left[\exp\left(w_{t}X_{t}+w_{t}\frac{\eta_{t}^{2}}{1-L\alpha_{t}\eta_{t}}\left\Vert \xi_{t}\right\Vert _{*}^{2}\right)\mid\F_{t}\right]\nonumber \\
 & \leq\exp\left(3\left(w_{t}^{2}\eta_{t}^{2}\left\Vert x^{*}-z_{t-1}\right\Vert ^{2}+w_{t}\frac{\eta_{t}^{2}}{1-L\alpha_{t}\eta_{t}}\right)\sigma^{2}\right)\nonumber \\
 & \leq\exp\left(3\left(2w_{t}^{2}\eta_{t}^{2}\breg\left(x^{*},z_{t-1}\right)+w_{t}\frac{\eta_{t}^{2}}{1-L\alpha_{t}\eta_{t}}\right)\sigma^{2}\right).\label{eq:4-1}
\end{align}
On the last line we used that $\breg\left(x^{*},z_{t-1}\right)\geq\frac{1}{2}\left\Vert x^{*}-z_{t-1}\right\Vert ^{2}$,
which follows from the strong convexity of $\psi$.

Plugging in (\ref{eq:4-1}) into (\ref{eq:3-1}) and using that $v_{t}=6\sigma^{2}w_{t}^{2}\eta_{t}^{2}$,
we obtain
\[
\E\left[\exp\left(S_{t}\right)\mid\F_{t}\right]\leq\exp\left(3\sigma^{2}\sum_{i=t}^{T}w_{i}\frac{\eta_{i}^{2}}{1-L\alpha_{i}\eta_{i}}\right)
\]
as needed.
\end{proof}

\begin{proof}[Proof of Corollary \ref{cor:acc-md-convergence}]
Let
\[
K=3\sigma^{2}\sum_{t=1}^{T}w_{t}\frac{\eta_{t}^{2}}{1-L\alpha_{t}\eta_{t}}+\log\left(\frac{1}{\delta}\right).
\]
By Theorem \ref{thm:acc-md-concentration-subgaussian} and Markov's
inequality, we have
\begin{align*}
\Pr\left[S_{1}\geq K\right] & \leq\Pr\left[\exp\left(S_{1}\right)\geq\exp\left(K\right)\right]\\
 & \leq\exp\left(-K\right)\E\left[\exp\left(S_{1}\right)\right]\\
 & \leq\exp\left(-K\right)\exp\left(3\sigma^{2}\sum_{t=1}^{T}w_{t}\frac{\eta_{t}^{2}}{1-L\alpha_{t}\eta_{t}}\right)\\
 & =\delta.
\end{align*}
Note that since $v_{t}+w_{t}\le w_{t-1}$
\begin{align*}
S_{1} & =\sum_{t=1}^{T}Z_{t}\\
 & =\sum_{t=1}^{T}w_{t}\left(\frac{\eta_{t}}{\alpha_{t}}\left(f\left(y_{t}\right)-f\left(x^{*}\right)\right)-\frac{\eta_{t}\left(1-\alpha_{t}\right)}{\alpha_{t}}\left(f\left(y_{t-1}\right)-f\left(x^{*}\right)\right)\right)\\
 & \quad+\sum_{t=1}^{T}w_{t}\breg\left(x^{*},z_{t}\right)-(v_{t}+w_{t})\breg\left(x^{*},z_{t-1}\right)-G^{2}\sum_{t=1}^{T}w_{t}\frac{\eta_{t}^{2}}{1-L\alpha_{t}\eta_{t}}\\
\ge & \sum_{t=1}^{T}w_{t}\left(\frac{\eta_{t}}{\alpha_{t}}\left(f\left(y_{t}\right)-f\left(x^{*}\right)\right)-\frac{\eta_{t}\left(1-\alpha_{t}\right)}{\alpha_{t}}\left(f\left(y_{t-1}\right)-f\left(x^{*}\right)\right)\right)\\
 & \quad+\sum_{t=1}^{T}w_{t}\breg\left(x^{*},z_{t}\right)-w_{t-1}\breg\left(x^{*},z_{t-1}\right)-G^{2}\sum_{t=1}^{T}w_{t}\frac{\eta_{t}^{2}}{1-L\alpha_{t}\eta_{t}}\\
= & \sum_{t=1}^{T}w_{t}\left(\frac{\eta_{t}}{\alpha_{t}}\left(f\left(y_{t}\right)-f\left(x^{*}\right)\right)-\frac{\eta_{t}\left(1-\alpha_{t}\right)}{\alpha_{t}}\left(f\left(y_{t-1}\right)-f\left(x^{*}\right)\right)\right)\\
 & \quad+w_{T}\breg\left(x^{*},z_{T}\right)-w_{0}\breg\left(x^{*},z_{0}\right)-G^{2}\sum_{t=1}^{T}w_{t}\frac{\eta_{t}^{2}}{1-L\alpha_{t}\eta_{t}}.
\end{align*}
Therefore, with probability at least $1-\delta$, we have
\begin{align*}
 & \sum_{t=1}^{T}w_{t}\left(\frac{\eta_{t}}{\alpha_{t}}\left(f\left(y_{t}\right)-f\left(x^{*}\right)\right)-\frac{\eta_{t}\left(1-\alpha_{t}\right)}{\alpha_{t}}\left(f\left(y_{t-1}\right)-f\left(x^{*}\right)\right)\right)+w_{T}\breg\left(x^{*},z_{T}\right)\\
 & \leq w_{0}\breg\left(x^{*},z_{0}\right)+\left(G^{2}+3\sigma^{2}\right)\sum_{t=1}^{T}w_{t}\frac{\eta_{t}^{2}}{1-L\alpha_{t}\eta_{t}}+\log\left(\frac{1}{\delta}\right).
\end{align*}
\end{proof}

\begin{cor}
\label{cor:acc-md-convergence-final}Suppose we run the Accelerated
Stochastic Mirror Descent algorithm with the standard choices $\alpha_{t}=\frac{2}{t+1}$
and $\eta_{t}=\eta t$ with $\eta\leq\frac{1}{4L}$. Let $w_{T}=\frac{1}{3\sigma^{2}\eta^{2}T\left(T+1\right)\left(2T+1\right)}$
and $w_{t-1}=w_{t}+6\sigma^{2}\eta_{t}^{2}w_{t}^{2}$ for all $1\leq t\leq T$.
The sequence $\left\{ w_{t}\right\} _{0\leq t\leq T}$ satisfies the
conditions required by Corollary \ref{cor:acc-md-convergence}. By
Corollary \ref{cor:acc-md-convergence}, with probability at least
$1-\delta$, $\breg\left(x^{*},z_{T}\right)\leq2\breg\left(x^{*},z_{0}\right)+12\left(G^{2}+\left(1+\log\left(\frac{1}{\delta}\right)\right)\sigma^{2}\right)\eta^{2}T^{3}$
and
\begin{align*}
f\left(y_{T}\right)-f\left(x^{*}\right)\le & \frac{4\breg\left(x^{*},z_{0}\right)}{\eta T^{2}}+24\left(G^{2}+\left(1+\log\left(\frac{1}{\delta}\right)\right)\sigma^{2}\right)\eta T
\end{align*}
In particular, setting $\eta=\min\left\{ \frac{1}{4L},\frac{\sqrt{\breg\left(x^{*},z_{0}\right)}}{\sqrt{6}\sqrt{G^{2}+\sigma^{2}\left(1+\log\left(\frac{1}{\delta}\right)\right)}T^{3/2}}\right\} $,
we obtain the first case of Theorem \ref{thm:acc-md-convergence}.
\end{cor}
\begin{proof}[Proof of Corollary \ref{cor:acc-md-convergence-final}]
Recall from Corollary \ref{cor:acc-md-convergence} that the sequence
$\left\{ w_{t}\right\} $ needs to satisfy the following conditions:
\begin{align}
w_{t}+6\sigma^{2}\eta_{t}^{2}w_{t}^{2} & \leq w_{t-1}\quad\forall1\leq t\leq T\label{eq:C1}\\
\frac{w_{t}\eta_{t}^{2}}{1-L\alpha_{t}\eta_{t}} & \leq\frac{1}{4\sigma^{2}}\quad\forall0\leq t\leq T\label{eq:C2}
\end{align}
We will set $\left\{ w_{t}\right\} $ so that it satisfies the following
additional condition, which will allow us to telescope the sum on
the RHS of Corollary \ref{cor:acc-md-convergence}:
\begin{equation}
w_{t-1}\frac{\eta_{t-1}}{\alpha_{t-1}}\geq w_{t}\frac{\eta_{t}\left(1-\alpha_{t}\right)}{\alpha_{t}}\quad\forall1\leq t\leq T.\label{eq:C3}
\end{equation}
Given $w_{T}$, we set $w_{t-1}$ for every $1\leq t\leq T$ so that
the first condition (\ref{eq:C1}) holds with equality:
\[
w_{t-1}=w_{t}+6\sigma^{2}\eta_{t}^{2}w_{t}^{2}=w_{t}+6\sigma^{2}\eta^{2}t^{2}w_{t}^{2}.
\]
Let $C=\sigma^{2}\eta^{2}T\left(T+1\right)\left(2T+1\right)$. We
set 
\[
w_{T}=\frac{1}{C+6\sigma^{2}\eta^{2}\sum_{i=1}^{T}i^{2}}=\frac{1}{C+\sigma^{2}\eta^{2}T\left(T+1\right)\left(2T+1\right)}=\frac{1}{2\sigma^{2}\eta^{2}T\left(T+1\right)\left(2T+1\right)}.
\]
Given this choice for $w_{T}$, we now verify that, for all $0\leq t\leq T$,
we have
\[
w_{t}\leq\frac{1}{C+6\sigma^{2}\eta^{2}\sum_{i=1}^{t}i^{2}}=\frac{1}{C+\sigma^{2}\eta^{2}t\left(t+1\right)\left(2t+1\right)}.
\]
We proceed by induction on $t$. The base case $t=T$ follows from
the definition of $w_{T}$. Consider $t\le T$. Using the definition
of $w_{t-1}$ and the inductive hypothesis, we obtain
\begin{align*}
w_{t-1} & =w_{t}+6\sigma^{2}\eta^{2}t^{2}w_{t}^{2}\\
 & \leq\frac{1}{C+6\sigma^{2}\eta^{2}\sum_{i=1}^{t}i^{2}}+\frac{6\sigma^{2}\eta^{2}t^{2}}{\left(C+6\sigma^{2}\eta^{2}\sum_{i=1}^{t}i^{2}\right)^{2}}\\
 & \leq\frac{1}{C+6\sigma^{2}\eta^{2}\sum_{i=1}^{t}i^{2}}+\frac{\left(C+6\sigma^{2}\eta^{2}\sum_{i=1}^{t}i^{2}\right)-\left(C+6\sigma^{2}\eta^{2}\sum_{i=1}^{t-1}i^{2}\right)}{\left(C+6\sigma^{2}\eta^{2}\sum_{i=1}^{t}i^{2}\right)\left(C+6\sigma^{2}\eta^{2}\sum_{i=1}^{t-1}i^{2}\right)}\\
 & =\frac{1}{C+6\sigma^{2}\eta^{2}\sum_{i=1}^{t-1}i^{2}}
\end{align*}
as needed.

Let us now verify that the second condition (\ref{eq:C2}) also holds.
Using that $\frac{2t}{t+1}\leq2$, $L\eta\leq\frac{1}{4}$, and $T\geq2$,
we obtain
\begin{align*}
\frac{w_{t}\eta_{t}^{2}}{1-L\alpha_{t}\eta_{t}} & =\frac{w_{t}\eta^{2}t^{2}}{1-L\eta\frac{2t}{t+1}}\leq2w_{t}\eta^{2}t^{2}\leq\frac{2\eta^{2}t^{2}}{C+6\sigma^{2}\eta^{2}t^{2}}\\
 & =\frac{t^{2}}{\sigma^{2}T\left(T+1\right)\left(2T+1\right)+3\sigma^{2}t^{2}}\\
 & \leq\frac{1}{\sigma^{2}\left(2T+1\right)+3\sigma^{2}}\leq\frac{1}{4\sigma^{2}}
\end{align*}
as needed.

Let us now verify that the third condition (\ref{eq:C3}) also holds.
Since $\eta_{t}=\eta t$ and $\alpha_{t}=\frac{2}{t+1}$, we have
$\frac{\eta_{t-1}}{\alpha_{t-1}}=\frac{\eta_{t}\left(1-\alpha_{t}\right)}{\alpha_{t}}=\frac{\eta t\left(t-1\right)}{2}$.
Since $w_{t}\leq w_{t-1}$, it follows that condition (\ref{eq:C3})
holds.

We now turn our attention to the convergence. By Corollary \ref{cor:acc-md-convergence},
with probability $\geq1-\delta$, we have
\begin{align*}
 & \sum_{t=1}^{T}w_{t}\left(\frac{\eta_{t}}{\alpha_{t}}\left(f\left(y_{t}\right)-f\left(x^{*}\right)\right)-\frac{\eta_{t}\left(1-\alpha_{t}\right)}{\alpha_{t}}\left(f\left(y_{t-1}\right)-f\left(x^{*}\right)\right)\right)+w_{T}\breg\left(x^{*},z_{T}\right)\\
 & \leq w_{0}\breg\left(x^{*},z_{0}\right)+\left(G^{2}+3\sigma^{2}\right)\sum_{t=1}^{T}w_{t}\frac{\eta_{t}^{2}}{1-L\alpha_{t}\eta_{t}}+\log\left(\frac{1}{\delta}\right).
\end{align*}
Grouping terms on the LHS and using that $\alpha_{1}=1$, we obtain
\begin{align*}
 & \sum_{t=1}^{T-1}\left(w_{t}\frac{\eta_{t}}{\alpha_{t}}-w_{t+1}\frac{\eta_{t+1}\left(1-\alpha_{t+1}\right)}{\alpha_{t+1}}\right)\left(f\left(y_{t}\right)-f\left(x^{*}\right)\right)+w_{T}\frac{\eta_{T}}{\alpha_{T}}\left(f\left(y_{T}\right)-f\left(x^{*}\right)\right)+w_{T}\breg\left(x^{*},z_{T}\right)\\
 & \leq w_{0}\breg\left(x^{*},z_{0}\right)+\left(G^{2}+3\sigma^{2}\right)\sum_{t=1}^{T}w_{t}\frac{\eta_{t}^{2}}{1-L\alpha_{t}\eta_{t}}+\log\left(\frac{1}{\delta}\right).
\end{align*}
Since $\left\{ w_{t}\right\} $ satisfies condition (\ref{eq:C3}),
the coefficient of $f\left(y_{t}\right)-f\left(x^{*}\right)$ is non-negative
and thus we can drop the above sum. We obtain
\begin{align*}
w_{T}\frac{\eta_{T}}{\alpha_{T}}\left(f\left(y_{T}\right)-f\left(x^{*}\right)\right)+w_{T}\breg\left(x^{*},z_{T}\right) & \leq w_{0}\breg\left(x^{*},z_{0}\right)+\left(G^{2}+3\sigma^{2}\right)\sum_{t=1}^{T}w_{t}\frac{\eta_{t}^{2}}{1-L\alpha_{t}\eta_{t}}+\log\left(\frac{1}{\delta}\right).
\end{align*}
 Using that $w_{T}=\frac{1}{2C}$ and $w_{t}\leq\frac{1}{C}$ for
all $0\leq t\leq T-1$, we obtain
\begin{align*}
 & \frac{1}{2C}\frac{\eta_{T}}{\alpha_{T}}\left(f\left(y_{T}\right)-f\left(x^{*}\right)\right)+\frac{1}{2C}\breg\left(x^{*},z_{T}\right)\\
 & \leq\frac{1}{C}\breg\left(x^{*},z_{0}\right)+\frac{1}{C}\left(G^{2}+3\sigma^{2}\right)\sum_{t=1}^{T}\frac{\eta_{t}^{2}}{1-L\alpha_{t}\eta_{t}}+\log\left(\frac{1}{\delta}\right).
\end{align*}
Thus
\begin{align*}
 & \frac{\eta_{T}}{\alpha_{T}}\left(f\left(y_{T}\right)-f\left(x^{*}\right)\right)+\breg\left(x^{*},z_{T}\right)\\
 & \leq2\breg\left(x^{*},z_{0}\right)+2\left(G^{2}+3\sigma^{2}\right)\sum_{t=1}^{T}\frac{\eta_{t}^{2}}{1-L\alpha_{t}\eta_{t}}+2C\log\left(\frac{1}{\delta}\right)\\
 & =2\breg\left(x^{*},z_{0}\right)+2\left(G^{2}+3\sigma^{2}\right)\sum_{t=1}^{T}\frac{\eta_{t}^{2}}{1-L\alpha_{t}\eta_{t}}+2\sigma^{2}\log\left(\frac{1}{\delta}\right)\eta^{2}T\left(T+1\right)\left(2T+1\right).
\end{align*}
Using that $L\eta\leq\frac{1}{4}$ and $\frac{2t}{t+1}\leq2$, we
obtain
\[
\sum_{t=1}^{T}\frac{\eta_{t}^{2}}{1-L\alpha_{t}\eta_{t}}=\sum_{t=1}^{T}\frac{\eta^{2}t^{2}}{1-L\eta\frac{2t}{t+1}}\leq\sum_{t=1}^{T}2\eta^{2}t^{2}=\frac{1}{3}\eta^{2}T\left(T+1\right)\left(2T+1\right).
\]
Plugging in and using that $\eta_{T}=\eta T$ and $\alpha_{T}=\frac{2}{T+1}$,
we obtain
\begin{align*}
 & \eta\frac{T\left(T+1\right)}{2}\left(f\left(y_{T}\right)-f\left(x^{*}\right)\right)+\breg\left(x^{*},z_{T}\right)\\
 & \leq2\breg\left(x^{*},z_{0}\right)+\left(\frac{2}{3}G^{2}+2\left(1+\log\left(\frac{1}{\delta}\right)\right)\sigma^{2}\right)\eta^{2}T\left(T+1\right)\left(2T+1\right)\\
 & \leq2\breg\left(x^{*},z_{0}\right)+2\left(G^{2}+\left(1+\log\left(\frac{1}{\delta}\right)\right)\sigma^{2}\right)\eta^{2}T\left(T+1\right)\left(2T+1\right).
\end{align*}
We can further simplify the bound by lower bounding $T\left(T+1\right)\geq T^{2}$
and upper bounding $T\left(T+1\right)\left(2T+1\right)\leq6T^{3}$.
We obtain
\begin{align*}
\eta T^{2}\left(f\left(y_{T}\right)-f\left(x^{*}\right)\right)+2\breg\left(x^{*},z_{T}\right) & \leq4\breg\left(x^{*},z_{0}\right)+24\left(G^{2}+\left(1+\log\left(\frac{1}{\delta}\right)\right)\sigma^{2}\right)\eta^{2}T^{3}.
\end{align*}
Thus we obtain
\[
f\left(y_{T}\right)-f\left(x^{*}\right)\leq\frac{4\breg\left(x^{*},z_{0}\right)}{\eta T^{2}}+24\left(G^{2}+\left(1+\log\left(\frac{1}{\delta}\right)\right)\sigma^{2}\right)\eta T
\]
and
\begin{align*}
\breg\left(x^{*},z_{T}\right) & \leq2\breg\left(x^{*},z_{0}\right)+12\left(G^{2}+\left(1+\log\left(\frac{1}{\delta}\right)\right)\sigma^{2}\right)\eta^{2}T^{3}.
\end{align*}
\end{proof}

\begin{cor}
\label{cor:acc-md-convergence-final-vary}Suppose we run the Accelerated
Stochastic Mirror Descent algorithm with the standard choices $\alpha_{t}=\frac{2}{t+1}$
and $\eta_{t}=\min\left\{ \frac{t}{4L},\frac{\eta}{\sqrt{t}}\right\} $.
Let $w_{T}=\frac{1}{12\sigma^{2}\sum_{i=1}^{T}\eta_{t}^{2}}$ and
$w_{t-1}=w_{t}+6\sigma^{2}\eta_{t}^{2}w_{t}^{2}$ for all $1\leq t\leq T$.
The sequence $\left\{ w_{t}\right\} _{0\leq t\leq T}$ satisfies the
conditions required by Corollary \ref{cor:acc-md-convergence}. By
Corollary \ref{cor:acc-md-convergence}, with probability at least
$1-\delta$, $\breg\left(x^{*},z_{T}\right)\leq2\breg\left(x^{*},z_{0}\right)+12\left(G^{2}+\left(1+\log\left(\frac{1}{\delta}\right)\right)\sigma^{2}\right)\eta^{2}(1+\log T)$
and
\begin{align*}
 & f\left(y_{T}\right)-f\left(x^{*}\right)\le\frac{16L}{T^{2}}\breg\left(x^{*},z_{0}\right)+\frac{2}{T^{1/2}\eta}\left(2\breg\left(x^{*},z_{0}\right)+12\left(G^{2}+\left(1+\log\left(\frac{1}{\delta}\right)\right)\sigma^{2}\right)\eta^{2}(1+\log T)\right)
\end{align*}
In particular, setting $\eta_{t}=\min\left\{ \frac{t}{4L},\frac{\sqrt{\breg\left(x^{*},z_{0}\right)}}{\sqrt{6}\sqrt{G^{2}+\sigma^{2}\left(1+\log\left(\frac{1}{\delta}\right)\right)}t^{1/2}}\right\} $,
we obtain the second case of Theorem \ref{thm:acc-md-convergence}.
\end{cor}
\begin{proof}[Proof of Corollary \ref{cor:acc-md-convergence-final-vary}]
Recall from Corollary \ref{cor:acc-md-convergence} that the sequence
$\left\{ w_{t}\right\} $ needs to satisfy the following conditions:
\begin{align}
w_{t}+6\sigma^{2}\eta_{t}^{2}w_{t}^{2} & \leq w_{t-1}\quad\forall1\leq t\leq T,\label{eq:C1-1}\\
\frac{w_{t}\eta_{t}^{2}}{1-L\alpha_{t}\eta_{t}} & \leq\frac{1}{4\sigma^{2}}\quad\forall0\leq t\leq T.\label{eq:C2-1}
\end{align}
We will set $\left\{ w_{t}\right\} $ so that it satisfies the following
additional condition, which will allow us to telescope the sum on
the RHS of Corollary \ref{cor:acc-md-convergence}:
\begin{equation}
w_{t-1}\frac{\eta_{t-1}}{\alpha_{t-1}}\geq w_{t}\frac{\eta_{t}\left(1-\alpha_{t}\right)}{\alpha_{t}}\quad\forall1\leq t\leq T-1.\label{eq:C3-1}
\end{equation}
Given $w_{T}$, we set $w_{t-1}$ for every $1\leq t\leq T$ so that
the first condition (\ref{eq:C1-1}) holds with equality:
\[
w_{t-1}=w_{t}+6\sigma^{2}\eta_{t}^{2}w_{t}^{2}=w_{t}+6\sigma^{2}\eta^{2}t^{2}w_{t}^{2}.
\]
Let $C=6\sigma^{2}\sum_{i=1}^{T}\eta_{t}^{2}$. We set 
\[
w_{T}=\frac{1}{12\sigma^{2}\sum_{i=1}^{T}\eta_{t}^{2}}=\frac{1}{2C}.
\]
Given this choice for $w_{T}$, we now verify that, for all $0\leq t\leq T$,
we have
\[
w_{t}\leq\frac{1}{C+6\sigma^{2}\sum_{i=1}^{t}\eta_{i}^{2}}.
\]
We proceed by induction on $t$. The base case $t=T$ follows from
the definition of $w_{T}$. Consider $t\le T$. Using the definition
of $w_{t-1}$ and the inductive hypothesis, we obtain
\begin{align*}
w_{t-1} & =w_{t}+6\sigma^{2}\eta_{t}^{2}w_{t}^{2}\\
 & \leq\frac{1}{C+6\sigma^{2}\sum_{i=1}^{t}\eta_{i}^{2}}+\frac{6\sigma^{2}\eta_{t}^{2}}{\left(C+6\sigma^{2}\sum_{i=1}^{t}\eta_{i}^{2}\right)^{2}}\\
 & \leq\frac{1}{C+6\sigma^{2}\sum_{i=1}^{t}\eta_{i}^{2}}+\frac{\left(C+6\sigma^{2}\sum_{i=1}^{t}\eta_{i}^{2}\right)-\left(C+6\sigma^{2}\sum_{i=1}^{t-1}\eta_{i}^{2}\right)}{\left(C+6\sigma^{2}\sum_{i=1}^{t}\eta_{i}^{2}\right)\left(C+6\sigma^{2}\sum_{i=1}^{t-1}\eta_{i}^{2}\right)}\\
 & =\frac{1}{C+6\sigma^{2}\sum_{i=1}^{t-1}\eta_{i}^{2}}
\end{align*}
as needed.

Let us now verify that the second condition (\ref{eq:C2-1}) also
holds. Using that $L\eta_{t}\leq\frac{t}{4}$, and $T\geq2$, we obtain
\[
\frac{w_{t}\eta_{t}^{2}}{1-L\alpha_{t}\eta_{t}}\le\frac{w_{t}\eta_{t}^{2}}{1-\frac{t}{4}\frac{2}{t+1}}\leq2w_{t}\eta_{t}^{2}\leq\frac{2\eta_{t}^{2}}{6\sigma^{2}\sum_{i=1}^{T}\eta_{t}^{2}+6\sigma^{2}\sum_{i=1}^{t}\eta_{i}^{2}}\le\frac{2\eta_{t}^{2}}{12\sigma^{2}\eta_{t}^{2}}\leq\frac{1}{4\sigma^{2}}
\]
as needed.

Let us now verify that the third condition (\ref{eq:C3-1}) also holds.
Since $\alpha_{t}=\frac{2}{t+1}$, we have 
\begin{align*}
\frac{\eta_{t-1}}{\alpha_{t-1}} & =\frac{\eta_{t-1}t}{2},\\
\frac{\eta_{t}\left(1-\alpha_{t}\right)}{\alpha_{t}} & =\frac{\eta_{t}\left(t-1\right)}{2}.
\end{align*}
If $\eta_{t-1}=\frac{t-1}{4L}$ then we have $\eta_{t}\le\frac{t}{4L}$
and $\frac{\eta_{t}\left(1-\alpha_{t}\right)}{\alpha_{t}}\le\frac{\eta_{t-1}}{\alpha_{t-1}}=\frac{t(t-1)}{8L}$.
If $\eta_{t-1}=\frac{\eta}{\sqrt{t-1}}$ then $\eta_{t}=\frac{\eta}{\sqrt{t}},$we
also have $\frac{\eta_{t}\left(1-\alpha_{t}\right)}{\alpha_{t}}\le\frac{\eta_{t-1}}{\alpha_{t-1}}$.
Since $w_{t}\leq w_{t-1}$, it follows that condition (\ref{eq:C3-1})
holds.

We now turn our attention to the convergence. By Corollary \ref{cor:acc-md-convergence},
with probability $\geq1-\delta$, we have
\begin{align*}
 & \sum_{t=1}^{T}w_{t}\left(\frac{\eta_{t}}{\alpha_{t}}\left(f\left(y_{t}\right)-f\left(x^{*}\right)\right)-\frac{\eta_{t}\left(1-\alpha_{t}\right)}{\alpha_{t}}\left(f\left(y_{t-1}\right)-f\left(x^{*}\right)\right)\right)+w_{T}\breg\left(x^{*},z_{T}\right)\\
 & \leq w_{0}\breg\left(x^{*},z_{0}\right)+\left(G^{2}+3\sigma^{2}\right)\sum_{t=1}^{T}w_{t}\frac{\eta_{t}^{2}}{1-L\alpha_{t}\eta_{t}}+\log\left(\frac{1}{\delta}\right).
\end{align*}
Grouping terms on the LHS and using that $\alpha_{1}=1$, we obtain
\begin{align*}
 & \sum_{t=1}^{T-1}\left(w_{t}\frac{\eta_{t}}{\alpha_{t}}-w_{t+1}\frac{\eta_{t+1}\left(1-\alpha_{t+1}\right)}{\alpha_{t+1}}\right)\left(f\left(y_{t}\right)-f\left(x^{*}\right)\right)+w_{T}\frac{\eta_{T}}{\alpha_{T}}\left(f\left(y_{T}\right)-f\left(x^{*}\right)\right)+w_{T}\breg\left(x^{*},z_{T}\right)\\
 & \leq w_{0}\breg\left(x^{*},z_{0}\right)+\left(G^{2}+3\sigma^{2}\right)\sum_{t=1}^{T}w_{t}\frac{\eta_{t}^{2}}{1-L\alpha_{t}\eta_{t}}+\log\left(\frac{1}{\delta}\right).
\end{align*}
Since $\left\{ w_{t}\right\} $ satisfies condition (\ref{eq:C3-1}),
the coefficient of $f(y_{t})-f(x^{*})$ is non-negative and thus we
can drop the above sum. We obtain
\begin{align*}
w_{T}\frac{\eta_{T}}{\alpha_{T}}\left(f\left(y_{T}\right)-f\left(x^{*}\right)\right)+w_{T}\breg\left(x^{*},z_{T}\right) & \leq w_{0}\breg\left(x^{*},z_{0}\right)+\left(G^{2}+3\sigma^{2}\right)\sum_{t=1}^{T}w_{t}\frac{\eta_{t}^{2}}{1-L\alpha_{t}\eta_{t}}+\log\left(\frac{1}{\delta}\right).
\end{align*}
 Using that $w_{T}=\frac{1}{2C}$ and $w_{t}\leq\frac{1}{C}$ for
all $0\leq t\leq T-1$, we obtain
\begin{align*}
 & \frac{1}{2C}\frac{\eta_{T}}{\alpha_{T}}\left(f\left(y_{T}\right)-f\left(x^{*}\right)\right)+\frac{1}{2C}\breg\left(x^{*},z_{T}\right)\\
 & \leq\frac{1}{C}\breg\left(x^{*},z_{0}\right)+\frac{1}{C}\left(G^{2}+3\sigma^{2}\right)\sum_{t=1}^{T}\frac{\eta_{t}^{2}}{1-L\alpha_{t}\eta_{t}}+\log\left(\frac{1}{\delta}\right).
\end{align*}
Thus,
\begin{align*}
 & \frac{\eta_{T}}{\alpha_{T}}\left(f\left(y_{T}\right)-f\left(x^{*}\right)\right)+\breg\left(x^{*},z_{T}\right)\\
 & \leq2\breg\left(x^{*},z_{0}\right)+2\left(G^{2}+3\sigma^{2}\right)\sum_{t=1}^{T}\frac{\eta_{t}^{2}}{1-L\alpha_{t}\eta_{t}}+2C\log\left(\frac{1}{\delta}\right).
\end{align*}
Using that $L\eta_{t}\leq\frac{t}{4}$, we obtain
\[
\sum_{t=1}^{T}\frac{\eta_{t}^{2}}{1-L\alpha_{t}\eta_{t}}=\sum_{t=1}^{T}\frac{\eta_{t}^{2}}{1-\frac{t}{4}\frac{2}{t+1}}\leq\sum_{t=1}^{T}2\eta_{t}^{2}=\frac{C}{3\sigma^{2}}.
\]
Plugging in and using that $\eta_{T}=\eta T$ and $\alpha_{T}=\frac{2}{T+1}$,
we obtain
\begin{align*}
 & \frac{\eta_{T}\left(T+1\right)}{2}\left(f\left(y_{T}\right)-f\left(x^{*}\right)\right)+\breg\left(x^{*},z_{T}\right)\\
 & \leq2\breg\left(x^{*},z_{0}\right)+\left(2G^{2}+6\left(1+\log\left(\frac{1}{\delta}\right)\right)\sigma^{2}\right)\frac{C}{3\sigma^{2}}.
\end{align*}
If $\frac{T}{4L}\le\frac{\eta}{\sqrt{T}}$ which means $T^{3/2}\le4L\eta$,
then $\eta_{T}=\frac{T}{4L}$. We have 
\begin{align*}
C & =6\sigma^{2}\sum_{i=1}^{T}\eta_{t}^{2}=\frac{6\sigma^{2}}{16L^{2}}\sum_{i=1}^{T}t^{2}\le\frac{3\sigma^{2}T^{3}}{8L^{2}}\le6\sigma^{2}\eta^{2}.
\end{align*}
Hence
\begin{align*}
 & \frac{\eta_{T}\left(T+1\right)}{2}\left(f\left(y_{T}\right)-f\left(x^{*}\right)\right)+\breg\left(x^{*},z_{T}\right)\\
 & \leq2\breg\left(x^{*},z_{0}\right)+\left(G^{2}+\left(1+\log\left(\frac{1}{\delta}\right)\right)\sigma^{2}\right)\frac{3T^{3}}{4L^{2}},
\end{align*}
which entails
\begin{align*}
f\left(y_{T}\right)-f\left(x^{*}\right) & \le\frac{16L}{T^{2}}\breg\left(x^{*},z_{0}\right)+\left(G^{2}+\left(1+\log\left(\frac{1}{\delta}\right)\right)\sigma^{2}\right)\frac{6T}{L}\\
 & =\frac{16L}{T^{2}}\breg\left(x^{*},z_{0}\right)+\frac{6}{\sqrt{T}}\left(G^{2}+\left(1+\log\left(\frac{1}{\delta}\right)\right)\sigma^{2}\right)\frac{T^{3/2}}{L}\\
 & \le\frac{16L}{T^{2}}\breg\left(x^{*},z_{0}\right)+\frac{24}{\sqrt{T}}\left(G^{2}+\left(1+\log\left(\frac{1}{\delta}\right)\right)\sigma^{2}\right)\eta,
\end{align*}
and 
\begin{align*}
\breg\left(x^{*},z_{T}\right) & \leq2\breg\left(x^{*},z_{0}\right)+12\left(G^{2}+\left(1+\log\left(\frac{1}{\delta}\right)\right)\sigma^{2}\right)\eta^{2}.
\end{align*}
If $\frac{\eta}{\sqrt{T}}\le\frac{T}{4L}$ then $\eta_{T}=\frac{\eta}{\sqrt{T}}$.
Let $T_{0}$ be the largest $t$ such that $\frac{\eta}{\sqrt{t}}\ge\frac{t}{4L}$,
we have $T_{0}^{3}\le16L^{2}\eta^{2}$
\begin{align*}
C & =6\sigma^{2}\sum_{i=1}^{T}\eta_{t}^{2}\\
 & =6\sigma^{2}\sum_{i=1}^{T_{0}}\eta_{t}^{2}+6\sigma^{2}\sum_{i=T_{0}+1}^{T}\eta_{t}^{2}\\
 & =\frac{6\sigma^{2}}{16L^{2}}\sum_{i=1}^{T_{0}}t^{2}+6\sigma^{2}\eta^{2}\sum_{i=T_{0}+1}^{T}\frac{1}{t}\\
 & \le\frac{6\sigma^{2}}{16L^{2}}T_{0}^{3}+6\sigma^{2}\eta^{2}\sum_{i=T_{0}+1}^{T}\frac{1}{t}\\
 & \le6\sigma^{2}\eta^{2}\sum_{i=1}^{T}\frac{1}{t}\le6\sigma^{2}\eta^{2}(1+\log T).
\end{align*}
Hence
\begin{align*}
f\left(y_{T}\right)-f\left(x^{*}\right) & \le\frac{2}{T^{1/2}\eta}\left(2\breg\left(x^{*},z_{0}\right)+12\left(G^{2}+\left(1+\log\left(\frac{1}{\delta}\right)\right)\sigma^{2}\right)\eta^{2}(1+\log T)\right)
\end{align*}
and 
\begin{align*}
\breg\left(x^{*},z_{T}\right) & \le2\breg\left(x^{*},z_{0}\right)+12\left(G^{2}+\left(1+\log\left(\frac{1}{\delta}\right)\right)\sigma^{2}\right)\eta^{2}(1+\log T).
\end{align*}
 Overall, we have 
\begin{align*}
f\left(y_{T}\right)-f\left(x^{*}\right) & \le\frac{16L}{T^{2}}\breg\left(x^{*},z_{0}\right)+\frac{2}{T^{1/2}\eta}\left(2\breg\left(x^{*},z_{0}\right)+12\left(G^{2}+\left(1+\log\left(\frac{1}{\delta}\right)\right)\sigma^{2}\right)\eta^{2}(1+\log T)\right).
\end{align*}
\end{proof}

\section{Missing Proofs from Section \ref{sec:nonconvex}}

\begin{proof}[Proof of Lemma \ref{lem:sgd-basic-inequality}]
We start from the smoothness of $f$
\begin{align*}
f(x_{t+1})-f(x_{t}) & \le\left\langle \nabla f(x_{t}),x_{t+1}-x_{t}\right\rangle +\frac{L}{2}\left\Vert x_{t+1}-x_{t}\right\Vert ^{2}\\
 & =-\eta_{t}\left\langle \nabla f(x_{t}),\hn(x_{t})\right\rangle +\frac{L\eta_{t}^{2}}{2}\left\Vert \hn(x_{t})\right\Vert ^{2}.
\end{align*}
By writing $\hn(x_{t})=\xi_{t}+\nabla f(x_{t})$ we have 
\begin{align*}
f(x_{t+1})-f(x_{t}) & \le-\eta_{t}\left\langle \nabla f(x_{t}),\xi_{t}+\nabla f(x_{t})\right\rangle +\frac{L\eta_{t}^{2}}{2}\left\Vert \xi_{t}+\nabla f(x_{t})\right\Vert ^{2}\\
 & =-\eta_{t}\left\Vert \nabla f(x_{t})\right\Vert ^{2}-\eta_{t}\left\langle \nabla f(x_{t}),\xi_{t}\right\rangle \\
 & \quad+\frac{L\eta_{t}^{2}}{2}\left\Vert \xi_{t}\right\Vert ^{2}+\frac{L\eta_{t}^{2}}{2}\left\Vert \nabla f(x_{t})\right\Vert ^{2}+L\eta_{t}^{2}\left\langle \nabla f(x_{t}),\xi_{t}\right\rangle .
\end{align*}
We obtain the inequality (\ref{eq:sgd-basic-inequality}) by rearranging
the terms.
\end{proof}

\begin{proof}[Proof of Theorem \ref{thm:sgd-moment-inequality}]
We prove by induction. The base case $t=T+1$ trivially holds. Consider
$1\le t\le T$, we have
\begin{align*}
\E\left[\exp\left(S_{t}\right)\mid\F_{t}\right] & =\E\left[\E\left[\exp\left(Z_{t}+S_{t+1}\right)\mid\F_{t+1}\right]\mid\F_{t}\right]\\
 & =\E\left[\exp\left(Z_{t}\right)\E\left[\exp\left(S_{t+1}\right)\mid\F_{t+1}\right]\mid\F_{k}\right].
\end{align*}
From the induction hypothesis we have $\E\left[\exp\left(S_{t+1}\right)\mid\F_{t+1}\right]\le\exp\left(3\sigma^{2}\sum_{i=t+1}^{T}\frac{w_{i}\eta_{i}^{2}L}{2}\right)$,
hence 
\begin{align*}
\E\left[\exp\left(S_{t}\right)\mid\F_{t}\right] & \le\exp\left(3\sigma^{2}\sum_{i=t+1}^{T}\frac{w_{i}\eta_{i}^{2}L}{2}\right)\E\left[\exp\left(Z_{t}\right)\mid\F_{t}\right].
\end{align*}
We have then
\begin{align*}
\E\left[\exp\left(Z_{t}\right)\mid\F_{t}\right] & =\E\left[\exp\left(w_{t}\left(\eta_{t}\left(1-\frac{\eta_{t}L}{2}\right)\left\Vert \nabla f(x_{t})\right\Vert ^{2}+\Delta_{t+1}-\Delta_{t}\right)-v_{t}\left\Vert \nabla f(x_{T})\right\Vert ^{2}\right)\mid\F_{t}\right]\\
 & \leq\E\left[\exp\left(w_{t}\left(\eta_{t}(\eta_{t}L-1)\left\langle \nabla f(x_{t}),\xi_{t}\right\rangle +\frac{\eta_{t}^{2}L}{2}\left\Vert \xi_{t}\right\Vert ^{2}\right)-v_{t}\left\Vert \nabla f(x_{t})\right\Vert ^{2}\right)\mid\F_{t}\right]\\
 & =\exp\left(-v_{t}\left\Vert \nabla f(x_{t})\right\Vert ^{2}\right)\E\left[\exp\left(w_{t}\left(\eta_{t}(\eta_{t}L-1)\left\langle \nabla f(x_{t}),\xi_{t}\right\rangle +\frac{\eta_{t}^{2}L}{2}\left\Vert \xi_{t}\right\Vert ^{2}\right)\right)\mid\F_{t}\right]\\
 & \leq\exp\left(-v_{t}\left\Vert \nabla f(x_{t})\right\Vert ^{2}\right)\exp\left(3\sigma^{2}\left(w_{t}^{2}\eta_{t}^{2}(\eta_{t}L-1)^{2}\left\Vert \nabla f(x_{t})\right\Vert ^{2}+\frac{w_{t}\eta_{t}^{2}L}{2}\right)\right)\\
 & =\exp\left(3\sigma^{2}\frac{w_{t}\eta_{t}^{2}L}{2}\right).
\end{align*}
where the second line is due to (\ref{eq:sgd-basic-inequality}) in
Lemma \ref{lem:sgd-basic-inequality} and the second to last line
is due to Lemma \ref{lem:helper-taylor-vector}.Therefore 
\begin{align*}
\E\left[\exp\left(S_{t}\right)\mid\F_{t}\right] & \leq\exp\left(3\sigma^{2}\sum_{i=t}^{T}\frac{w_{i}\eta_{i}^{2}L}{2}\right)
\end{align*}
which is what we want to show.
\end{proof}

\begin{proof}[Proof of Corollary \ref{cor:sgd-general-guarantee}]
In Lemma \ref{thm:sgd-moment-inequality}, Let $t=1$ we obtain 
\begin{align*}
\E\left[\exp\left(S_{1}\right)\right] & \leq\exp\left(3\sigma^{2}\sum_{t=1}^{T}\frac{w_{t}\eta_{t}^{2}L}{2}\right).
\end{align*}
Hence, by Markov's inequality, we have 
\[
\Pr\left[S_{1}\geq\left(3\sigma^{2}\sum_{t=1}^{T}\frac{w_{t}\eta_{t}^{2}L}{2}\right)+\log\frac{1}{\delta}\right]\leq\delta.
\]
In other words, with probability $\geq1-\delta$ (once the condition
in Lemma \ref{thm:sgd-moment-inequality} is satisfied)
\begin{align*}
 & \sum_{t=1}^{T}\left[w_{t}\eta_{t}\left(1-\frac{\eta_{t}L}{2}\right)-v_{t}\right]\left\Vert \nabla f(x_{t})\right\Vert ^{2}+w_{t}\left(\Delta_{t+1}-\Delta_{t}\right)\\
 & \leq3\sigma^{2}\sum_{t=1}^{T}\frac{w_{t}\eta_{t}^{2}L}{2}+\log\frac{1}{\delta}.
\end{align*}
This gives 
\begin{align*}
\sum_{t=1}^{T}\left[w_{t}\eta_{t}\left(1-\frac{\eta_{t}L}{2}\right)-v_{t}\right]\left\Vert \nabla f(x_{t})\right\Vert ^{2}+w_{T}\Delta_{T+1} & \le w_{1}\Delta_{1}+\left(\sum_{t=2}^{T}(w_{t}-w_{t-1})\Delta_{t}+3\sigma^{2}\sum_{t=1}^{T}\frac{w_{t}\eta_{t}^{2}L}{2}\right)+\log\frac{1}{\delta}
\end{align*}
 as needed.
\end{proof}

\begin{proof}[Proof of Theorem \ref{thm:sgd-convergence} ]

\textbf{First case.}

Starting from this inequality, we will specify the choice of $\eta_{t}$
and $w_{t}$ to obtain the bound. Consider $\eta_{t}=\eta$ with $\eta L\leq1$,
$w_{t}=w=\frac{1}{6\sigma^{2}\eta}$. Note that $w_{t}\eta_{t}^{2}L=\frac{\eta L}{6\sigma^{2}}\le\frac{1}{2\sigma^{2}}$
satisfies the condition of Lemma \ref{thm:sgd-moment-inequality},
we have
\begin{align*}
\mbox{LHS of \eqref{eq:sgd-general-bound}} & =w\Delta_{T+1}+\sum_{t=1}^{T}\left[w\eta\left(1-\frac{\eta L}{2}\right)-3\sigma^{2}w^{2}\eta^{2}(\eta L-1)^{2}\right]\left\Vert \nabla f(x_{t})\right\Vert ^{2}\\
 & =w\Delta_{T+1}+w\eta\sum_{t=1}^{T}\left[1-\frac{\eta L}{2}-\frac{1}{2}(\eta L-1)^{2}\right]\left\Vert \nabla f(x_{t})\right\Vert ^{2}\\
 & \ge w\Delta_{T+1}+\frac{w\eta}{2}\sum_{t=1}^{T}\left\Vert \nabla f(x_{t})\right\Vert ^{2}
\end{align*}
where the last inequality is due to $1-\frac{\eta L}{2}-\frac{(1-\eta L)^{2}}{2}\geq\frac{1}{2}$
when $0\leq\eta L\leq1$. Besides, 
\begin{align*}
\mbox{RHS of \eqref{eq:sgd-general-bound}}= & w\Delta_{1}+\frac{3\sigma^{2}}{2}w\eta^{2}LT+\log\frac{1}{\delta}.
\end{align*}
Hence with probability $\geq1-\delta$
\begin{align*}
\sum_{t=1}^{T}\left\Vert \nabla f(x_{t})\right\Vert ^{2}+\frac{2\Delta_{T+1}}{\eta} & \leq\frac{2\Delta_{1}}{\eta}+3\sigma^{2}\eta LT+\frac{2}{w\eta}\log\frac{1}{\delta}\\
 & =\frac{2\Delta_{1}}{\eta}+3\sigma^{2}\eta LT+12\sigma^{2}\log\frac{1}{\delta}.
\end{align*}
Finally, by choosing $\eta=\min\left\{ \frac{1}{L};\sqrt{\frac{\Delta_{1}}{\sigma^{2}LT}}\right\} $
and noticing $\Delta_{T+1}\geq0$, we obtain the desired inequality.

\textbf{Second case.}

Consider $\eta_{t}=\frac{\eta}{\sqrt{t}}$ with $\eta L\leq1$, $w_{t}=w=\frac{1}{6\sigma^{2}\eta}$
. Again, we have $w_{t}\eta_{t}^{2}L=\frac{\eta L}{6\sigma^{2}t}\leq\frac{1}{2\sigma^{2}}$,
then
\begin{align*}
\mbox{LHS of \eqref{eq:sgd-general-bound}}= & \sum_{t=1}^{T}\left[\frac{w\eta}{\sqrt{t}}\left(1-\frac{\eta L}{2\sqrt{t}}\right)-\frac{3\sigma^{2}w^{2}\eta^{2}}{t}\left(1-\frac{\eta L}{\sqrt{t}}\right)^{2}\right]\left\Vert \nabla f(x_{t})\right\Vert ^{2}+w\Delta_{T+1}\\
= & \sum_{t=1}^{T}\frac{w\eta}{\sqrt{t}}\left[1-\frac{\eta L}{2\sqrt{t}}-\frac{3\sigma^{2}w\eta}{\sqrt{t}}\left(1-\frac{\eta L}{\sqrt{t}}\right)^{2}\right]\left\Vert \nabla f(x_{t})\right\Vert ^{2}+w\Delta_{T+1}\\
\text{\ensuremath{\geq}} & \sum_{t=1}^{T}\frac{w\eta}{\sqrt{t}}\left[1-\frac{\eta L}{2\sqrt{t}}-3\sigma^{2}w\eta\left(1-\frac{\eta L}{\sqrt{t}}\right)^{2}\right]\left\Vert \nabla f(x_{t})\right\Vert ^{2}+w\Delta_{T+1}\\
= & \sum_{t=1}^{T}\frac{w\eta}{\sqrt{t}}\left[1-\frac{\eta L}{2\sqrt{t}}-\frac{1}{2}\left(1-\frac{\eta L}{\sqrt{t}}\right)^{2}\right]\left\Vert \nabla f(x_{t})\right\Vert ^{2}+w\Delta_{T+1}\\
\geq & \sum_{t=1}^{T}\frac{w\eta}{2\sqrt{t}}\left\Vert \nabla f(x_{t})\right\Vert ^{2}+w\Delta_{T+1}\geq\frac{w\eta}{2\sqrt{T}}\sum_{t=1}^{T}\left\Vert \nabla f(x_{t})\right\Vert ^{2}+w\Delta_{T+1}
\end{align*}
where the second inequality is due to $1-\frac{\eta L}{2\sqrt{t}}-\frac{1}{2}\left(1-\frac{\eta L}{\sqrt{t}}\right)^{2}\geq\frac{1}{2}$
when $0\leq\frac{\eta L}{\sqrt{t}}\leq1$. Besides,
\begin{align*}
\mbox{RHS of \eqref{eq:sgd-general-bound}}= & w\Delta_{1}+\frac{3\sigma^{2}}{2}w\eta^{2}L\sum_{t=1}^{T}\frac{1}{t}+\log\frac{1}{\delta}\\
\le & w\Delta_{1}+\frac{3\sigma^{2}}{2}w\eta^{2}L(1+\log T)+\log\frac{1}{\delta}.
\end{align*}
Therefore with probability $\geq1-\delta$
\begin{align*}
 & \sum_{t=1}^{T}\left\Vert \nabla f(x_{t})\right\Vert ^{2}+\frac{2\sqrt{T}\Delta_{T+1}}{\eta}\\
\le & \sqrt{T}\left(\frac{2\Delta_{1}}{\eta}+3\sigma^{2}\eta L\left(1+\log T\right)+\frac{2}{w\eta}\log\frac{1}{\delta}\right)\\
= & \sqrt{T}\left(\frac{2\Delta_{1}}{\eta}+3\sigma^{2}\eta L\left(1+\log T\right)+12\sigma^{2}\log\frac{1}{\delta}\right).
\end{align*}
Choose $\eta=\frac{1}{L}$, and notice $\Delta_{T+1}\geq0$, we obtain
\begin{align*}
\frac{1}{T}\sum_{t=1}^{T}\left\Vert \nabla f(x_{t})\right\Vert ^{2} & \le\frac{2\Delta_{1}L+3\sigma^{2}\left(1+\log T\right)+12\sigma^{2}\log\frac{1}{\delta}}{\sqrt{T}}.
\end{align*}
\end{proof}

\section{AdaGrad-Norm omitted proofs \label{sec:AdaGrad-Norm-second-proof}}

We first provide the proofs for some of the Lemmas in Section \ref{subsec:main-body-adagradnorm-and-adagrad}.

\begin{proof}[Proof of Lemma \ref{lem:adagrad-basics}]
We start by using the smoothness of $f$
\begin{align}
f(x_{t+1})-f(x_{t}) & \le\left\langle \n(x_{t}),x_{t+1}-x_{t}\right\rangle +\frac{L}{2}\left\Vert x_{t+1}-x_{t}\right\Vert ^{2}\nonumber \\
 & =-\frac{\eta}{b_{t}}\left\langle \n(x_{t}),\hn(x_{t})\right\rangle +\frac{L\eta^{2}}{2b_{t}^{2}}\left\Vert \hn(x_{t})\right\Vert ^{2}\nonumber \\
 & =-\frac{\eta}{b_{t}}\left\Vert \n(x_{t})\right\Vert ^{2}-\frac{\eta}{b_{t}}\left\langle \n(x_{t}),\xi_{t}\right\rangle +\frac{L\eta^{2}}{2b_{t}^{2}}\left\Vert \hn(x_{t})\right\Vert ^{2}\nonumber \\
 & =\eta\left(\frac{1}{a_{t}}-\frac{1}{b_{t}}\right)\left\langle \nf(x_{t}),\xi_{t}\right\rangle -\frac{\eta}{a_{t}}\left\langle \nf(x_{t}),\xi_{t}\right\rangle -\frac{\eta}{b_{t}}\norm{\nf(x_{t})}^{2}+\frac{L\eta^{2}}{2b_{t}^{2}}\norm{\hn(x_{t})}^{2}\label{eq:adagrad-starting-point-1}
\end{align}
First, by Lemma \ref{lem:adagrad-proxy-difference}, we have
\[
\left|\frac{1}{a_{t}}-\frac{1}{b_{t}}\right|\leq\frac{\norm{\xi_{t}}}{a_{t}b_{t}}.
\]
This gives 
\begin{align*}
\left(\frac{1}{a_{t}}-\frac{1}{b_{t}}\right)\left\langle \nf(x_{t}),\xi_{t}\right\rangle  & \leq\left|\frac{1}{a_{t}}-\frac{1}{b_{t}}\right|\norm{\nf(x_{t})}\norm{\xi_{t}}\\
 & \leq\frac{\norm{\xi_{t}}}{a_{t}b_{t}}\norm{\nf(x_{t})}\norm{\xi_{t}}\\
 & \leq\norm{\xi_{t}}\left(\frac{\norm{\nf(x_{t})}^{2}}{2a_{t}^{2}}+\frac{\norm{\xi_{t}}^{2}}{2b_{t}^{2}}\right).
\end{align*}
Plugging this back into \ref{eq:adagrad-starting-point-1}, we have
\begin{align*}
f(x_{t+1})-f(x_{t}) & \leq\eta\norm{\xi_{t}}\left(\frac{\norm{\nf(x_{t})}^{2}}{2a_{t}^{2}}+\frac{\norm{\xi_{t}}^{2}}{2b_{t}^{2}}\right)-\frac{\eta\left\langle \nf(x_{t}),\xi_{t}\right\rangle }{a_{t}}-\frac{\eta}{b_{t}}\norm{\nf(x_{t})}^{2}+\frac{L\eta^{2}}{2b_{t}^{2}}\norm{\hn(x_{t})}^{2}
\end{align*}
After summing up, rearranging the terms, using $\norm{\xi_{t}}\le M_{T}$
and $f(x_{1})-f(x_{T+1})\le\Delta_{1}$, we obtain 
\begin{align*}
\sum_{t=1}^{T}\frac{\norm{\nf(x_{t})}^{2}}{b_{t}} & \leq\frac{\Delta_{1}}{\eta}+\frac{M_{T}}{2}\left[\sum_{t=1}^{T}\frac{\norm{\nf(x_{t})}^{2}}{a_{t}^{2}}+\sum_{t=1}^{T}\frac{\norm{\xi_{t}}^{2}}{b_{t}^{2}}\right]-\sum_{t=1}^{T}\frac{\left\langle \nf(x_{t}),\xi_{t}\right\rangle }{a_{t}}+\sum_{t=1}^{T}\frac{L\eta}{2b_{t}^{2}}\norm{\hn(x_{t})}^{2}.
\end{align*}
\end{proof}

\begin{proof}[Proof of Lemma \ref{lem:martingale-dif-error-term}]
By Lemma \ref{lem:helper-taylor-vector} with some $w>0$, we have
\[
\E\left[\exp\left(\left\langle -w\frac{\nf(x_{t}),\xi_{t}}{a_{t}}\right\rangle -2\sigma^{2}w^{2}\frac{\norm{\nf(x_{t})}^{2}}{a_{t}^{2}}\right)\mid\F_{t}\right]\leq1.
\]
Thus it is not difficult to verify that 
\[
\E\left[\exp\left(\sum_{t=1}^{T}\left\langle -w\frac{\nf(x_{t}),\xi_{t}}{a_{t}}\right\rangle -2\sigma^{2}w^{2}\frac{\norm{\nf(x_{t})}^{2}}{a_{t}^{2}}\right)\right]\leq1.
\]
By Markov's inequality we obtain, with probability at least $1-\delta$,
\[
\sum_{t=1}^{T}-\frac{\left\langle \nf(x_{t}),\xi_{t}\right\rangle }{a_{t}}\leq2\sigma^{2}w\sum_{t=1}^{T}\frac{\norm{\nf(x_{t})}^{2}}{a_{t}^{2}}+\frac{1}{w}\log\frac{1}{\delta}.
\]
It is also known that with probability at least $1-\delta$, $M_{T}\leq\sigma\sqrt{1+\log\frac{T}{\delta}}\leq2\sigma\sqrt{\log\frac{T}{\delta}}$
\citet{li2020high,liu2022convergence} for $T\geq1$ and $\delta\in(0,1)$.
Thus by a union bound and setting $w:=\frac{\sqrt{\log\frac{1}{\delta}}}{\sigma}$,
we can bound Lemma \ref{lem:adagrad-basics} with probability at least
$1-2\delta$
\begin{align*}
\sum_{t=1}^{T}\frac{\norm{\nf(x_{t})}^{2}}{b_{t}} & \leq\frac{\Delta_{1}}{\eta}+M_{T}\sum_{t=1}^{T}\frac{\norm{\nf(x_{t})}^{2}}{2a_{t}^{2}}+M_{T}\sum_{t=1}^{T}\frac{\norm{\xi_{t}}^{2}}{2b_{t}^{2}}-\sum_{t=1}^{T}\frac{\left\langle \nf(x_{t}),\xi_{t}\right\rangle }{a_{t}}+\sum_{t=1}^{T}\frac{L\eta}{2b_{t}^{2}}\norm{\hn(x_{t})}^{2}\\
 & \leq\frac{\Delta_{1}}{\eta}+\sigma\sqrt{\log\frac{T}{\delta}}\left[2\underbrace{\sum_{t=1}^{T}\frac{\norm{\nf(x_{t})}^{2}}{a_{t}^{2}}}_{B}+\sum_{t=1}^{T}\frac{\norm{\xi_{t}}^{2}}{b_{t}^{2}}\right]+\sigma\sqrt{\log\frac{1}{\delta}}+\frac{L\eta}{2}\underbrace{\sum_{t=1}^{T}\frac{\norm{\hn(x_{t})}^{2}}{b_{t}^{2}}}_{A}.
\end{align*}
Let us consider the term $A$. We have
\begin{align*}
\sum_{t=1}^{T}\frac{\norm{\hn(x_{t})}^{2}}{b_{t}^{2}} & =\sum_{t=1}^{T}\frac{b_{t}^{2}-b_{t-1}^{2}}{b_{t}^{2}}=\sum_{t=1}^{T}1-\frac{b_{t-1}^{2}}{b_{t}^{2}}\\
 & \leq2\sum_{t=1}^{T}\log\frac{b_{t}}{b_{t-1}}=2\log\frac{b_{T}}{b_{0}}.
\end{align*}
For $B$, note that since $\norm{\nf(x_{t})}^{2}\leq2\norm{\hn(x_{t})}^{2}+2\norm{\xi_{t}}^{2}$,
we have
\begin{align*}
\sum_{t=1}^{T}\frac{\norm{\nf(x_{t})}^{2}}{a_{t}^{2}} & =\sum_{t=1}^{T}\frac{\norm{\nf(x_{t})}^{2}}{b_{t-1}^{2}+\norm{\nf(x_{t})}^{2}}\\
 & \overset{(*)}{\leq}\sum_{t=1}^{T}\frac{2\norm{\hn(x_{t})}^{2}+2\norm{\xi_{t}}^{2}}{b_{t-1}^{2}+2\norm{\hn(x_{t})}^{2}+2\norm{\xi_{t}}^{2}}\\
 & =\sum_{t=1}^{T}\frac{2\norm{\hn(x_{t})}^{2}}{b_{t-1}^{2}+2\norm{\hn(x_{t})}^{2}+2\norm{\xi_{t}}^{2}}+\sum_{t=1}^{T}\frac{2\norm{\xi_{t}}^{2}}{b_{t-1}^{2}+2\norm{\hn(x_{t})}^{2}+2\norm{\xi_{t}}^{2}}\\
 & \leq2\sum_{t=1}^{T}\frac{\norm{\hn(x_{t})}^{2}}{b_{t}^{2}}+2\sum_{t=1}^{T}\frac{\norm{\xi_{t}}^{2}}{b_{t}^{2}}\\
 & \leq4\log\left(\frac{b_{T}}{b_{0}}\right)+2\sum_{t=1}^{T}\frac{\norm{\xi_{t}}^{2}}{b_{t}^{2}}.
\end{align*}
For $(*)$ we use the fact that $\frac{x}{c+x}$ is an increasing
function. Combining the bound for $A$ and $B$, we obtain, with probability
at least $1-2\delta$, 
\begin{align*}
\sum_{t=1}^{T}\frac{\norm{\nf(x_{t})}^{2}}{b_{t}} & \leq\frac{\Delta_{1}}{\eta}+\sigma\sqrt{\log\frac{T}{\delta}}\left[8\log\left(\frac{b_{T}}{b_{0}}\right)+5\sum_{t=1}^{T}\frac{\norm{\xi_{t}}^{2}}{b_{t}^{2}}\right]+\sigma\sqrt{\log\frac{1}{\delta}}+L\eta\log\frac{b_{T}}{b_{0}}.
\end{align*}
\end{proof}

\begin{lem}
\label{lem:bound-on-bT}For AdaGrad-Norm stepsizes $b_{t}$, if $f$
is $L$-smooth and the stochastic gradients have $\sigma$-subgaussian
noise, then with probability at least $1-\delta$
\begin{align*}
b_{T} & \leq4b_{0}+4\frac{\Delta_{1}}{\eta}+\frac{32}{\eta^{2}b_{0}}\sigma^{2}\ln\left(\frac{2}{\delta}\right)+\frac{16\sigma}{\eta^{2}}\sqrt{T+\log\frac{2}{\delta}}+4L\eta\log\frac{L\eta}{b_{0}}\\
 & =O\left(\Delta_{1}+\sigma\sqrt{T}+\sigma^{2}\log\frac{1}{\delta}+L\log L\right).
\end{align*}
\end{lem}
\begin{proof}
We start from function value analysis
\begin{align*}
f(x_{t+1})-f(x_{t}) & \leq\left\langle \nabla f(x_{t}),x_{t+1}-x_{t}\right\rangle +\frac{L}{2}\norm{x_{t+1}-x_{t}}^{2}\\
 & =\frac{1}{b_{t}}\left\langle \widehat{\nabla}f(x_{t}),\xi_{t}\right\rangle +\eta\left(\frac{L\eta}{2b_{t}^{2}}-\frac{1}{2b_{t}}\right)\norm{\widehat{\nabla}f(x_{t})}^{2}-\frac{\eta}{2b_{t}}\norm{\widehat{\nabla}f(x_{t})}^{2}.
\end{align*}
We can bound $\sum_{t=1}^{T}\left(\frac{L\eta}{2b_{t}^{2}}-\frac{1}{2b_{t}}\right)\norm{\widehat{\nabla}f(x_{t})}^{2}$via
a standard argument. Let $\tau=\max\left\{ t\leq T\mid b_{t}\leq\eta L\right\} $
so that $t\geq\tau$ implies $b_{t}\geq\eta L\iff\frac{L\eta}{b_{t}^{2}}\leq\frac{1}{b_{t}}$
. Then
\begin{align*}
\sum_{t=1}^{T}\left(\frac{L\eta}{2b_{t}^{2}}-\frac{1}{2b_{t}}\right)\norm{\widehat{\nabla}f(x_{t})}^{2} & \leq\sum_{t=1}^{\tau}\left(\frac{L\eta}{2b_{t}^{2}}-\frac{1}{2b_{t}}\right)\norm{\widehat{\nabla}f(x_{t})}^{2}\\
 & \leq\frac{L\eta}{2}\sum_{t=1}^{\tau}\frac{1}{b_{t}^{2}}\norm{\widehat{\nabla}f(x_{t})}^{2}\\
 & =L\eta\log\frac{b_{\tau}}{b_{0}}\leq L\eta\log\frac{L\eta}{b_{0}}.
\end{align*}
Summing and plugging in the above gives 
\begin{align*}
f(x_{T+1})-f(x_{1}) & \leq\sum_{t=1}^{T}\frac{1}{b_{t}}\left\langle \widehat{\nabla}f(x_{t}),\xi_{t}\right\rangle +L\eta^{2}\log\frac{L\eta}{b_{0}}-\frac{\eta}{2}\sum_{t=1}^{T}\frac{\norm{\widehat{\nabla}f(x_{t})}^{2}}{b_{t}}\\
 & \leq\frac{1}{\eta}\sum_{t=1}^{T}\frac{\norm{\xi_{t}}^{2}}{b_{t}}-\frac{\eta}{4}\sum_{t=1}^{T}\frac{\norm{\widehat{\nabla}f(x_{t})}^{2}}{b_{t}}+L\eta^{2}\log\frac{L\eta}{b_{0}},
\end{align*}
where we use $\frac{1}{b_{t}}\left\langle \widehat{\nabla}f(x_{t}),\xi_{t}\right\rangle \leq\frac{\norm{\xi_{t}}^{2}}{\eta b_{t}}+\frac{\eta\norm{\widehat{\nabla}f(x_{t})}^{2}}{4b_{t}}$
in the second inequality. Rearranging and dividing by $\eta$, we
get
\[
\sum_{t=1}^{T}\frac{\norm{\widehat{\nabla}f(x_{t})}^{2}}{4b_{t}}\leq\frac{f(x_{1})-f(x_{T+1})}{\eta}+\frac{1}{\eta^{2}}\sum_{t=1}^{T}\frac{\norm{\xi_{t}}^{2}}{b_{t}}+L\eta\log\frac{L\eta}{b_{0}}.
\]
On the LHS, we have 
\begin{align*}
\sum_{t=1}^{T}\frac{\norm{\widehat{\nabla}f(x_{t})}^{2}}{b_{t}} & =\sum_{t=1}^{T}\frac{b_{t}^{2}-b_{t-1}^{2}}{b_{t}}\geq\sum_{t=1}^{T}b_{t}-\frac{b_{t-1}^{2}}{b_{t-1}}=b_{T}-b_{0}.
\end{align*}
Combining this with Lemma \ref{lem:bound-on-eps-sq-by-bt} where $\sum_{t=1}^{T}\frac{\norm{\xi_{t}}^{2}}{b_{t}}\leq\frac{8}{b_{0}}\sigma^{2}\ln\left(\frac{2}{\delta}\right)-b_{0}+4\sigma\sqrt{T+\log\frac{2}{\delta}}$,
we get the result.
\end{proof}

Now, we can prove Theorem \ref{thm:adagrad-nonconvex}.

\begin{proof}[Proof of Theorem \ref{thm:adagrad-nonconvex}]
From Lemma \ref{lem:martingale-dif-error-term}, we have with probability
at least $1-2\delta$
\[
\sum_{t=1}^{T}\frac{\norm{\nf(x_{t})}^{2}}{b_{t}}\leq\frac{\Delta_{1}}{\eta}+\sigma\sqrt{\log\frac{T}{\delta}}\left[8\log\left(\frac{b_{T}}{b_{0}}\right)+5\sum_{t=1}^{T}\frac{\norm{\xi_{t}}^{2}}{b_{t}^{2}}\right]+\sigma\sqrt{\log\frac{1}{\delta}}+L\eta\log\frac{b_{T}}{b_{0}}.
\]
Since $b_{t}$ is increasing, we have $\sum_{t=1}^{T}\frac{\norm{\nf(x_{t})}^{2}}{b_{t}}\geq\sum_{t=1}^{T}\frac{\norm{\nf(x_{t})}^{2}}{b_{T}}$.
That means 
\[
\sum_{t=1}^{T}\norm{\nf(x_{t})}^{2}\leq b_{T}\left[\frac{\Delta_{1}}{\eta}+\sigma\sqrt{\log\frac{T}{\delta}}\left[8\log\left(\frac{b_{T}}{b_{0}}\right)+5\sum_{t=1}^{T}\frac{\norm{\xi_{t}}^{2}}{b_{t}^{2}}\right]+\sigma\sqrt{\log\frac{1}{\delta}}+L\eta\log\left(\frac{b_{T}}{b_{0}}\right)\right].
\]
Combining this with the event from Lemma \ref{lem:bound-on-eps-sq-by-bt-sq}
that bounds $\sum_{t=1}^{T}\frac{\norm{\xi_{t}}^{2}}{b_{t}^{2}}$
and Lemma \ref{lem:bound-on-bT} that bounds $b_{T}$ gives us the
Theorem.
\end{proof}

\subsection{Additional helper lemmas}
\begin{lem}
\label{lem:adagrad-proxy-difference}For $t\ge1$ and $a_{t}$, $\xi_{t}$
defined in Lemma \ref{lem:adagrad-basics}, we have 
\[
\left|\frac{1}{a_{t}}-\frac{1}{b_{t}}\right|\leq\frac{\norm{\xi_{t}}}{a_{t}b_{t}}.
\]
\end{lem}
\begin{proof}
We have 
\begin{align*}
\left|\frac{1}{a_{t}}-\frac{1}{b_{t}}\right| & =\left|\frac{b_{t}-a_{t}}{a_{t}b_{t}}\right|\\
 & =\left|\frac{b_{t}^{2}-a_{t}^{2}}{a_{t}b_{t}\left(b_{t}+a_{t}\right)}\right|\\
 & =\left|\frac{b_{t}^{2}-b_{t-1}^{2}-\norm{\nf(x_{t})}^{2}}{a_{t}b_{t}\left(b_{t}+a_{t}\right)}\right|\\
 & =\left|\frac{\norm{\hn(x_{t})}^{2}-\norm{\nf(x_{t})}^{2}}{a_{t}b_{t}\left(b_{t}+a_{t}\right)}\right|\\
 & \leq\left|\frac{\left(\norm{\hn(x_{t})}-\norm{\nf(x_{t})}\right)\left(\norm{\hn(x_{t})}+\norm{\nf(x_{t})}\right)}{a_{t}b_{t}\left(b_{t}+a_{t}\right)}\right|.
\end{align*}
Since $b_{t}=\sqrt{b_{t-1}^{2}+\norm{\hn(x_{t})}^{2}}\geq\norm{\hn(x_{t})}$
and $a_{t}=\sqrt{b_{t-1}^{2}+\norm{\nf(x_{t})}^{2}}\geq\norm{\nf(x_{t})}$,
we have
\begin{align*}
\left|\frac{1}{a_{t}}-\frac{1}{b_{t}}\right| & \leq\left|\frac{\norm{\hn(x_{t})}-\norm{\nf(x_{t})}}{a_{t}b_{t}}\right|\\
 & \leq\frac{\norm{\hn(x_{t})-\nf(x_{t})}}{a_{t}b_{t}}\\
 & =\frac{\norm{\xi_{t}}}{a_{t}b_{t}}.
\end{align*}
\end{proof}

\begin{lem}
\label{lem:upperbound-on-error}With prob $\geq1-\delta$, for any
$0\leq t\leq T$, we have
\[
\sum_{s=1}^{t}\|\xi_{s}\|^{2}\leq\sum_{s=1}^{t}\|\widehat{\nabla}f(x_{s})\|^{2}+4\sigma^{2}\log\frac{1}{\delta}.
\]
\end{lem}
\begin{proof}
Note that
\begin{align*}
\|\widehat{\nabla}f(x_{t})\|^{2} & =\|\nabla f(x_{t})\|^{2}+2\langle\xi_{t},\nabla f(x_{t})\rangle+\|\xi_{t}\|^{2}\\
\Rightarrow\|\nabla f(x_{t})\|-\|\widehat{\nabla}f(x_{t})\|^{2}+\|\xi_{t}\|^{2} & =2\langle\xi_{t},\nabla f(x_{t})\rangle.
\end{align*}
Define for $t\in\left\{ 0,1,\cdots,T\right\} $
\begin{align*}
U_{t+1} & =\exp\left(\sum_{s=1}^{t}w_{s}\left(\|\nabla f(x_{s})\|^{2}-\|\widehat{\nabla}f(x_{s})\|^{2}+\|\xi_{s}\|^{2}\right)-v_{s}\|\nabla f(x_{s})\|^{2}\right);\quad v_{s}=4\sigma^{2}w_{s}^{2}.
\end{align*}
Let $\F_{t}=\sigma(\xi_{i\leq t-1})$. We know $U_{t}\in\F_{t}$.
Note that $U_{t}$ is a supermartingale
\begin{align*}
\E\left[U_{t+1}\mid\F_{t}\right] & =U_{t}\exp\left(-v_{t}\|\nabla f(x_{t})\|^{2}\right)\E\left[\exp\left(2w_{t}\langle\xi_{t},\nabla f(x_{t})\rangle\right)\mid\F_{t}\right]\\
 & \leq U_{t}\exp\left(-v_{t}\|\nabla f(x_{t})\|^{2}\right)\E\left[\exp\left(4\sigma^{2}w_{t}^{2}\|\nabla f(x_{t})\|^{2}\right)\mid\F_{t}\right]\\
 & =U_{t}
\end{align*}
By Doob's supermartingale inequality, there is
\[
\Pr\left[\max_{t\in\left[T+1\right]}U_{t}\geq\delta^{-1}\right]\leq\delta\E\left[U_{1}\right]=\delta
\]
which implies w.p. $\geq1-\delta$, $\forall0\leq t\leq T$
\begin{align*}
\sum_{s=1}^{t}w_{s}\left(\|\nabla f(x_{s})\|^{2}-\|\widehat{\nabla}f(x_{s})\|^{2}+\|\xi_{s}\|^{2}\right)-v_{s}\|\nabla f(x_{s})\|^{2} & \leq\log\frac{1}{\delta}\\
\sum_{s=1}^{t}\left(w_{s}-4\sigma^{2}w_{s}^{2}\right)\|\nabla f(x_{s})\|^{2}+w_{s}\|\xi_{s}\|^{2} & \leq\sum_{s=1}^{t}w_{s}\|\widehat{\nabla}f(x_{s})\|^{2}+\log\frac{1}{\delta}.
\end{align*}
Set $w_{s}=\frac{1}{4\sigma^{2}}$ to get
\[
\sum_{s=1}^{t}\|\xi_{s}\|^{2}\leq\sum_{s=1}^{t}\|\widehat{\nabla}f(x_{s})\|^{2}+4\sigma^{2}\log\frac{1}{\delta}.
\]
\end{proof}

\begin{lem}
\label{lem:sum-error-naive-bound}With probability $\geq1-\delta$,
we have
\[
\sum_{t=1}^{T}\|\xi_{t}\|^{2}\leq\sigma^{2}T+\sigma^{2}\log\frac{1}{\delta}.
\]
\end{lem}
\begin{proof}
Note that
\begin{align*}
\Pr\left[\sum_{t=1}^{T}\|\xi_{t}\|^{2}\geq u\right] & =\Pr\left[\exp\left(\sum_{t=1}^{T}\frac{\|\xi_{t}\|^{2}}{\sigma^{2}}\right)\geq\exp\left(\frac{u}{\sigma^{2}}\right)\right]\\
 & \leq\frac{\E\left[\exp\left(\sum_{t=1}^{T}\frac{\|\xi_{t}\|^{2}}{\sigma^{2}}\right)\right]}{\exp\left(\frac{u}{\sigma^{2}}\right)}\\
 & \leq\frac{\exp(T)}{\exp\left(\frac{u}{\sigma^{2}}\right)}
\end{align*}
where we choose 
\[
u=\sigma^{2}T+\sigma^{2}\log\frac{1}{\delta}.
\]
\end{proof}

\begin{lem}
\label{lem:bound-on-eps-sq-by-bt}For AdaGrad stepsize $b_{t}$ and
$\sigma$-subgaussian noise $\norm{\xi_{t}}$, with probability at
least $1-\delta$
\[
\sum_{t=1}^{T}\frac{\norm{\xi_{t}}^{2}}{b_{t}}\leq\frac{8}{b_{0}}\sigma^{2}\ln\left(\frac{2}{\delta}\right)-b_{0}+4\sigma\sqrt{T+\log\frac{2}{\delta}}.
\]
\end{lem}
\begin{proof}
First, Lemma \ref{lem:upperbound-on-error} gives that with probability
at least $1-\delta$, for all $t\le T$

\begin{align*}
\sum_{i=1}^{t}\norm{\xi_{i}}^{2} & \leq\sum_{i=1}^{t}\norm{\widehat{\nabla}f(x_{i})}^{2}+4\sigma^{2}\ln\left(\frac{1}{\delta}\right)\\
 & =b_{t}^{2}-b_{0}^{2}+4\sigma^{2}\ln\left(\frac{1}{\delta}\right)\\
\implies b_{t}^{2} & \geq\sum_{i=1}^{t}\norm{\xi_{i}}^{2}-\underbrace{\left[4\sigma^{2}\ln\left(\frac{1}{\delta}\right)-b_{0}^{2}\right]}_{=:C}.
\end{align*}
This means that 
\[
b_{t}\geq\max\left\{ b_{0},\sqrt{\left(\sum_{i=1}^{t}\norm{\xi_{i}}^{2}-C\right)^{+}}\right\} .
\]
Let $\tau=\max\left(\left\{ 0\right\} \cup\left\{ t\in\mathbb{N}_{\leq T}\mid\sum_{i=1}^{t}\norm{\xi_{i}}^{2}\leq2C\right\} \right)$.
Then 
\begin{align*}
\sum_{t=1}^{T}\frac{1}{b_{t}}\norm{\xi_{t}}^{2} & \leq\sum_{t=1}^{\tau}\frac{1}{b_{t}}\norm{\xi_{t}}^{2}+\sum_{t=\tau+1}^{T}\frac{1}{b_{t}}\norm{\xi_{t}}^{2}\\
 & \leq\frac{1}{b_{0}}\sum_{t=1}^{\tau}\norm{\xi_{t}}^{2}+\sum_{t=\tau+1}^{T}\frac{\norm{\xi_{t}}^{2}}{\max\left\{ b_{0},\sqrt{\sum_{i=1}^{t}\norm{\xi_{i}}^{2}-C}\right\} }\\
 & \leq\frac{2C}{b_{0}}+\sum_{t=\tau+1}^{T}\frac{\norm{\xi_{t}}^{2}}{\sqrt{\sum_{i=1}^{t}\norm{\xi_{i}}^{2}-C}}\\
 & \leq\frac{2C}{b_{0}}+\sum_{t=\tau+1}^{T}\frac{\norm{\xi_{t}}^{2}}{\sqrt{\frac{1}{2}\sum_{i=1}^{t}\norm{\xi_{i}}^{2}}}\quad\left(\mbox{since}\tag{\ensuremath{\sum_{i=1}^{t}\norm{\xi_{i}}^{2}>2C} for \ensuremath{t>\tau}}\right)\\
 & \leq\frac{2C}{b_{0}}+2\sum_{t=1}^{T}\frac{\norm{\xi_{t}}^{2}}{\sqrt{\sum_{i=1}^{t}\norm{\xi_{i}}^{2}}}\\
 & \leq\frac{2C}{b_{0}}+4\sqrt{\sum_{t=1}^{T}\norm{\xi_{t}}^{2}}.
\end{align*}
Hence, with probability at least $1-\delta$,
\[
\sum_{t=1}^{T}\frac{\norm{\xi_{t}}^{2}}{b_{t}}\leq\frac{8}{b_{0}}\sigma^{2}\ln\left(\frac{1}{\delta}\right)-b_{0}+4\sqrt{\sum_{t=1}^{T}\norm{\xi_{t}}^{2}}.
\]
Combining with Lemma \ref{lem:sum-error-naive-bound}, we get with
probability at least $1-2\delta$
\[
\sum_{t=1}^{T}\frac{\norm{\xi_{t}}^{2}}{b_{t}}\leq\frac{8}{b_{0}}\sigma^{2}\ln\left(\frac{1}{\delta}\right)-b_{0}+4\sigma\sqrt{T+\log\frac{1}{\delta}}.
\]
\end{proof}

\begin{lem}
\label{lem:bound-on-eps-sq-by-bt-sq}For AdaGrad-Norm stepsize $b_{t}$
and $\sigma$-subgaussian noise $\norm{\xi_{t}}$, with probability
at least $1-\delta$,

\begin{align*}
\sum_{t=1}^{T}\frac{\norm{\xi_{t}}^{2}}{b_{t}^{2}} & \leq\frac{4\sigma^{2}}{b_{0}^{2}}\log\left(\frac{2}{\delta}\right)+2\log\left(1+\sigma^{2}T+\sigma^{2}\log\frac{2}{\delta}\right)\\
 & =O\left(\sigma^{2}\log\left(\frac{1}{\delta}\right)+\log\left(1+\sigma^{2}T+\sigma^{2}\log\frac{1}{\delta}\right)\right).
\end{align*}
\end{lem}
\begin{proof}
Lemma \ref{lem:upperbound-on-error} gives that with probability at
least $1-\delta$

\begin{align*}
\sum_{i=1}^{t}\norm{\xi_{i}}^{2} & \leq\sum_{i=1}^{t}\norm{\widehat{\nabla}f(x_{i})}^{2}+4\sigma^{2}\ln\left(\frac{1}{\delta}\right)\\
 & =b_{t}^{2}-b_{0}^{2}+4\sigma^{2}\ln\left(\frac{1}{\delta}\right)\\
\implies b_{t}^{2} & \geq\sum_{i=1}^{t}\norm{\xi_{i}}^{2}-\underbrace{\left[4\sigma^{2}\ln\left(\frac{1}{\delta}\right)-b_{0}^{2}\right]}_{=:C}.
\end{align*}
Let $\tau=\max\left(\left\{ 0\right\} \cup\left\{ t\in\mathbb{N}_{\leq T}\mid\sum_{i=1}^{t}\norm{\xi_{i}}^{2}\leq2C\right\} \right)$.
We have
\begin{align*}
\sum_{t=1}^{T}\frac{\norm{\xi_{t}}^{2}}{b_{t}^{2}} & \leq\sum_{t=1}^{\tau}\frac{\norm{\xi_{t}}^{2}}{b_{t}^{2}}+\sum_{t=\tau+1}^{T}\frac{\norm{\xi_{t}}^{2}}{b_{t}^{2}}\\
 & \leq\frac{1}{b_{0}^{2}}\sum_{t=1}^{\tau}\norm{\xi_{t}}^{2}+\sum_{t=\tau+1}^{T}\frac{\norm{\xi_{t}}^{2}}{\sum_{i=1}^{t}\norm{\xi_{i}}^{2}-C}\tag{\ensuremath{\quad}\ensuremath{\left(\mbox{since }\ensuremath{\sum_{i=1}^{t}\norm{\xi_{i}}^{2}>2C}\mbox{ for }\ensuremath{t>\tau}\right)}}\\
 & \leq\frac{2C}{b_{0}^{2}}+2\sum_{t=\tau+1}^{T}\frac{\norm{\xi_{t}}^{2}}{\sum_{i=1}^{t}\norm{\xi_{i}}^{2}}\\
 & \leq\frac{2C}{b_{0}^{2}}+2\sum_{t=1}^{T}\frac{\norm{\xi_{t}}^{2}}{\sum_{i=1}^{t}\norm{\xi_{i}}^{2}}\\
 & \leq\frac{2C}{b_{0}^{2}}+2+2\log\left(1+\sum_{t=1}^{T}\norm{\xi_{t}}^{2}\right)\\
 & =\frac{4\sigma^{2}}{b_{0}^{2}}\log\left(\frac{1}{\delta}\right)+2\log\left(1+\sum_{t=1}^{T}\norm{\xi_{t}}^{2}\right).
\end{align*}
Then, we can combine this with Lemma \ref{lem:sum-error-naive-bound}
to get that with probability at least $1-2\delta$
\[
\sum_{t=1}^{T}\frac{\norm{\xi_{t}}^{2}}{b_{t}^{2}}\leq\frac{4\sigma^{2}}{b_{0}^{2}}\log\left(\frac{1}{\delta}\right)+2\log\left(1+\sigma^{2}T+\sigma^{2}\log\frac{1}{\delta}\right).
\]
Replacing $\delta$ with $\delta/2$ yields the result.
\end{proof}

\section{AdaGrad (coordinate) analysis \label{sec:AdaGrad-(coordinate)-analysis}}

\global\long\def\hn{\widehat{\nabla}}%

In this section, we show that our same technique can be generalized
to the standard (per-coordinate) version of AdaGrad. The analysis
is analogous to our AdaGrad-norm analysis but with the coordinates
taken into account. 

\subsection{Preliminaries and notations}

Let $d\in\mathbb{N}$ be the dimension of the problem. We let $v_{i}$
denote the $i$-th coordinate of a vector $v\in\R^{d}$. If a vector
like $x_{t}$ is already indexed as part of a sequence of vectors
(where $x_{t}$ denotes the $t$-th update) then we use $x_{t,i}$
to denote $x_{t}$'s $i$-th coordinate. For gradients, we let $\nabla_{i}f(x):=\frac{\partial f}{\partial x_{i}}$
denote the partial derivative wrt the $i$-th coordinate. Similarly,
for stochastic gradients $\hn f(x)$, we let $\hn_{i}f(x)$ denotes
its $i$-th coordinate. For simplicity, in our analysis, we will use
$\hn_{t,i}:=\widehat{\nabla}_{i}f(x_{t})$ and $\nabla_{t,i}:=\nabla_{i}f(x_{t})$
to denote the $i$-th coordinate of the stochastic gradients and gradients
at iterate $t$, respectively. If $a,b\in\R^{d}$, then $ab$ and
$a/b$ denotes coordinate-wise multiplication and division, respectively
i.e. $(ab)_{i}=a_{i}b_{i}$ and $(a/b)_{i}=a_{i}/b_{i}$.

If we denote the noise as $\xi_{t}:=\widehat{\nabla}f(x_{t})-\nabla f(x_{t})$
and $\xi_{t,i}$ as the $i$-th coordinate of $\xi_{t}$, then we
assume the noise is per-coordinate sub-gaussian i.e. there exists
$\sigma_{i}>0$ for $i\in[d]$ such that $\xi_{t}$ satisfies
\[
\E\left[\exp\left(\lambda^{2}\xi_{t,i}^{2}\right)\right]\leq\exp\left(\lambda^{2}\sigma_{i}^{2}\right),\forall\left|\lambda\right|\leq\frac{1}{\sigma_{i}},\forall i\in\left[d\right].
\]

Note that $\left\Vert \xi_{t}\right\Vert $ being $\sigma$-subgaussian
implies that each $\xi_{t,i}$ is also $\sigma$-subgaussian, thus
the assumption above is more general.

\subsection{Analysis}

Similarly to our Adagrad-norm analysis in Section \ref{sec:AdaGrad-Norm-second-proof},
we define a proxy step size $a_{t}$ that replaces the stochastic
gradient at time $t$ with the true gradient: $a_{t,i}^{2}:=b_{t-1,i}^{2}+\nabla_{t,i}^{2}$,
for $i\in[d]$. First, we present an analogous starting point to Lemma
\ref{lem:adagrad-basics} in the Lemma below:
\begin{lem}
\label{lem:adagrad-coord-basics}For $t\ge1$, let $\xi_{t,i}=\hn_{t,i}-\nabla_{t,i}$,
$a_{t,i}^{2}:=b_{t-1,i}^{2}+\nabla_{t,i}^{2}$ and $M_{t,i}=\max_{j\le t}\left|\xi_{j,i}\right|$,
then we have
\[
\sum_{t=1}^{T}\sum_{i=1}^{d}\frac{\nabla_{t,i}^{2}}{b_{t,i}}\leq\frac{\Delta_{1}}{\eta}-\sum_{t=1}^{T}\sum_{i=1}^{d}\frac{\nabla_{t,i}\xi_{t,i}}{a_{t,i}}+\sum_{t=1}^{T}\sum_{i=1}^{d}\left|\xi_{t,i}\right|\left[\frac{\nabla_{t,i}^{2}}{2a_{t,i}^{2}}+\frac{\xi_{t,i}^{2}}{2b_{t,i}^{2}}\right]+\frac{\eta L}{2}\sum_{t=1}^{T}\sum_{i=1}^{d}\frac{\hn_{t,i}^{2}}{b_{t,i}^{2}}.
\]
\end{lem}
\begin{proof}
We start with the smoothness of $f$

\begin{align*}
\Delta_{t+1}-\Delta_{t} & \leq\left\langle \nabla f(x_{t}),x_{t+1}-x_{t}\right\rangle +\frac{L}{2}\norm{x_{t+1}-x_{t}}^{2}\\
 & =-\eta\sum_{i=1}^{d}\frac{\nabla_{t,i}\hn_{t,i}}{b_{t,i}}+\frac{\eta^{2}L}{2}\sum_{i=1}^{d}\frac{\hn_{t,i}^{2}}{b_{t,i}^{2}}\\
 & =-\eta\sum_{i=1}^{d}\frac{\nabla_{t,i}^{2}}{b_{t,i}}-\eta\sum_{i=1}^{d}\frac{\nabla_{t,i}\xi_{t,i}}{b_{t,i}}+\frac{\eta^{2}L}{2}\sum_{i=1}^{d}\frac{\hn_{t,i}^{2}}{b_{t,i}^{2}}\\
 & =-\eta\sum_{i=1}^{d}\frac{\nabla_{t,i}^{2}}{b_{t,i}}-\eta\sum_{i=1}^{d}\frac{\nabla_{t,i}\xi_{t,i}}{a_{t,i}}+\eta\sum_{i=1}^{d}\left(\frac{1}{a_{t,i}}-\frac{1}{b_{t,i}}\right)\nabla_{t,i}\xi_{t,i}+\frac{\eta^{2}L}{2}\sum_{i=1}^{d}\frac{\hn_{t,i}^{2}}{b_{t,i}^{2}}.
\end{align*}
Similarly to Lemma \ref{lem:adagrad-proxy-difference}, we have:
\[
\left|\frac{1}{a_{t,i}}-\frac{1}{b_{t,i}}\right|\leq\frac{\left|\xi_{t,i}\right|}{a_{t,i}b_{t,i}}.
\]
Then
\begin{align*}
\Delta_{t+1}-\Delta_{t} & \leq-\eta\sum_{i=1}^{d}\frac{\nabla_{t,i}^{2}}{b_{t,i}}-\eta\sum_{i=1}^{d}\frac{\nabla_{t,i}\xi_{t,i}}{a_{t,i}}+\eta\sum_{i=1}^{d}\left(\frac{1}{a_{t,i}}-\frac{1}{b_{t,i}}\right)\nabla_{t,i}\xi_{t,i}+\frac{\eta^{2}L}{2}\sum_{i=1}^{d}\frac{\hn_{t,i}^{2}}{b_{t,i}^{2}}\\
 & \leq-\eta\sum_{i=1}^{d}\frac{\nabla_{t,i}^{2}}{b_{t,i}}-\eta\sum_{i=1}^{d}\frac{\nabla_{t,i}\xi_{t,i}}{a_{t,i}}+\eta\sum_{i=1}^{d}\frac{\left|\xi_{t,i}\right|}{a_{t,i}b_{t,i}}\left|\nabla_{t,i}\xi_{t,i}\right|+\frac{\eta^{2}L}{2}\sum_{i=1}^{d}\frac{\hn_{t,i}^{2}}{b_{t,i}^{2}}\\
 & \leq-\eta\sum_{i=1}^{d}\frac{\nabla_{t,i}^{2}}{b_{t,i}}-\eta\sum_{i=1}^{d}\frac{\nabla_{t,i}\xi_{t,i}}{a_{t,i}}+\eta\sum_{i=1}^{d}\left|\xi_{t,i}\right|\left[\frac{\nabla_{t,i}^{2}}{2a_{t,i}^{2}}+\frac{\xi_{t,i}^{2}}{2b_{t,i}^{2}}\right]+\frac{\eta^{2}L}{2}\sum_{i=1}^{d}\frac{\hn_{t,i}^{2}}{b_{t,i}^{2}}.
\end{align*}
Rearranging and summing give us the Lemma. 
\end{proof}

Next, we present an analogous per-coordinate result to Lemma \ref{sec:AdaGrad-Norm-second-proof}. 
\begin{lem}
\label{lem:adagrad-per-coord-martingale-dif}With $M_{T,i}=\max_{t\le T}\left|\xi_{t,i}\right|,\sigma_{\max}=\max_{i\in[d]}\sigma_{i}$,
and for any $w>0$, we have with probability at least $1-2d\delta$
\[
\frac{1}{\left\Vert b_{T}\right\Vert _{1}}\sum_{t=1}^{T}\left\Vert \nabla f(x_{t})\right\Vert _{1}^{2}\leq\frac{\Delta_{1}}{\eta}+d\sigma_{\max}\sqrt{\log\frac{1}{\delta}}+\left(8\left\Vert \sigma\right\Vert _{1}\sqrt{\log\frac{T}{\delta}}+d\eta L\right)\log\left(\frac{\left\Vert b_{T}\right\Vert _{1}}{\min b_{0,i}}\right)+\sum_{i=1}^{d}6\sigma_{i}\sqrt{\log\frac{T}{\delta}}\sum_{t=1}^{T}\frac{\xi_{t,i}^{2}}{b_{t,i}^{2}}.
\]
\end{lem}
\begin{proof}
We first take care of the term $\sum_{t=1}^{T}\sum_{i=1}^{d}\frac{\hn_{t,i}^{2}}{b_{t,i}^{2}}$
from Lemma \ref{lem:adagrad-coord-basics}:
\[
\sum_{t=1}^{T}\sum_{i=1}^{d}\frac{\hn_{t,i}^{2}}{b_{t,i}^{2}}=\sum_{i=1}^{d}\sum_{t=1}^{T}\frac{\hn_{t,i}^{2}}{b_{t,i}^{2}}=\sum_{i=1}^{d}\sum_{t=1}^{T}\frac{b_{t,i}^{2}-b_{t-1,i}^{2}}{b_{t,i}^{2}}\leq\sum_{i=1}^{d}2\log\frac{b_{T,i}}{b_{0,i}}.
\]
Next, we deal with $-\sum_{t=1}^{T}\sum_{i=1}^{d}\frac{\nabla_{t,i}\xi_{t,i}}{a_{t,i}}$
via our martingale argument. For any $w>0$, we have for each $i\in[d]$:
\begin{align*}
\E\left[\exp\left(-w\frac{\nabla_{t,i}\xi_{t,i}}{a_{t,i}}-2w^{2}\frac{\sigma_{i}^{2}\nabla_{t,i}^{2}}{a_{t,i}^{2}}\right)\mid\F_{t}\right] & =\exp\left(-2w^{2}\frac{\sigma_{i}^{2}\nabla_{t,i}^{2}}{a_{t,i}^{2}}\right)\E\left[\exp\left(-w\frac{\nabla_{t,i}\xi_{t,i}}{a_{t,i}}\right)\mid\F_{t}\right]\\
 & \leq1.
\end{align*}
Then a simple inductive argument gives with probability at least $1-\delta$:
\[
-w\sum_{t=1}^{T}\frac{\nabla_{t,i}\xi_{t,i}}{a_{t,i}}\leq2w^{2}\sum_{t=1}^{T}\frac{\sigma_{i}^{2}\nabla_{t,i}^{2}}{a_{t,i}^{2}}+\log\frac{1}{\delta}.
\]
By a union bound across all coordinate $d$, we have w.p. at least
$1-d\delta$:
\begin{equation}
-\sum_{t=1}^{T}\sum_{i=1}^{d}\frac{\nabla_{t,i}\xi_{t,i}}{a_{t,i}}\leq\sum_{t=1}^{T}\sum_{i=1}^{d}\frac{w\sigma_{i}^{2}\nabla_{t,i}^{2}}{a_{t,i}^{2}}+\frac{d}{w}\log\frac{1}{\delta}.\label{eq:bound-inner-prod-adagrad-coord}
\end{equation}
Let's call the event that (\ref{eq:bound-inner-prod-adagrad-coord})
happens $E_{1}$. Now, we deal with $\sum_{t=1}^{T}\sum_{i=1}^{d}\frac{\nabla_{t,i}^{2}}{a_{t,i}^{2}}$.
Note that
\[
\frac{\nabla_{t,i}^{2}}{a_{t,i}^{2}}=\frac{\nabla_{t,i}^{2}}{b_{t-1,i}^{2}+\nabla_{t,i}^{2}+\sigma_{i}^{2}}\leq\frac{2\hn_{t,i}^{2}+2\xi_{t,i}^{2}}{b_{t-1,i}^{2}+2\hn_{t,i}^{2}+2\xi_{t,i}^{2}+\sigma_{i}^{2}}\leq2\frac{\hn_{t,i}^{2}}{b_{t,i}^{2}}+2\frac{\xi_{t,i}^{2}}{b_{t,i}^{2}}.
\]
Under the even $E_{1}$ and the above result, we can bound Lemma \ref{lem:adagrad-coord-basics}
with probability at least $1-d\delta$:
\begin{align*}
\sum_{t=1}^{T}\sum_{i=1}^{d}\frac{\nabla_{t,i}^{2}}{b_{t,i}} & \leq\frac{\Delta_{1}}{\eta}+w\sum_{t=1}^{T}\sum_{i=1}^{d}\frac{\sigma_{i}^{2}\nabla_{t,i}^{2}}{a_{t,i}^{2}}+\frac{d}{w}\log\frac{1}{\delta}+\sum_{t=1}^{T}\sum_{i=1}^{d}\frac{M_{T,i}}{2}\left[\frac{\nabla_{t,i}^{2}}{a_{t,i}^{2}}+\frac{\xi_{t,i}^{2}}{b_{t,i}^{2}}\right]+\frac{\eta L}{2}\sum_{i=1}^{d}2\log\frac{b_{T,i}}{b_{0,i}}\\
 & =\frac{\Delta_{1}}{\eta}+\frac{d}{w}\log\frac{1}{\delta}+\sum_{t=1}^{T}\sum_{i=1}^{d}\left(2w\sigma_{i}^{2}+M_{T,i}\right)\frac{\hn_{t,i}^{2}}{b_{t,i}^{2}}+2\sum_{t=1}^{T}\sum_{i=1}^{d}\left(M_{T,i}+w\sigma_{i}^{2}\right)\frac{\xi_{t,i}^{2}}{b_{t,i}^{2}}+\eta L\sum_{i=1}^{d}\log\frac{b_{T,i}}{b_{0,i}}\\
 & =\frac{\Delta_{1}}{\eta}+\frac{d}{w}\log\frac{1}{\delta}+\sum_{i=1}^{d}\left(4w\sigma_{i}^{2}+2M_{T,i}+\eta L\right)\log\frac{b_{T,i}}{b_{0,i}}+2\sum_{i=1}^{d}\left(M_{T,i}+w\sigma_{i}^{2}\right)\sum_{t=1}^{T}\frac{\xi_{t,i}^{2}}{b_{t,i}^{2}}.
\end{align*}
Note that
\begin{align*}
\left(\sum_{i=1}^{d}\frac{\nabla_{t,i}^{2}}{b_{t,i}}\right)\left(\sum_{i=1}^{d}b_{t,i}\right) & \geq\left(\sum_{i=1}^{d}\left\Vert \nabla_{t,i}\right\Vert \right)^{2}=\left\Vert \nabla f(x_{t})\right\Vert _{1}^{2}\\
\Rightarrow\left(\sum_{i=1}^{d}\frac{\nabla_{t,i}^{2}}{b_{t,i}}\right) & \geq\frac{\left\Vert \nabla f(x_{t})\right\Vert _{1}^{2}}{\left\Vert b_{t}\right\Vert _{1}}\geq\frac{\left\Vert \nabla f(x_{t})\right\Vert _{1}^{2}}{\left\Vert b_{T}\right\Vert _{1}}.
\end{align*}
Hence, we have
\[
\frac{1}{\left\Vert b_{T}\right\Vert _{1}}\sum_{t=1}^{T}\left\Vert \nabla f(x_{t})\right\Vert _{1}^{2}\leq\sum_{t=1}^{T}\sum_{i=1}^{d}\frac{\nabla_{t,i}^{2}}{b_{t,i}}.
\]
Since it is known that with probability at least $1-\delta$, $\max_{t\in\left[T\right]}\left|\xi_{t,i}\right|\leq\sigma_{i}\sqrt{1+\log\frac{T}{\delta}}$
for each $i\in[d]$ \citet{li2020high,liu2022convergence}, a union
bound over all $d$ gives us that w.p.$\geq1-d\delta$
\begin{equation}
M_{T,i}\leq2\sigma_{i}\sqrt{\log\frac{T}{\delta}},\forall i\in\left[d\right].\label{eq:uniform-bound-over-xit-coord}
\end{equation}
Condition under this event and choosing $\frac{1}{w}=\frac{\sigma_{\max}}{\sqrt{\log\frac{1}{\delta}}}$
gives us with probability at least $1-2d\delta$
\begin{align*}
\frac{1}{\left\Vert b_{T}\right\Vert _{1}}\sum_{t=1}^{T}\left\Vert \nabla f(x_{t})\right\Vert _{1}^{2} & \leq\frac{\Delta_{1}}{\eta}+d\sigma_{\max}\sqrt{\log\frac{1}{\delta}}+\sum_{i=1}^{d}\left(\frac{4\sigma_{i}^{2}}{\sigma_{\max}}\sqrt{\log\frac{1}{\delta}}+4\sigma_{i}\sqrt{\log\frac{T}{\delta}}+\eta L\right)\log\frac{b_{T,i}}{b_{0,i}}\\
 & \quad+2\sum_{i=1}^{d}\left(2\sigma_{i}\sqrt{\log\frac{T}{\delta}}+\frac{\sigma_{i}^{2}}{\sigma_{\max}}\sqrt{\log\frac{1}{\delta}}\right)\sum_{t=1}^{T}\frac{\xi_{t,i}^{2}}{b_{t,i}^{2}}\\
 & \leq\frac{\Delta_{1}}{\eta}+d\sigma_{\max}\sqrt{\log\frac{1}{\delta}}+\left(8\left\Vert \sigma\right\Vert _{1}\sqrt{\log\frac{T}{\delta}}+d\eta L\right)\log\left(\frac{\left\Vert b_{T}\right\Vert _{1}}{\min b_{0,i}}\right)+\sum_{i=1}^{d}6\sigma_{i}\sqrt{\log\frac{T}{\delta}}\sum_{t=1}^{T}\frac{\xi_{t,i}^{2}}{b_{t,i}^{2}}.
\end{align*}
\end{proof}

Finally, it remains to bound $\norm{b_{T}}_{1}$ and $\sum_{t=1}^{T}\frac{\xi_{t,i}^{2}}{b_{t,i}^{2}}$.
For this we use Lemma \ref{lem: xit-sq-bt-sq-coord} to show the following
bound on $\norm{b_{T}}_{1}$:
\begin{lem}
With probability at least $1-2d\delta$
\begin{align*}
\left\Vert b_{T}\right\Vert _{1} & \leq2\left\Vert b_{0}\right\Vert _{1}+\frac{4\Delta_{1}}{\eta}+\log\left(\frac{2}{\delta}\right)\sum_{i=1}^{d}\frac{8\sigma_{i}^{2}}{b_{0,i}}+4\sum_{i=1}^{d}\sqrt{\sigma_{i}^{2}T+\sigma_{i}^{2}\log\frac{2}{\delta}}+4\eta^{2}L\sum_{i=1}^{d}\log\frac{4\eta^{2}L}{b_{0,i}}\\
 & =O\left(\left\Vert \sigma\right\Vert _{1}\sqrt{T}+\left\Vert b_{0}\right\Vert _{1}+\frac{\Delta_{1}}{\eta}+\left\Vert \frac{\sigma^{2}}{b_{0}}\right\Vert _{1}\log\left(\frac{1}{\delta}\right)+\left\Vert \sigma\right\Vert _{1}\sqrt{\log\frac{1}{\delta}}+\eta^{2}L\sum_{i=1}^{d}\log\frac{\eta^{2}L}{b_{0,i}}\right).
\end{align*}
\end{lem}
\begin{proof}
We start via the smoothness of $f$
\begin{align*}
f(x_{t+1})-f(x_{t}) & \leq\left\langle \nabla f(x_{t}),x_{t+1}-x_{t}\right\rangle +\frac{L}{2}\norm{x_{t+1}-x_{t}}^{2}\\
 & =-\frac{\eta}{b_{t}}\left\langle \nabla f(x_{t}),\hn f(x_{t})\right\rangle +\frac{\eta^{2}L}{2}\sum_{i=1}^{d}\frac{\hn_{t,i}^{2}}{b_{t,i}^{2}}\\
 & =-\eta\sum_{i=1}^{d}\frac{\hn_{t,i}^{2}}{b_{t,i}}+\eta\sum_{i=1}^{d}\frac{\xi_{t,i}\hn_{t,i}}{b_{t,i}}+\frac{\eta^{2}L}{2}\sum_{i=1}^{d}\frac{\hn_{t,i}^{2}}{b_{t,i}^{2}}\\
 & \leq-\frac{\eta}{2}\sum_{i=1}^{d}\frac{\hn_{t,i}^{2}}{b_{t,i}}+\frac{\eta}{2}\sum_{i=1}^{d}\frac{\xi_{t,i}^{2}}{b_{t,i}}+\frac{\eta^{2}L}{2}\sum_{i=1}^{d}\frac{\hn_{t,i}^{2}}{b_{t,i}^{2}}.
\end{align*}
Summing up over $t$ we obtain
\begin{align*}
\sum_{t=1}^{T}\sum_{i=1}^{d}\frac{\hn_{t,i}^{2}}{b_{t,i}} & \leq\frac{2\Delta_{1}}{\eta}+\sum_{t=1}^{T}\sum_{i=1}^{d}\frac{\xi_{t,i}^{2}}{b_{t,i}}+\sum_{t=1}^{T}\sum_{i=1}^{d}\eta^{2}L\frac{\hn_{t,i}^{2}}{b_{t,i}^{2}}\\
 & \leq\frac{2\Delta_{1}}{\eta}+\sum_{t=1}^{T}\sum_{i=1}^{d}\frac{\xi_{t,i}^{2}}{b_{t,i}}+\sum_{i=1}^{d}2\eta^{2}L\log\frac{b_{T,i}}{b_{0,i}}.
\end{align*}
Note that the LHS of the above inequality is lower-bounded by $\left\Vert b_{T}\right\Vert _{1}-\left\Vert b_{0}\right\Vert _{1}$.
Thus
\begin{align*}
\left\Vert b_{T}\right\Vert _{1}-\left\Vert b_{0}\right\Vert _{1} & \leq\frac{2\Delta_{1}}{\eta}+\sum_{t=1}^{T}\sum_{i=1}^{d}\frac{\xi_{t,i}^{2}}{b_{t,i}}+\sum_{i=1}^{d}2\eta^{2}L\log\frac{b_{T,i}}{b_{0,i}}\\
 & \leq\frac{2\Delta_{1}}{\eta}+\sum_{t=1}^{T}\sum_{i=1}^{d}\frac{\xi_{t,i}^{2}}{b_{t,i}}+\sum_{i=1}^{d}2\eta^{2}L\left(\log\frac{b_{T,i}}{4\eta^{2}L}+\log\frac{4\eta^{2}L}{b_{0,i}}\right)\\
 & \leq\frac{2\Delta_{1}}{\eta}+\sum_{t=1}^{T}\sum_{i=1}^{d}\frac{\xi_{t,i}^{2}}{b_{t,i}}+\frac{\left\Vert b_{T}\right\Vert _{1}}{2}+\sum_{i=1}^{d}2\eta^{2}L\log\frac{4\eta^{2}L}{b_{0,i}};\\
\left\Vert b_{T}\right\Vert _{1} & \leq2\left\Vert b_{0}\right\Vert _{1}+\frac{4\Delta_{1}}{\eta}+2\sum_{t=1}^{T}\sum_{i=1}^{d}\frac{\xi_{t,i}^{2}}{b_{t,i}}+4\eta^{2}L\sum_{i=1}^{d}\log\frac{4\eta^{2}L}{b_{0,i}}.
\end{align*}
Note that by Lemma \ref{lem: xit-sq-bt-sq-coord}, with probability
at least $1-2d\delta$
\begin{align*}
\sum_{t=1}^{T}\sum_{i=1}^{d}\frac{\xi_{t,i}^{2}}{b_{t,i}} & \le\sum_{i=1}^{d}\frac{8\sigma_{i}^{2}\log\frac{1}{\delta}}{b_{0,i}}+4\sqrt{\sigma_{i}^{2}T+\sigma_{i}^{2}\log\frac{1}{\delta}}\\
 & =O\left(\left\Vert \frac{\sigma^{2}}{b_{0}}\right\Vert _{1}\log\frac{1}{\delta}+\left\Vert \sigma\right\Vert _{1}\left(\sqrt{\log\frac{1}{\delta}}+\sqrt{T}\right)\right).
\end{align*}
 Hence, under this event, we have that with probability at least $1-2d\delta$
\begin{align*}
\left\Vert b_{T}\right\Vert _{1} & \leq2\left\Vert b_{0}\right\Vert _{1}+\frac{4\Delta_{1}}{\eta}+O\left(\left\Vert \frac{\sigma^{2}}{b_{0}}\right\Vert _{1}\log\frac{1}{\delta}+\left\Vert \sigma\right\Vert _{1}\left(\sqrt{\log\frac{1}{\delta}}+\sqrt{T}\right)\right)+4\eta^{2}L\sum_{i}\log\frac{4\eta^{2}L}{b_{0,i}}.
\end{align*}
\end{proof}

Now we are ready to prove Theorem \ref{thm:adagrad-coord-nonconvex}.\begin{proof}[Proof of Theorem \ref{thm:adagrad-coord-nonconvex}]
Combining Lemma \ref{lem:adagrad-coord-basics} with Lemma \ref{lem: xit-sq-bt-sq-coord},
we get with probability at least $1-4d\delta$
\begin{align*}
\frac{1}{\left\Vert b_{T}\right\Vert _{1}}\sum_{t=1}^{T}\left\Vert \nabla f(x_{t})\right\Vert _{1}^{2} & \leq\frac{\Delta_{1}}{\eta}+d\sigma_{\max}\sqrt{\log\frac{1}{\delta}}+\left(8\left\Vert \sigma\right\Vert _{1}\sqrt{\log\frac{T}{\delta}}+d\eta L\right)\log\left(\frac{\left\Vert b_{T}\right\Vert _{1}}{\min b_{0,i}}\right)+\sum_{i=1}^{d}6\sigma_{i}\sqrt{\log\frac{T}{\delta}}\sum_{t=1}^{T}\frac{\xi_{t,i}^{2}}{b_{t,i}^{2}}.\\
 & \leq\frac{\Delta_{1}}{\eta}+d\sigma_{\max}\sqrt{\log\frac{1}{\delta}}+\left(8\left\Vert \sigma\right\Vert _{1}\sqrt{\log\frac{T}{\delta}}+d\eta L\right)\log\left(\frac{\left\Vert b_{T}\right\Vert _{1}}{\min b_{0,i}}\right)\\
 & \quad+6\sqrt{\log\frac{T}{\delta}}\sum_{i=1}^{d}\sigma_{i}\left(\frac{8\sigma_{i}^{2}}{b_{0,i}^{2}}\log\frac{1}{\delta}+2\log\left(1+\frac{\sigma_{i}^{2}T+\sigma_{i}^{2}\log\frac{1}{\delta}}{2b_{0,i}^{2}}\right)\right).
\end{align*}
Rearranging, combining this with the bound for $\left\Vert b_{T}\right\Vert _{1}$,
and replacing $\delta$ with $\frac{\delta}{6d}$ yield the Theorem. 
\end{proof}

\subsection{Additional helper lemmas}
\begin{lem}
\label{lem:bound-on-xit-coord}We have w.p. $\geq1-d\delta$
\[
\sum_{t=1}^{\tau}\xi_{t,i}^{2}\leq\sum_{t=1}^{\tau}\hn_{t,i}^{2}+4\sigma_{i}^{2}\log\frac{1}{\delta},\forall\tau\in\left[T\right],\forall i\in\left[d\right].
\]
\end{lem}
\begin{proof}
We apply Lemma \ref{lem:upperbound-on-error} to each coordinate individually
and then union bound over all the dimensions to get the result. 
\end{proof}

\begin{lem}
\label{lem:bound-on-xit-coord-naive}We have w.p.$\geq1-d\delta$
\[
\sum_{t=1}^{T}\xi_{t,i}^{2}\leq\sigma_{i}^{2}T+\sigma_{i}^{2}\log\frac{1}{\delta},\forall i\in\left[d\right].
\]
\end{lem}
\begin{proof}
We apply Lemma \ref{lem:sum-error-naive-bound} to each coordinate
individually and then union bound over all the dimensions to get the
result. 
\end{proof}

We can show a bound on $\sum_{t=1}^{T}\frac{\xi_{t,i}^{2}}{b_{t,i}}$
and $\sum_{t=1}^{T}\frac{\xi_{t,i}^{2}}{b_{t,i}^{2}}$ for each $i\in[d]$:
\begin{lem}
\label{lem: xit-sq-bt-sq-coord}We have
\begin{enumerate}
\item With probability at least $1-2d\delta$, we have for all $i\in[d]$
\begin{align*}
\sum_{t=1}^{T}\frac{\xi_{t,i}^{2}}{b_{t,i}} & \le\frac{8\sigma_{i}^{2}\log\frac{1}{\delta}}{b_{0,i}}+4\sqrt{\sigma_{i}^{2}T+\sigma_{i}^{2}\log\frac{1}{\delta}}.
\end{align*}
\item With probability at least $1-2d\delta$, we have for all $i\in[d]$
\begin{align*}
\sum_{t=1}^{T}\frac{\xi_{t,i}^{2}}{b_{t,i}^{2}} & \leq\frac{8\sigma_{i}^{2}}{b_{0,i}^{2}}\log\frac{1}{\delta}+2\log\left(1+\frac{\sigma_{i}^{2}T+\sigma_{i}^{2}\log\frac{1}{\delta}}{2b_{0,i}^{2}}\right).
\end{align*}
\end{enumerate}
\end{lem}
\begin{proof}
For (1), we have with probability at least $1-2d\delta$ 
\begin{align*}
\sum_{t=1}^{T}\frac{\xi_{t,i}^{2}}{b_{t,i}} & =\sum_{t=1}^{T}\frac{\xi_{t,i}^{2}}{\sqrt{b_{0,i}^{2}+\sum_{s=1}^{t}\hn_{s,i}^{2}}}\\
 & \overset{(1)}{\leq}\sum_{t=1}^{T}\frac{\xi_{t,i}^{2}}{\sqrt{b_{0,i}^{2}+\left(\sum_{s=1}^{t}\xi_{s,i}^{2}-4\sigma_{i}^{2}\log\frac{1}{\delta}\right)^{+}}}\\
 & \leq\frac{8\sigma_{i}^{2}\log\frac{1}{\delta}}{b_{0,i}}+2\sqrt{2}\sqrt{\sum_{s=1}^{T}\xi_{s,i}^{2}}\\
 & \overset{(2)}{\leq}\frac{8\sigma_{i}^{2}\log\frac{1}{\delta}}{b_{0,i}}+4\sqrt{\sigma_{i}^{2}T+\sigma_{i}^{2}\log\frac{1}{\delta}},
\end{align*}
where (1) is due to Lemma \ref{lem:bound-on-xit-coord} and (2) is
due to Lemma \ref{lem:bound-on-xit-coord-naive}.

For (2), we have with probability at least $1-2d\delta$ 
\begin{align*}
\sum_{t}\frac{\xi_{t,i}^{2}}{b_{t,i}^{2}} & =\sum_{t}\frac{\xi_{t,i}^{2}}{b_{0,i}^{2}+\sum_{s=1}^{t}\hn_{s,i}^{2}}\\
 & \overset{(1)}{\leq}\sum_{t}\frac{\xi_{t,i}^{2}}{b_{0,i}^{2}+\left(\sum_{s=1}^{t}\xi_{s,i}^{2}-4\sigma_{i}^{2}\log\frac{1}{\delta}\right)^{+}}\\
 & \leq\frac{8\sigma_{i}^{2}\log\frac{1}{\delta}}{b_{0,i}^{2}}+2\log\left(1+\frac{\sum_{t=1}^{T}\xi_{t,i}^{2}}{2b_{0,i}^{2}}\right)\\
 & \overset{(2)}{\leq}\frac{8\sigma_{i}^{2}\log\frac{1}{\delta}}{b_{0,i}^{2}}+2\log\left(1+\frac{\sigma_{i}^{2}T+\sigma_{i}^{2}\log\frac{1}{\delta}}{2b_{0,i}^{2}}\right),
\end{align*}
where (1) is due to Lemma \ref{lem:bound-on-xit-coord} and (2) is
due to Lemma \ref{lem:bound-on-xit-coord-naive}.
\end{proof}

\end{document}